\documentclass[11pt,reqno]{amsart}
\usepackage{geometry}
\usepackage{amsmath,amssymb,amsthm}
\usepackage[numbers]{natbib}
\usepackage[dvips]{graphicx}
\usepackage{float}
\usepackage{epsfig,array,multirow}
\usepackage[usenames]{color}

\PassOptionsToPackage{normalem}{ulem}
\usepackage{ulem}
\usepackage{bbm}
\usepackage{mathrsfs}
\usepackage{tikz}
\usetikzlibrary{arrows, automata}
\usetikzlibrary{calc}
\usetikzlibrary{positioning}

\numberwithin{equation}{section}
\allowdisplaybreaks[4]

\theoremstyle{plain}
\newtheorem{theorem}{Theorem}[section]	
\newtheorem{lemma}{Lemma}[section]
\newtheorem{corollary}{Corollary}[section]
\newtheorem{proposition}{Proposition}[section]
\theoremstyle{definition}

\newtheorem{remark}{Remark}[section]

\DeclareMathOperator{\Hess}{Hess}

\newcommand{\eps}{\varepsilon}

\newcommand{\R}{\mathbb{R}}

\renewcommand{\qed}{\hfill{\tiny \ensuremath{\blacksquare} }}%

\renewcommand{\Pr}{{ \mathrm{I\!P}}}
\newcommand{\Ep}{{ \mathrm{I\!E}}}

\renewcommand{\Hess}{\nabla^2}

\def\be#1{\begin{equation*}#1\end{equation*}}
\def\ben#1{\begin{equation}#1\end{equation}}

\def\besn#1{\begin{equation}\begin{split}#1\end{split}\end{equation}}

\def\bm#1{\begin{multline*}#1\end{multline*}}
\def\bmn#1{\begin{multline}#1\end{multline}}
\def\ba#1{\begin{align*}#1\end{align*}}
\def\ban#1{\begin{align}#1\end{align}}

\begin{document}
\title[Nearly optimal high-dimensional CLT and Bootstrap]{Nearly optimal central limit theorem and bootstrap approximations in high dimensions}

\thanks{We are grateful to Xiaohong Chen, Xiao Fang, and Kengo Kato for helpful discussions.}

\author[Chernozhukov]{Victor Chernozhukov}
\author[Chetverikov]{Denis Chetverikov}
\author[Koike]{Yuta Koike}

\address[V. Chernozhukov]{
Department of Economics and Center for Statistics \& Data Science, MIT, 50 Memorial Drive, Cambridge, MA 02142, USA.
}
\email{vchern@mit.edu}

\address[D. Chetverikov]{
Department of Economics, UCLA, Bunche Hall, 8283, 315 Portola Plaza, Los Angeles, CA 90095, USA.
}
\email{chetverikov@econ.ucla.edu}

\address[Y. Koike]{
Mathematics and Informatics Center and Graduate School of Mathematical Sciences, The University of Tokyo, 3-8-1 Komaba, Meguro-ku, Tokyo 153-8914, Japan.
}
\email{kyuta@ms.u-tokyo.ac.jp}

\date{\today.}

\begin{abstract}
In this paper, we derive new, nearly optimal bounds for the Gaussian approximation to scaled averages  of $n$ independent high-dimensional centered random vectors $X_1,\dots,X_n$ over the class of rectangles in the case when the covariance matrix of the scaled average is non-degenerate. In the case of bounded $X_i$'s, the implied bound for the Kolmogorov distance between the distribution of the scaled average and the Gaussian vector takes the form $$C (B^2_n \log^3 d/n)^{1/2} \log n,$$ where $d$ is the dimension of the vectors and $B_n$ is a uniform envelope constant on components of $X_i$'s. This bound is sharp in terms of $d$ and $B_n$, and is nearly (up to $\log n$) sharp in terms of the sample size $n$.  In addition, we show that similar bounds hold for the multiplier and empirical bootstrap approximations. Moreover, we establish bounds that allow for unbounded $X_i$'s, formulated solely in terms of moments of $X_i$'s.  Finally, we demonstrate that the bounds can be further improved in some special smooth and zero-skewness cases.
\end{abstract}

\maketitle

\section{Introduction}
Let $X_1,\dots,X_n$ be a sequence of centered independent random vectors in $\mathbb R^d$. Denote
$$
W := \frac{1}{\sqrt n}\sum_{i=1}^n X_i
$$
and let $\mathcal R$ be the class of rectangles in $\mathbb R^d$, which we choose to be sets of the form $A=\prod_{j=1}^d(a_j,b_j]$ for some $-\infty\leq a_j\leq b_j\leq \infty$, $j=1,\dots,d$. In this paper, we are interested in deriving new bounds on
\begin{equation}\label{eq: varrho}
\varrho :=\sup_{A\in\mathcal R}|\Pr (W\in A)-\Pr (Z\in A)|,\quad Z\sim N(0,\Sigma_W),
\end{equation}
where $\Sigma_W := \Ep [WW^T]$. We are particularly interested in the high-dimensional case, where $d$ is potentially much larger than $n$.

The problem of bounding $\varrho$ has attracted considerable attention in the literature because the class of rectangles $\mathcal R$ strikes an interesting balance: it is large enough so that bounds on $\varrho$ are useful in mathematical statistics, and, at the same time, it is small enough so that, as $n\to\infty$, under minimal conditions, we have $\varrho=\varrho_n\to0$ even if $d=d_n\to\infty$ much faster than $n\to\infty$, as was originally shown in \cite{CCK13}, making bounds on $\varrho$ particularly useful in high-dimensional statistics and machine learning, e.g.~in multiple hypothesis testing with the family-wise error rate control and in selecting penalty parameters for regularized estimators of high-dimensional models; see \cite{BCCHKN18} for details on these and other examples.
%but, at the same time, is large enough to yield many applications in mathematical statistics, such as the choice of the regularization parameter of the least absolute shrinkage and selection operator (lasso) and multiple hypothesis testing with the family-wise error rate control.

Various bounds on $\varrho$ and on closely related quantities were derived in \cite{CCK13, CCK14, CCK17, ZW17, DZ17, ZC18, OST18, K19, CCKK19, FK20, L20, KR20, D20, DL20} but in our discussion, we only focus on the results that are particular relevant for comparisons with our results. In addition, for clarity of the introduction, we assume below that components of $X_i$'s are uniformly bounded by the envelope constant $B_n = B_n (d)$, i.e. $\|X_i\|_\infty:=\max_{1\leq j\leq d}|X_{ij}| \leq B_n$, for all $i=1,\dots,n$, even though all aforementioned papers, as well as ours, considered the case of unbounded $X_i$'s as well. It then follows from \cite{CCKK19} that
\begin{equation}\label{eq: previous bound}
\varrho \leq C\left( \frac{ B^2_n ( \log(d n))^5}{n}\right)^{1/4},
\end{equation}
where $C>0$ is a constant that is independent of $n$ and $d$. This bound is conjectured to be near-optimal when $\Sigma_W$ is unrestricted.

Next, \cite{FK20} demonstrated that if, in addition, we assume that all eigenvalues of  $\Sigma_W$ are bounded below from zero (strongly non-degenerate case, in their terminology), then the bound \eqref{eq: previous bound} can be substantially improved: they showed that
\begin{equation}\label{eq: previous bound 2}
\varrho \leq C\left( \frac{B^2_n(\log(d n))^4}{n}\right)^{1/3}
\end{equation}
in this case. Moreover, they established that this bound can be further improved to the near-sharp $n^{-1/2} \log n$ and the sharp $(\log d)^3$ dependencies but only for the case of jointly log-concave $X_i$'s. Their results exploit the implicit smoothing that occurs when $\Sigma_W$ is strongly non-degenerate. Further, \cite{L20} and \cite{KR20} demonstrated, again in the strongly non-degenerate case, that
\begin{equation}\label{eq: previous bound 3}
\varrho 
\leq C\left( \left( \frac{B_n^6 (\log d)^4 \log(d n)}{n} \right)^{1/2}+ 
\left( \frac{(\log d)^{7} \log(d n)}{n} \right)^{1/2} \right) \log n,
%\leq C'\left(\frac{B_n^3(\log d)^2\sqrt{\log(d n)}}{\sqrt n} + \frac{(\log d)^{7/2}\sqrt{\log(d n)}}{\sqrt n}\right)\log n,
\end{equation}
which nearly matches the dependence on $n$ in the classical Berry-Esseen bound for the one-dimensional ($d=1$) case, e.g.~Theorem 2.2.15 in \cite{T12}, but does not provide optimal dependence on $B_n = B_n(d)$ and $\log d$.

In this paper, our main result is to establish that in the strongly non-degenerate case, \begin{equation}\label{eq: our main bound introduction}
\varrho \leq C\left(\frac{B^2_n (\log d)^3}{n}\right)^{1/2}\log n,
\end{equation}
which we show to be optimal up to the $\log n$ factor, i.e. in general
\begin{equation}
\varrho \geq c\left( \frac{B^2_n(\log d)^3}{n}\right)^{1/2}.
\end{equation}
In addition, we extend this result to allow for unbounded $X_i$'s, which yields a bound depending solely on some moments of $X_i$'s. %We show that in certain smooth case the $\log n$ factor in the upper bound can be removed.  
A critical ingredient in our proofs is an explicit use of smoothing, combined with the previous implicit smoothing ideas, as we further comment on below.

Our result (\ref{eq: our main bound introduction}) strongly improves bounds \eqref{eq: previous bound 2} and \eqref{eq: previous bound 3}, and features the optimal dependence on the ambient dimension $d$, the optimal dependence on the envelope constant $B_n$,
and a nearly optimal dependence on the sample size $n$ (up to the $\log n$ factor).  This result improves  over  \eqref{eq: previous bound 3} by replacing
$(\log d)^8$ by the optimal $(\log d)^3$ and replacing $B_n^6$ by the optimal $B^2_n$. The first improvement is particularly important 
when $\log d$ is growing as some fractional power of the sample size $n$, in which case  our bound (\ref{eq: our main bound introduction})  has much better dependence on $n$.
The second improvement  is important 
when the envelope constant $B_n$ is increasing with the sample size $n$, in which case  our bound (\ref{eq: our main bound introduction})  also has much better dependence on $n$. This, for example, occurs in the many local means settings arising in nonparametric statistics, discussed in detail in Section 5, where the envelope constant $B_n = B_n(d)$ has dependency on the dimension of problem $d$ of the form $B_n(d) \propto \sqrt{d}$. In particular, the bound \eqref{eq: previous bound 3} would then require $d^3 \ll n$ for $\varrho \to 0$, whereas our bound would only require $d \ll n$. Therefore, the improvement
is critical for obtaining the sharp dependency on the dimension $d$ in general. Also, as discussed in Section 5, in the many local means setting, our bound tends to be either at least as sharp (up to log factors) or much sharper than the Gaussian approximation based on the Hungarian coupling.  

Moreover, we also consider bootstrap approximations, i.e. we derive bounds on
$$
\varrho^* := \sup_{A\in\mathcal R}|\Pr(Z\in A)-\Pr(W^*\in A\mid X_1,\dots,X_n)|,
$$
where $W^*$ denotes a bootstrap version of $W$. These approximations are important in mathematical statistics because they allow to estimate probabilities $\Pr(Z\in A)$, $A\in\mathcal R$, using random vectors $X_1,\dots,X_n$, which is useful when $\Sigma_W$ is unknown so that probabilities can not be calculated directly. For the multiplier and empirical bootstrap approximations, we derive bounds that are generally similar to those for the Gaussian approximation \eqref{eq: varrho}.

Finally, we show that the $\log n$ factor in \eqref{eq: our main bound introduction} can be removed if we assume that $X_i$'s have a Gaussian component, and we also show that if $X_i$'s have a Gaussian component and satisfy a zero-skewness condition, then
$$
\varrho \leq \frac{C(\log d)^2}{n}.
$$
This last bound substantially extends the results of \cite{DL20}, who showed that $\varrho\to 0$ if $(\log d)^2/n\to0$ in the case when $X_i$'s have independent components with vanishing odd moments.

Our results are built using the exchangeable pair approach coupled with the Slepian-Stein method and employ ideas of many authors, e.g.~\cite{B90, G91, AHT98, B03, ChMe08, ReRo09, CCK13, CCK17, FK20, L20, KR20}, but the key technical tool behind our results is a set of new smoothing inequalities, which we call mixed smoothing inequalities. Specifically, for any rectangle $A\in\mathcal R$, we first approximate the indicator of $A$  by a Lipschitz-smooth function and then approximate it further via convolution with a centered Gaussian distribution. Building on the results of \cite{B90, AHT98, FK20}, we then prove several bounds for sums of absolute values of partial derivatives of the resulting function, which play a crucial role in our derivations. 

This mixed smoothing turns out important for two reasons. First, smoothing via convolutions allows to obtain nearly optimal dependence on $n$, as demonstrated by \cite{G91} in the moderate-dimensional case and then by \cite{FK20, L20, KR20} in the high-dimensional case. Second, smoothing via Lipschitz-smooth functions allows to obtain optimal dependence on $d$, as follows from our results. Our approach here is inspired by \cite{B03}, who used related but different mixed smoothing to derive a Berry-Esseen bound with good dependence on $d$ for convex sets in the moderate-dimensional case.

%Thanks to the special structure of rectangles, these bounds depend on the dimension $d$ only via $\log d$, despite 

%in order to bound \eqref{eq: varrho}, for any rectangle $A\in\mathcal R$, we first approximate the indicator of $A$ by a Lipschitz-smooth function $h_A$ and then consider the Slepian-Stein interpolating function $[0,1]\ni t\mapsto T_th_A:= \Ep h_A(\sqrt t W + \sqrt{1-t} Z) - \Ep h_A(Z)$.

% for any rectangle $A\in\mathcal R$, we first approximate the indicator of $A$, $1\{\cdot\in A\}$, by a Lipschitz-smooth function and then approximate it further via convolution with a centered Gaussian distribution. Building on the results of \cite{B90} and \cite{FK20}, we then prove several bounds for sums of absolute values of partial derivatives of the resulting function. Thanks to the special structure of rectangles, these bounds depend on the dimension $d$ only via $\log d$, despite 
 
The rest of the paper is organized as follows. In the next section, we consider Gaussian approximations and derive various bounds on $\varrho$. In Section \ref{sec: bootstrap approximation}, we derive bounds for bootstrap approximations. In Section \ref{sec: special cases}, we discuss special cases, where bounds for the Gaussian approximations can be improved. In Section \ref{sub}, we demonstrate usefulness of our results in a particular problem of nonparametric statistics: many local means problem. In Section \ref{sec: smoothing inequalities}, we develop our new smoothing inequalities.  in Sections \ref{sec: main proofs}--\ref{sec: proofs smoothing inequalities}, we give all the proofs. Finally, in Section \ref{sec: auxiliary lemmas}, we collect several known lemmas, which are used in our derivations. 

\subsection{Notation} In the following, we assume $n\geq3$ and $d\geq3$ so that $\log n>1$ and $\log d>1$. Also, for any $\eta>0$, we use $R(0,\eta)$ to denote the centered $\ell_{\infty}$-ball with radius $\eta$, namely $R(0,\eta) := \{y\in\mathbb R^d\colon \|y\|_{\infty}\leq \eta\}$. For any $A = \prod_{j=1}^d(a_j,b_j]\in\mathcal R$ and $t\in\mathbb R$, we denote $A^t := \prod_{j=1}^d(a_j-t, b_j+t]$ and $(\partial A)^t := A^t\setminus A^{-t}$. For any $r = (r_1,\dots,r_d)^T\in\mathbb R^d$ and $t\in\mathbb R$, we denote $r+t := (r_1+t,\dots,r_d + t)^T\in\mathbb R^d$.
 %In addition, for any $r\in\mathbb R^d$, we define the function $1_r\colon\mathbb R^d\to\mathbb R$ by $1_r(y):=1\{y\leq r\}$ for all $y\in\mathbb R^d$. 
 Moreover, for any matrix $S = (S_{jk})_{j,k=1}^{J,K}$, we use $\|S\|_{\infty}$ to denote its $\ell_{\infty}$-norm, i.e. $\|S\|_{\infty} := \max_{1\leq j\leq J}\max_{1\leq k\leq K}|S_{ij}|$. For any matrices $S = (S_{jk})_{j,k=1}^{J,K}$ and $Q = (Q_{jk})_{j,k=1}^{J,K}$, we denote $\langle S,Q \rangle:= \sum_{j=1}^J\sum_{k=1}^K S_{jk}Q_{jk}$. Furthermore, we write $a\lesssim b$ if there exists a universal constant $C>0$ such that $a\leq Cb$. Finally, for any random variable $X$ and $q\geq 1$, we write $\|X\|_{L_q}$ and $\|X\|_{\psi_q}$ to denote the $L_q$- and $\psi_q$-norms of $X$, respectively, i.e. $\|X\|_{L_q}:=(\Ep |X|^q)^{1/q}$ and $\|X\|_{\psi_q} := \inf\{C>0\colon \Ep\psi_q(|X|/C)\leq 1\}$, where $\psi_q(x):=\exp(x^q)-1$ for all $x>0$.
We formally define $\|X\|_{\psi_q}$ in the same way even when $q\in(0,1]$, although it is not a norm but a quasi-norm. 

\section{Gaussian Approximations}\label{sec: main results}

%\subsection{Gaussian Approximation, Non-Smooth Case}
Let $\Sigma$ be any $d\times d$ positive definite symmetric matrix with unit diagonal entries and let $\sigma_*>0$ be the square root of its smallest eigenvalue. Define
$$
\varrho_{\Sigma} = \sup_{A\in\mathcal R}|\Pr (W\in A)-\Pr (Z\in A)|,\quad Z\sim N(0,\Sigma),
$$
so that $\varrho = \varrho_{\Sigma_W}$. In this subsection, we will derive bounds on $\varrho_{\Sigma}$. By substituting $\Sigma=\Sigma_W$, we are then able to derive direct bounds on $\varrho$. In addition, it will sometimes be possible to obtain better bounds on $\varrho$ using the triangle inequality, namely $\varrho \leq |\varrho - \varrho_{\Sigma}| + \varrho_{\Sigma}$. The latter is possible when $\Sigma_W$ is degenerate but can be well approximated by a non-degenerate $\Sigma$ in the $\|\cdot\|_{\infty}$-norm; see Remark \ref{rem: degenerate cases} below for details.

%Importantly, in this subsection, we only impose some mild moment conditions on the distribution of $X_i$'s and do not require existence of the density of this distribution. We therefore refer to results in this subsection as the Gaussian approximation in the non-smooth case.

Denote
$$
\Delta_0:=\frac{\log d}{\sigma_*^2}\|\Sigma-\Sigma_W\|_\infty\quad\text{and}\quad
\Delta_1:=\frac{(\log d)^2}{n^2\sigma_*^4}\max_{1\leq j\leq d}\sum_{i=1}^n\Ep X_{ij}^4.
$$
First, we derive a bound on $\varrho_{\Sigma}$ in the case of bounded $X_i$'s:
\begin{theorem}[Gaussian Approximation, Bounded Case]\label{t1}
Suppose that there is a constant $\delta>0$ such that $\|X_i\|_\infty/\sqrt{n} \leq \delta$ for every $i=1,\dots,n$ almost surely. Then
$$
\varrho_{\Sigma} \leq C\left\{\left(1\vee\left|\log\left(\frac{\Delta_1}{\log d} + \frac{\delta^2\log d}{\sigma_*^2}\right)\right|\right)\left(\Delta_0 + \sqrt{\Delta_1\log d} + \frac{(\delta\log d)^2}{\sigma_*^2}\right) + \frac{\delta(\log d)^{3/2}}{\sigma_*}\right\},
$$
where $C>0$ is a universal constant.
\end{theorem}
\begin{remark}[Optimality of Theorem \ref{t1}]\label{rem: main optimality}
The most important feature of Theorem \ref{t1} is that it implies a nearly optimal bound on $\varrho$. Indeed, assuming that (i) $n^{-1}\sum_{i=1}^n\Ep X_{ij}^2=1$ for all $j=1,\dots,d$, (ii) $|X_{i j}|\leq B_n$ almost surely for all $i=1,\dots,n$ and $j=1,\dots,d$ and some constant $B_n>0$, possibly depending on $n$, and (iii) $\sigma_{*,W}\geq b$ for some constant $b>0$, where $\sigma_{*,W}$ is the square root of the smallest eigenvalue of the correlation matrix of $W$, it follows from Theorem \ref{t1} that
\begin{equation}\label{eq: upper bound 1}
\varrho\leq \frac{CB_n(\log d)^{3/2}}{\sqrt n}\log n,
\end{equation}
where $C>0$ is a constant depending only on $b$; see Corollary \ref{thm: simple} below for details. On the other hand, we will show in Proposition \ref{prop:lower-bound} below that under mild conditions on $B_n$ and $d$, there exists a distribution of $X_i$'s such that
\begin{equation}\label{eq: lower bound 1}
\varrho\geq \frac{cB_n(\log d)^{3/2}}{\sqrt n},
\end{equation}
where $c>0$ is a constant that is independent of $(n,d,B_n)$. Comparing \eqref{eq: upper bound 1} and \eqref{eq: lower bound 1}, we conclude that the bound in Theorem \ref{t1} is optimal up to the $\log n$ factor. In addition, we will be able to get rid of the excessive $\log n$ factor in the case when $X_i$'s have an additive Gaussian component; see Theorem \ref{t-QG} below.
\qed
\end{remark}

\begin{remark}[Relation to Previous Work]\label{rem:ga-bounded}
The bound in \eqref{eq: upper bound 1} is as sharp as the bound obtained by \cite{FK20} for the log-concave $X_i$'s, which is the first
(nearly) sharp result in the non-degenerate case using implicit smoothing ideas and Stein's method. Subsequent work of \cite{L20} and \cite{KR20}, using the same implicit smoothing ideas combined with Lindeberg's method, obtained the following
bound for more general non-degenerate cases:
\begin{equation}\label{eq: upper bound alt}
\varrho\leq  C'\left( \left( \frac{B_n^6 (\log d)^4 \log(d n)}{n} \right)^{1/2}+ 
\left( \frac{(\log d)^{7} \log(d n)}{n} \right)^{1/2} \right) \log n,
\end{equation}
under the same conditions as those aforementioned in Remark \ref{rem: main optimality} and assuming also that $\Ep X_{ij}^2=1$ for all $i=1,\dots,n$ and $j=1,\dots,d$, where $C'>0$ is a constant depending only on $b$. Our bound \eqref{eq: upper bound 1} is considerably sharper. First, it has much better dependence on the dimension $d$. For example, with $B_n$ being independent of $n$ and $d\geq n$, \eqref{eq: upper bound 1} depends on $d$ via $(\log d)^{3/2}$ whereas \eqref{eq: upper bound alt} depends on $d$ via $(\log d)^4$, which is a large improvement in the high-dimensional case, where $\log d$ is growing as some fractional power of the sample size $n$. Second, \eqref{eq: upper bound 1} has much better dependence on the envelope constant $B_n$: \eqref{eq: upper bound 1} depends on $B_n$ linearly whereas \eqref{eq: upper bound alt} depends on $B_n$ via $B_n^3$. This second improvement is particularly important in classical applications to nonparametric statistics, where the intrinsic dimensionality of the problem often shows up not only via $d$ but also via $B_n$. We illustrate this point in Section \ref{sub} through an example. %  for $d=n$ and $B_n=1$, our bound improves
%the dependencies by a factor of $\sqrt{ 2^8 (\log n)^5 } \geq 7195$ for $n\geq 1000$. This emphasizes
%the value for obtaining good dependencies on the log of the dimension $d$.  Note however
%that all of the existing bounds remain qualitative in nature in the sense that the constants $C$ in \eqref{eq: upper bound 1} and $C'$ in \eqref{eq: upper bound alt} may be different.
\qed
\end{remark}

In Remark \ref{rem: degenerate cases} below, we discuss how the bound on $\varrho_{\Sigma}$ in Theorem \ref{t1} can be used to obtain bounds on $\varrho_{\Sigma_W}$ when $\Sigma_W$ is degenerate. To this end, we need the following Gaussian comparison lemma, which is a special case of Theorem 1.1 in \cite{FK20} and is similar to Theorem 2.2 in \cite{L20}.

\begin{lemma}[Gaussian Comparison; \cite{FK20}, Theorem 1.1]\label{FK:GC} Let $Z \sim N(0, \Sigma)$, where $\Sigma$ has unit entries on the diagonal, and $Z' \sim N(0, \Sigma')$, then
$$
\sup_{A \in \mathcal{R}} | \Pr(Z \in A) - \Pr (Z' \in A) | \leq  C \frac{D}{\sigma^{2}_*} \log d \left(1\vee \left | \log \frac{D}{\sigma^{2}_*} \right | \right),
$$
where $\sigma^{2}_*$ is the smallest eigenvalue of $\Sigma$ and $D =  \| \Sigma - \Sigma' \|_\infty$.
\end{lemma}

\begin{remark}[On Degenerate Cases]\label{rem: degenerate cases}
As we briefly mentioned above, having bounds on $\varrho_{\Sigma}$ for general $\Sigma$ in Theorem \ref{t1} rather than for $\Sigma=\Sigma_W$ is useful when $\Sigma_W$ is degenerate. Indeed, the direct application of Theorem \ref{t1} with $\Sigma=\Sigma_W$ gives a trivial bound as $\sigma_*=0$ in this case. Instead, by the triangle inequality and  Lemma 
\ref{FK:GC}, we have
\begin{equation}\label{eq: degenerate corollary}
\varrho \leq \varrho_{\Sigma} + C\Delta_0\left(1\vee \left|\log\left(\frac{\Delta_0}{\log d}\right)\right|\right),
\end{equation}
where $C>0$ is a universal constant. This bound can be combined with Theorem \ref{t1} to obtain useful bounds on $\varrho$ whenever there exists $\Sigma$ such that the square root of its smallest eigenvalue $\sigma_*$ is strictly positive and $\|\Sigma-\Sigma_W\|_{\infty}$ is small. We illustrate this point in Section \ref{sub} through an example.
\qed
\end{remark}

Next, we extend the result in Theorem \ref{t1} to allow for unbounded $X_i$'s. Denote
%$$
%\mathcal N:=\Big\{\mathrm N\subset\{1,\dots,n\}\colon |\mathrm N|\geq n-\log n -1\Big\}
%$$
%and
$$
\mathcal M := \left(\Ep\left[\max_{1\leq j\leq d}\max_{1\leq i\leq n}|X_{ij}|^4\right]\right)^{1/4}.
$$
Also, denote
$$
%\Lambda_0:=\frac{\log d}{\sigma_*^2}\max_{\mathrm N \in\mathcal N}\left\| \Sigma - \frac{1}{|\mathrm N|}\sum_{i\in\mathrm N}\Ep X_iX_i' \right\|_{\infty}\quad \text{and}\quad
\Lambda_1:=(\log d)^2(\log n)\log(dn).
$$
Finally, for all $\psi>0$, denote
$$
M(\psi):=\max_{1\leq i\leq n}\Ep\Big[\|X_i\|_{\infty}^41\{\|X_i\|_{\infty}>\psi\}\Big].
$$
We then have the following result:
\begin{theorem}[Gaussian Approximation, Unbounded Case]\label{thm: general unbounded}
For all $\psi>0$,
$$
\varrho_{\Sigma}\leq C\left\{(\log n)\left(\Delta_0 + \sqrt{\Delta_1\log d} + \frac{(\mathcal M \log d)^2}{n\sigma_*^2}\right) + \sqrt{\frac{\Lambda_1M(\psi)}{n\sigma_*^4}} + \frac{\psi(\log d)^{3/2}}{\sigma_*\sqrt n}\right\},
$$
where $C>0$ is a universal constant.
\end{theorem}

We now apply Theorems \ref{t1} and \ref{thm: general unbounded} to derive bounds on $\varrho = \varrho_{\Sigma_W}$ under easily interpretable conditions. Let $q\geq 4$ be a constant and let $\{B_n\}_{n\geq 1}$ be a sequence of positive constants, possibly growing to infinity. Also, let $\sigma_{*,W}$ be the square root of the smallest eigenvalue of the correlation matrix of $W$ and for all $j=1,\dots,d$, denote $\sigma_j:=(\Ep[W_{j}^2])^{1/2}$. Consider the following conditions:
\begin{itemize}
\item[{(E.1)}] $|X_{ij}/\sigma_j|\leq B_n\text{ for all }i=1,\dots,n\text{ and }j=1,\dots,d\text{ almost surely}$;
\medskip
\item[{(E.2)}] $\|X_{ij}/\sigma_j\|_{\psi_2}\leq B_n\text{ for all }i=1,\dots,n\text{ and }j=1,\dots,d$;
\medskip
\item[{(E.3)}] $\|\max_{1\leq j\leq d}|X_{ij}/\sigma_j|\|_{L_q}\leq B_n\text{ for all }i=1,\dots,n$;
\end{itemize}
and also consider condition 
\begin{itemize}
\item[{(M)\  }] $n^{-1} \sum_{i=1}^{n}\Ep|X_{i j}/\sigma_j|^{4}\leq B^2_n$ {\em for all} $j=1,\dots,p$;
\medskip
\end{itemize}
Similar conditions were previously used and motivated by applications in \cite{CCK13,CCK17,CCKK19,DZ17,FK20,D20}. %Note that (E.1) implies (M) but (E.2) and (E.3) do not. However, it is easy to bound the left-hand side of the inequality in (M) using either (E.2) or (E.3).
\begin{corollary}[Gaussian Approximation under Simple Conditions]\label{thm: simple}
Under condition (E.1), we have
\begin{equation}\label{eq: simple corrolary e1}
\varrho\leq \frac{C B_n(\log d)^{3/2}\log n}{\sqrt n \sigma_{*,W}^2},
\end{equation}
where $C>0$ is a universal constant; under conditions (M) and (E.2), we have
$$
\varrho \leq C\left( \frac{B_n(\log d)^{3/2}\log n}{\sqrt n \sigma_{*,W}^2} + \frac{B_n(\log d)^{2}}{\sqrt n\sigma_{*,W}} \right),
$$
where $C>0$ is a universal constant; under conditions (M) and (E.3), we have
$$
\varrho\leq C\left\{\frac{B_n(\log d)^{3/2}\log n}{\sqrt n \sigma_{*,W}^2} + \frac{B^2_n(\log d)^{2}\log n}{n^{1-2/q} \sigma_{*,W}^2} + \left(\frac{B_n^q(\log d)^{3q/2-4}(\log n)\log(d n)}{n^{q/2-1}\sigma_{*,W}^q}\right)^{\frac{1}{q-2}}\right\},
$$
where $C>0$ is a constant depending only on $q$.
%$$
%\varrho\leq C\left\{\frac{B_n(\log d)^{3/2}\log n}{\sqrt n \sigma_{*,W}^2} + \frac{B^2_n(\log d)^{2}\log n}{n^{1-2/q} \sigma_{*,W}^2} + \frac{B_n^{\frac{q}{q-2}}(\log d)^{3/2}}{\sqrt n \sigma_{*,W}^{\frac{q}{q-2}}}\left(1 + \left(\frac{(\log n)^2}{\log d}\right)^{\frac{1}{q-2}}\right)\right\}.
%$$
\end{corollary}
\begin{remark}[Dropping Condition (M)]
Like in the case of condition (E.1), meaningful bounds on $\varrho$ can be obtained without imposing condition (M) in the cases of (E.2) and (E.3) as well. This is so because both (E.2) and (E.3) imply bounds on the left-hand side of the inequality in condition (M). Indeed, it is straightforward to check that,  for all $j=1,\dots,d$, under (E.2), we have $n^{-1}\sum_{i=1}^n \Ep | X_{i j}/\sigma_j |^4 \lesssim B_n^2\log n + B_n^4/n^2$ whereas under (E.3), we have $n^{-1}\sum_{i=1}^n \Ep | X_{i j}/\sigma_j |^4 \lesssim B_n^{2q/(q-2)}$. We do not present the implied bounds on $\varrho$ for brevity of the paper.
\qed
\end{remark}
%\begin{remark}[Relation to Previous Work]
%ADD COMPARISON TO KUCHIBHOTLA AND RINALDO?
%\end{remark}

We conclude this section with the proposition that provides a lower bound on the convergence rate of $\varrho$ and demonstrates that the convergence rate in \eqref{eq: upper bound 1} is sharp up to the $\log n$ factor:
\begin{proposition}[Lower Bound on $\varrho$]\label{prop:lower-bound}
Let $\{B_n\}_{n\geq1}$ be a sequence of positive constants such that $B_n\geq2$ for all $n$. 
Suppose that $d$ depends on $n$ so that 
\[
\frac{B_n(\log d)^{3/2}}{\sqrt n}\to0,\qquad
\frac{B_n^4}{\sqrt{\log d}}\to0,\qquad
\frac{\sqrt n}{dB_n(\log d)^{3/2}}\to0
\]
as $n\to\infty$. Then, we can construct i.i.d.~random vectors $X_{n,1},\dots,X_{n,n}$ in $\mathbb{R}^d$ for every $n$ such that
$$\Ep[X_{n,ij}]=0, \ \Ep[X_{n,ij}^2]=1, \ |X_{n,ij}|\leq B_n,\quad\text{for all }n\geq 1, \ i=1,\dots,n, \ j=1,\dots,d;$$
and
\[
\liminf_{n\to\infty}\frac{\sqrt n}{B_n(\log d)^{3/2}}\sup_{x\in\mathbb{R}}\left|\Pr\left(\max_{1\leq j\leq d}\frac{1}{\sqrt n}\sum_{i=1}^nX_{n,ij}\leq x\right)-\Pr\left(\max_{1\leq j\leq d}Z_{j}\leq x\right)\right|>0,
\]
where $Z_1,Z_2,\dots$ are independent standard normal variables. 
\end{proposition} 
This proposition extends Proposition 1.1 in \cite{FK20} to allow for the $n$-dependent envelope constant $B_n$.

\section{Bootstrap Approximations}\label{sec: bootstrap approximation}
Since $\Sigma_W$ is in practice typically unknown, the Gaussian approximations obtained in the previous section are typically infeasible in the sense that we are unable to calculate probabilities $\Pr(Z\in A)$, $A\in\mathcal R$ and $Z\sim N(0,\Sigma_W)$, which is needed in statistical applications. In this section, we therefore consider bootstrap approximations. These approximations allow to estimate $\Pr(Z\in A)$ from the sample $X_1,\dots,X_n$. We focus on the multiplier and empirical bootstrap approximations.

Throughout this section, let $\Sigma$ be any $d\times d$ positive definite symmetric matrix with unit diagonal entries and let $\sigma_*>0$ be the square root of its smallest eigenvalue. This is the same convention as that in the previous section.

\subsection{Multiplier Bootstrap Approximation}
Let $\xi_1,\dots,\xi_n$ be i.i.d.~$N(0,1)$ random variables that are independent of $X:=(X_1,\dots,X_n)$. Denote $\bar X:=(\bar X_1,\dots,\bar X_d)^T:=n^{-1}\sum_{i=1}^n X_i$ and consider the (Gaussian) multiplier bootstrap version of $W$:
$$
W^{\xi}:=\frac{1}{\sqrt n}\sum_{i=1}^n\xi_i(X_i-\bar X).
$$
In this subsection, we are interested in bounding
$$
\varrho^{\xi}_{\Sigma}:=\sup_{A\in\mathcal R}|\Pr (W^{\xi}\in A\mid X)-\Pr (Z\in A)|,\quad Z\sim N(0,\Sigma),
$$
and, in particular, $\varrho^{\xi}:=\varrho^{\xi}_{\Sigma_W}$. Denote
$$
\Delta_0':=\frac{\log d}{\sigma_*^2}\left\| \Sigma - \frac{1}{n}\sum_{i=1}^n(X_i-\bar X)(X_i-\bar X)^T \right\|_{\infty}.
$$
The following result is as an easy consequence of Lemma \ref{FK:GC}.
\begin{theorem}[Multiplier Bootstrap]\label{thm: multiplier bootstrap}
We have
$$
\varrho^{\xi}_{\Sigma} \leq  C\Delta_0'\left(1\vee\left|\log\left(\frac{\Delta_0'}{\log d}\right)\right|\right),
$$
where $C>0$ is a universal constant.
\end{theorem}
%Next, we present a multiplier bootstrap analog of Corollary \ref{thm: simple}:
Applying Theorem \ref{thm: multiplier bootstrap} under easily interpretable conditions (M) and (E), we obtain the following analog of Corollary \ref{thm: simple}.
\begin{corollary}[Multiplier Bootstrap under Simple Conditions]\label{thm: multiplier bootstrap simple}
Under condition (E.1), we have with probability at least $1-\alpha$ that
$$
\varrho^{\xi}\leq \frac{C B_n(\log d)(\log n)\sqrt{\log(d/\alpha)}}{\sqrt n\sigma_{*,W}^2},
$$
where $C>0$ is a universal constant; under conditions (M) and (E.2), we have with probability at least $1-\alpha$ that
$$
\varrho^{\xi}\leq \frac{C B_n(\log d)(\log n)\sqrt{\log(d/\alpha)}}{\sqrt n\sigma_{*,W}^2},
$$
where $C>0$ is a universal constant; under conditions (M) and (E.3), we have with probability at least $1-\alpha$ that
$$
\varrho^{\xi}\leq \frac{C (\log d)(\log n)}{\sigma_{*,W}^2}\left(\frac{B_n\sqrt{\log(d/\alpha)}}{\sqrt n} + \frac{B_n^2(\log  d + \alpha^{-2/q})}{n^{1-2/q}}\right),
$$
where $C>0$ is a constant depending only on $q$.
\end{corollary}

\begin{remark}[Main Features of Corollary \ref{thm: multiplier bootstrap simple}]
The bounds in Corollary \ref{thm: multiplier bootstrap simple} are generally comparable with the corresponding bounds in Corollary \ref{thm: simple}. For example, under condition (E.1), combining Corollaries \ref{thm: simple} and \ref{thm: multiplier bootstrap simple}, we have that for some universal constant $C>0$, with probability at least $1-1/d$,
\begin{equation}\label{eq: mult boot cor e1 main}
\sup_{A\in\mathcal R}\Big|\Pr(W\in A) - \Pr(W^{\xi}\in A\mid X)\Big| \leq \frac{C B_n(\log d)^{3/2}\log n}{\sqrt n\sigma_{*,W}^2},
\end{equation}
which has the same right-hand side as that of \eqref{eq: simple corrolary e1}. Thus, we are able to obtain a feasible bootstrap approximation bound to probabilities $\Pr(W\in A)$ with the same convergence rate as that of the infeasible Gaussian approximation. Note also that under the assumption that $\sigma_{*,W}$ is bounded below from zero (strongly non-degenerate case in the terminology of \cite{FK20}), \eqref{eq: mult boot cor e1 main} is much better than the general bound (which does not require $\sigma_{*,W}>0$) following from the results in \cite{CCKK19}.
\qed
\end{remark}
\begin{remark}[Other Types of Multipliers]
In Theorem \ref{thm: multiplier bootstrap} and Corollary \ref{thm: multiplier bootstrap simple}, we focused on Gaussian multipliers but we note that similar results can be obtained for other multipliers, e.g.~Rademacher or Mammen multipliers; see \cite{M93} and \cite{CCKK19} for definitions. To do so, we can apply Theorem \ref{t1} conditional on $X_i$'s to bound
$$
\sup_{A\in\mathcal R}|\Pr(W^{\xi}\in A\mid X) - \Pr(W^{\zeta}\in A\mid X)|,
$$
where $W^{\zeta}$ is defined by analogy with $W^{\xi}$ with multipliers represented by random variables $\zeta_1,\dots,\zeta_n$ instead of $\xi_1,\dots,\xi_n$.
\qed
\end{remark}

\subsection{Empirical Bootstrap Approximation}
Let $X_1^*,\dots,X_n^*$ be i.i.d.~draws from the empirical distribution of $X_1,\dots,X_n$ and consider the empirical bootstrap version of $W$:
$$
W^*:=\frac{1}{\sqrt n}\sum_{i=1}^n(X_i^*-\bar X).
$$
In this subsection, we are interested in bounding
$$
\varrho^{*}_{\Sigma}:=\sup_{A\in\mathcal R}|\Pr (W^{*}\in A\mid X)-\Pr (Z\in A)|,\quad Z\sim N(0,\Sigma),
$$
and, in particular, $\varrho^{*}:=\varrho^{*}_{\Sigma_W}$. To do so, denote
$$
\mathcal M^*:=\max_{1\leq j\leq d}\max_{1\leq i\leq n}|X_{ij} - \bar X_j|
$$
and, for all $\psi>0$,
$$
M^*(\psi) :=\frac{1}{n}\sum_{i=1}^n\|X_i-\bar X\|_{\infty}^41\{\|X_i-\bar X\|_{\infty} >\psi\}.
$$
Also, denote
$$
\Delta_1':= \frac{(\log d)^2}{n^2\sigma_*^4}\max_{1\leq j\leq d}\sum_{i=1}^n (X_{ij}-\bar X_j)^4.
$$
The following result is an easy consequence of Theorem \ref{thm: general unbounded}.
\begin{theorem}[Empirical Bootstrap]\label{thm: empirical bootstrap}
For all $\psi>0$,
$$
\varrho^*_{\Sigma}\leq C\left\{(\log n)\left(\Delta_0' + \sqrt{\Delta_1'\log d} + \frac{(\mathcal M^* \log d)^2}{n\sigma_*^2}\right) + \sqrt{\frac{\Lambda_1M^*(\psi)}{n\sigma_*^4}} + \frac{\psi(\log d)^{3/2}}{\sigma_*\sqrt n}\right\},
$$
where $C>0$ is a universal constant.
\end{theorem}
%Next, we present an empirical bootstrap analog of Corollary \ref{thm: simple}:
Like in the previous subsection, applying Theorem \ref{thm: empirical bootstrap} under easily interpretable conditions (M) and (E), we obtain the following analog of Corollary \ref{thm: simple}.
\begin{corollary}[Empirical Bootstrap under Simple Conditions]\label{cor: emp boot simple}
Under condition (E.1), we have with probability at least $1-\alpha$ that
$$
\varrho^{*}\leq \frac{C B_n(\log d)(\log n)\sqrt{\log(d/\alpha)}}{\sqrt n\sigma_{*,W}^2},
$$
where $C>0$ is a universal constant; under conditions (M) and (E.2), we have with probability at least $1 - \alpha$ that
$$
\varrho^{*}\leq C\left(\frac{B_n(\log d)(\log n)\sqrt{\log(d/\alpha)}}{\sqrt n\sigma_{*,W}^2}+\frac{B_n(\log(dn))^2\sqrt{\log(1/\alpha)}}{\sqrt n\sigma_{*,W}}\right),
$$
where $C>0$ is a universal constant; under conditions (M) and (E.3), we have with probability at least $1-\alpha$ that
$$
\varrho^{*} \leq C\left(\frac{B_n(\log d)(\log n)\sqrt{\log(d/\alpha)}}{\sqrt n\sigma_{*,W}^2}+\frac{B_n\sqrt{\log(d n)}\log d}{n^{1/2-1/q}\alpha^{1/q}\sigma_{*,W}}\right),
$$
where $C>0$ is a constant depending only on $q$.
\end{corollary}
\begin{remark}[Main Features of Corollary \ref{cor: emp boot simple}]
Bounds for the empirical bootstrap approximation in this corollary are comparable but slightly worse than the corresponding bounds in Corollary \ref{thm: multiplier bootstrap simple} for the multiplier bootstrap approximation. However, since we only have upper bounds on the approximation error, this does not imply that the multiplier bootstrap is necessarily more precise than the empirical bootstrap. In fact, simulations in \cite{DZ17,CCKK19,D20} suggest the opposite may be true, with approximation errors being rather similar for most practical purposes. Note also that, like in the case of Corollary \ref{thm: multiplier bootstrap simple}, under the assumption that $\sigma_{*,W}$ is bounded below from zero, bounds in Corollary \ref{cor: emp boot simple} are much better than the general bound (which does not require $\sigma_{*,W}>0$) following from the results in \cite{CCKK19}.
\qed
\end{remark}

\section{Gaussian Approximation for Special Cases: Smooth and Zero Skewness}\label{sec: special cases}
%In this section, we show that the bounds in Theorems \ref{t1} and \ref{thm: general unbounded} can be improved in some special cases.  

In some special cases, the bounds in Theorems \ref{t1} and \ref{thm: general unbounded} can be improved. In this section, we consider two such cases and derive an improved version of Theorem \ref{t1}. For brevity, we do not provide an improved version of Theorem \ref{thm: general unbounded}. 

An interesting practical case occurs when $X_i$'s are generated with additive Gaussian noise (for example, due to measurement error or injection of noise for data privacy). As we demonstrate here, we can improve the bound in Theorem \ref{t1} by removing a logarithmic pre-factor in this case.  The proof of this result is relatively simple, so that it may be useful to review it before reading the more complicated proofs of Theorems \ref{t1} and \ref{thm: general unbounded}.

As before, let $X_1,\dots,X_n$ be independent centered random vectors in $\mathbb R^d$ but now suppose that  we observe only their noisy versions, say $\tilde X_1,\dots,\tilde X_n$, where  $\tilde X_i = X_i + g_i$ for some centered Gaussian $g_i$ and $G= \sum_{i=1}^n g_i/\sqrt{n} \sim N(0, \Sigma_0)$,  such that $G$ is independent of $X_i$'s. Assume that $\Sigma_0$ is non-degenerate and let $\sigma_{*,0}>0$ be the square root of its smallest eigenvalue. Assume also that $\Sigma_0$ has unit diagonal entries (this assumption is not essential and is made to simplify the results below; by rescaling, similar results can be obtained as long as all diagonal entries of $\Sigma_0$ are of the same order). In addition, let $\Sigma$ be any $d\times d$ non-negative definite symmetric matrix and let $\tilde\Sigma:=\Sigma + \Sigma_0$. Denote
$
\tilde W := \sum_{i=1}^n \tilde X_i/\sqrt{n}$  and $W := \sum_{i=1}^n X_i/\sqrt{n}$, so that $\Sigma_{\tilde W}=\Sigma_{W}+\Sigma_0$, where $\Sigma_{\tilde W}:=\Ep\tilde W\tilde W^T$ and $\Sigma_W:=\Ep WW^T$.
Also, denote
$$
\tilde\varrho_{\Sigma} := \sup_{A\in\mathcal R}|\Pr (\tilde W\in A)-\Pr (\tilde Z\in A)|,\quad\tilde Z\sim N(0,\tilde \Sigma).
$$
Below, we derive a bound on $\tilde\varrho_{\Sigma}$. Since the distribution of $\tilde X_i$'s is smooth because of the presence of the additive Gaussian components $g_i$, we refer to the results below as the Gaussian approximation in the smooth case. Following the literature, e.g.~\cite{Z20}, we also sometimes refer to the distribution of $\tilde X_i$'s as quasi-Gaussian.

%Denote
%$$
%\Delta_2=\frac{\log^{3/2} d}{n^{3/2}\sigma_*^3}\max_{1\leq j\leq d}\sum_{i=1}^n\Ep {X}_{ij}^{3}.
%$$

%\begin{theorem}[Gaussian Approximation, Smooth Case]\label{t-QG}
%Suppose that there is a constant $\delta>0$ such that $\|X_i\|_\infty/\sqrt{n} \leq \delta$ for every $i=1,\dots,n$ almost surely. 
%Then
%$$
%\tilde\varrho_{\Sigma} \leq C \left(\tilde\Delta_0 + \sqrt{\tilde\Delta_1\log d} + \frac{\delta(\log d)^{3/2}}{\sigma_*} + \tilde\Delta_2\right),
%$$
%where $C>0$ is a universal constant and
%$$
%\tilde\Delta_0=\frac{\log d}{\sigma_{*,0}^2}\|\Sigma-\Sigma_{W}\|_\infty,\ \
%\tilde\Delta_1=\frac{(\log d)^2}{n^2\sigma_{*,0}^4}\max_{1\leq j\leq d}\sum_{i=1}^n\Ep {X}_{ij}^{4}, \ 
%\tilde\Delta_2=\frac{(\log d)^{3/2}}{n^{3/2}\sigma_{*,0}^3}\max_{1\leq j\leq d}\sum_{i=1}^n\Ep {X}_{ij}^{3}.
%$$
%In particular, if $n^{-1}\sum_{i=1}^n \Ep X_{ij}^2=1$ for all $j=1,\dots,d$, then
%$$
%\tilde\varrho_{\Sigma_W} \leq 3C\frac{\delta(\log d)^{3/2}}{\sigma_{*,0}^3}.
%$$
%\end{theorem}
\begin{theorem}[Gaussian Approximation, Smooth Case]\label{t-QG}
Suppose that there are constants $\delta,c>0$ such that $\|X_i\|_\infty/\sqrt{n} \leq \delta$ for every $i=1,\dots,n$ almost surely and $\delta\sqrt{\log d}\leq c\sigma_{*,0}$. 
Then
$$
\tilde\varrho_{\Sigma} \leq C \left(\tilde\Delta_0 + \tilde\Delta_1\right),
$$
where $C>0$ is a constant depending only on $c$ and
$$
\tilde\Delta_0:=\frac{\log d}{\sigma_{*,0}^2}\|\Sigma-\Sigma_{W}\|_\infty,\quad
%\tilde\Delta_1=\frac{(\log d)^2}{n^2\sigma_{*,0}^4}\max_{1\leq j\leq d}\sum_{i=1}^n\Ep {X}_{ij}^{4}, \ 
\tilde\Delta_1:=\frac{(\log d)^{3/2}}{n^{3/2}\sigma_{*,0}^3}\max_{1\leq j\leq d}\sum_{i=1}^n\Ep {X}_{ij}^{3}.
$$
%In particular, if $n^{-1}\sum_{i=1}^n \Ep X_{ij}^2\leq1$ for all $j=1,\dots,d$, then
%$$
%\tilde\varrho_{\Sigma_W} \leq \frac{C\delta(\log d)^{3/2}}{\sigma_{*,0}^3}.
%$$
\end{theorem}
\begin{remark}[Optimality of Theorem \ref{t-QG}]\label{rem:t-QG}
In comparison with Theorem \ref{t1} and Corollary \ref{thm: simple}, Theorem \ref{t-QG} does not contain the logarithmic pre-factor. Assuming that (i) $n^{-1}\sum_{i=1}^n \Ep X_{i j}^2\leq 1$ for all $j=1,\dots,d$, (ii) $|X_{ij}|\leq B_n$ almost surely for all $i=1,\dots,n$ and $j=1,\dots,d$ and some constant $B_n>0$, possibly depending on $n$, and (iii) $\sigma_{*,0}^3\geq b$ for some constant $b>0$, it follows from Theorem \ref{t-QG} that
\begin{equation}\label{eq: optimal bound quasi-gaussian}
\tilde\varrho_{\Sigma_W}\leq \frac{C B_n(\log d)^{3/2}}{\sqrt n},
\end{equation}
where $C>0$ is a constant depending only on $b$. This bound is optimal in the quasi-Gaussian case with respect to both the sample size $n$ and the dimension $d$, as follows from Proposition 1.1 in \cite{FK20}, which yields a lower bound and allows the lower bound to be achieved by the quasi-Gaussian distributions. Hence, it is not possible to obtain a better bound without imposing further conditions, such as zero skewness or symmetry of the distribution of $X_i$'s.
\qed
\end{remark}
%\begin{corollary}\label{cor: GA smooth case}

%\end{corollary}

Our second example in this section demonstrates that, with a bit more structure, namely assuming the zero skewness condition, we can further improve the bounds. Most notably, the theorem below implies dependence on $n$ via $1/n$ instead of $1/\sqrt n$ for uniformly bounded $X_i$'s.
%We can indeed derive better bounds than those of Theorem \ref{t-QG} in the zero skewness case:
\begin{theorem}[Gaussian Approximation, Smooth and Zero Skewness Case]\label{t-QG2}
Under the assumptions of Theorem \ref{t-QG}, assume additionally that
\ben{\label{zero-skew}
\Ep[X_{ij}X_{ik}X_{il}]=0\qquad\text{for all }i=1,\dots,n\text{ and }j,k,l=1,\dots,d.
}
Then
$$
\tilde\varrho_{\Sigma} \leq C \left(\tilde\Delta_0 + \tilde\Delta_2\right),
$$
where $C>0$ is a constant depending only on $c$ and
$$
\tilde\Delta_0:=\frac{\log d}{\sigma_{*,0}^2}\|\Sigma-\Sigma_{W}\|_\infty,\quad
\tilde\Delta_2:=\frac{(\log d)^2}{n^2\sigma_{*,0}^4}\max_{1\leq j\leq d}\sum_{i=1}^n\Ep {X}_{ij}^{4}.
$$
%In particular, if $n^{-1}\sum_{i=1}^n \Ep X_{ij}^2\leq1$ for all $j=1,\dots,d$, then
%\begin{equation}\label{eq: zero skewness bound theorem}
%\tilde\varrho_{\Sigma_W} \leq \frac{C\delta^2(\log d)^{2}}{\sigma_{*,0}^4}.
%\end{equation}
\end{theorem}

\begin{remark}[Optimality of Theorem \ref{t-QG2}]
Assuming that (i) $n^{-1}\sum_{i=1}^n \Ep X_{i j}^2\leq 1$ for all $j=1,\dots,d$, (ii) $|X_{ij}|\leq B_n$ almost surely for all $i=1,\dots,n$ and $j=1,\dots,d$ and some constant $B_n>0$, possibly depending on $n$, and (iii) $\sigma_{*,0}^3\geq b$ for some constant $b>0$, it follows from Theorem \ref{t-QG} that
\begin{equation}\label{eq: quasi gaussian zero skewness bound}
\tilde\varrho_{\Sigma_W}\leq \frac{C B_n^2(\log d)^{2}}{n},
\end{equation}
where $C>0$ is a constant depending only on $b$. This bound is optimal in the quasi-Gaussian case with zero skewness with respect to both the sample size $n$ and the dimension $d$, as we prove in Proposition \ref{prop:lower-bound2} below. Moreover, neither the zero skewness nor quasi-Gaussian conditions can be dropped in general to get such dependences. In fact, if the former is not satisfied, we can at best get \eqref{eq: optimal bound quasi-gaussian}, as discussed in Remark \ref{rem:t-QG} above. 
Similarly, it is well-known that the dependence on $n$ should be $1/\sqrt n$ in the normal approximation rate for sums of independent Rademacher variables (see e.g.~page 112 of \cite{Petrov75}), so we cannot drop the quasi-Gaussian assumption in general to get a bound proportional to $1/n$. Finally, note that \eqref{eq: quasi gaussian zero skewness bound} is substantially better than \eqref{eq: optimal bound quasi-gaussian}, meaning that imposing the zero skewness condition is rather helpful in the quasi-Gaussian case.
%However, note that it is still unclear whether we can establish a convergence rate of the form $(\log^2d/n)^a$ for some $a\in(0,1/2]$ ($a=1/2$ is of course preferable) under the quasi-Gaussian assumption. 
\qed
\end{remark}

We conclude this section with the proposition that provides a lower bound on the convergence rate of $\tilde\varrho_{\Sigma_W}$ under the quasi-Gaussian and zero skewness conditions and demonstrates that the convergence rate in \eqref{eq: quasi gaussian zero skewness bound} is sharp:

\begin{proposition}\label{prop:lower-bound2}
Let $X=(X_{ij})_{i,j=1}^\infty$ be an array of i.i.d.~random variables such that $\|X_{ij}\|_{\psi_1}<\infty$, $\Ep[X_{ij}]=0$, $\Ep[X_{ij}^2]=1$, $\Ep[X_{ij}^3]=0$ and $\gamma:=\Ep[X_{ij}^4]-3\neq0$. 
Let $W=n^{-1/2}\sum_{i=1}^n X_i$ with $X_i:=(X_{i1},\dots,X_{id})^T$. 
Suppose that $d$ depends on $n$ so that $(\log d)^2/n\to0$ and $(\log d)^3/n\to\infty$ as $n\to\infty$. 
Also, let $Z\sim N(0, I_d)$.
Then
\[
\limsup_{n\to\infty}\frac{n}{(\log d)^2}\sup_{x\in\mathbb{R}}\left|\Pr\left(\max_{1\leq j\leq d}W_j\leq x\right)-\Pr\left(\max_{1\leq j\leq d}Z_j\leq x\right)\right|>0.
\]
\end{proposition} 
\begin{remark}[Relation to Previous Work]
This proposition complements Theorem 3 in \cite{DL20}, who showed that the Gaussian approximation with vanishing error is not possible if $(\log d)^2/n^{1 + \delta}\nrightarrow 0$ for some $\delta>0$ and $X_{ij}$'s are Rademacher random variables.
\qed
\end{remark}
% It will be better to refer to Das and Lahiri (2020) here.

\section{Application to Many Local Means Problem}\label{sub}
An interesting setting that illustrates the value of our new bounds is the problem of many local
means, which plays a fundamental role in nonparametric statistics.  In this problem,
the dimensionality of the problem actually shows up in the envelope and moments of $X_i$'s and not just via $\log d$. 
We illustrate this point with the following simple example. Consider i.i.d. random vectors $V_1,\dots,V_n$ in $\Bbb{R}^{\kappa}$ and non-overlapping
regions $(R_j)_{j=1}^d$ that partition the support of $V_i$'s such that $p_j := \Pr \{V_i \in R_j\}= p$ for all $j=1,\dots,d$ and $d=1/p$. Define components
of $X_i$ via:
$$
X_{ij} = \frac{1\{V_i \in R_j\} - p}{\sqrt{ p(1- p)}}, \quad j=1,\dots,d,
$$
and set $W:=n^{-1/2}\sum_{i=1}^n X_i$. The distribution of $W$ over the class of rectangles $\mathcal R$ is of interest in testing hypotheses about the means of $X_{ij}$'s. 

To apply our results in this setting, observe that
$$
\Sigma_W = \Ep WW^T =   \frac{p}{{p(1-p)}} I_d  - \frac{p^2}{{p(1-p)}} 1_d 1_d',
$$
where $1_d:=(1,\dots,1)^T\in\mathbb R^d$. The smallest eigenvalue of $\Sigma_W$ is $$ \frac{p}{{p(1-p)}}- \frac{p^2}{{p(1-p)}} d = 0,$$
so this is actually a degenerate case.  On the other hand, we have for $\Sigma :=  \frac{p}{{p(1-p)}} I_d$ that
$$
\|\Sigma_W  -  \Sigma\|_\infty \leq p/(1-p) = 1/(d-1)
$$
and all eigenvalues of $\Sigma$ are bounded below from zero. Thus, applying \eqref{eq: degenerate corollary} and Theorem \ref{t1} with $\Sigma =  \frac{p}{{p(1-p)}} I_d$ and $\delta := \sqrt{d/n}$ to bound $\varrho_{\Sigma}$, we have that
\begin{equation}\label{eq: rho example}
\varrho \lesssim \frac{(\log d)^2}{d} + \sqrt{\frac{ d {(\log d)^3}}{n}} \log n.
\end{equation}
%This bound emphasizes that dependence on $\delta$ in Theorem \ref{t1} and thus on $B_n$ in Corollary \ref{thm: simple} under condition (E.1) is nearly optimal: if this dependence were polynomially better, we would be able to obtain $\varrho\to0$ with $d/n\to\infty$, which is not possible because $d/n\to\infty$ implies that the number of observations in many regions $R_j$ is converging to zero as $n\to\infty$.
This bound may be rather poor if $d\to\infty$ slowly. Fortunately, we can combine \eqref{eq: rho example} with the bound we previously derived in \cite{CCKK19} to obtain
\begin{equation}\label{combo}
\varrho\lesssim \left( \frac{(\log d)^2}{d} + \sqrt{\frac{ d {(\log d)^3}}{n}} \log n \right) \wedge \left(\frac{ d {(\log n)^5}}{n}\right)^{1/4},
\end{equation}
which is much better than \eqref{eq: rho example} when $d\to\infty$ slowly. Specifically, \eqref{combo} gives
\begin{equation}\label{eq: d conv new}
\varrho \to 0\quad\text{if}\quad\frac{d(\log n)^5}{n}\to 0.
\end{equation}

%If $d \geq n^{1/3}$ then the second term in the bound dominates. Combine this with bound in \cite{CCKK19}:
%Thus the combined bound is
%\begin{equation}\label{combo}
%\left( \frac{(\log d)^2}{d} + \sqrt{\frac{ d {(\log d)^3}}{n}} \log n \right) \wedge \left(\frac{ d {(\log(d n))^5}}{n}\right)^{1/4},
%\end{equation}
%If $d \ll n^{1/5}$, ignoring logs, the second part of the bound is better; otherwise the first part is better.

Turning now to the alternative bounds in the literature, we note that the direct application of results in  \cite{L20} and \cite{KR20} give an infinite bound on $\varrho$ because $\Sigma_W$ is degenerate. This is of course an unfair comparison, so it is possible to modify the arguments in   \cite{L20} and \cite{KR20} 
to have the dependencies in their bounds via $\| \Sigma - \Sigma_W\|_\infty$, as we did in Remark \ref{rem: degenerate cases}, and obtain
\begin{equation}\label{eq: kr}
\varrho\lesssim \frac{(\log d)^2}{d} + \left(\sqrt{\frac{d^3(\log d)^4\log n}{n}} + \sqrt{\frac{(\log d)^7 \log  (dn) }{n}}\right)\log n% \sqrt{\frac{ d^5 {(\log d)^3}}{n}} \log n,
\end{equation}
This bound gives, when $d \to \infty$:
\begin{equation}\label{eq: d conv alt}
\varrho \to 0\quad\text{if}\quad \frac{d^3(\log n)^7}{n}\to 0.
\end{equation}
Comparing \eqref{eq: d conv new} with \eqref{eq: d conv alt}, we conclude that \eqref{combo} is substantially better than \eqref{eq: kr}.

%where dependency on the envelope $B = \sqrt{d}$. If $d \geq n^{1/3}$ the bound still diverges to infinity.
%Only if $d \ll n^{1/3}$ and $d \to \infty$ the bound converges to zero, but is considerably worse than (\ref{combo})
%by polynomial in $n$ factor.  This illustrates the important point
%that the envelope $B$ is not a mere constant and may encode the dimensionality of the problem. It also highlights
%the value of the bounds developed here.

In addition, it is possible to obtain a bound on $\varrho$ via the Hungarian coupling. In particular, results in \cite{R94} and \cite{GN10} imply that one can construct a centered Gaussian random vector $G$ in $\mathbb R^d$ such that
$$
\| W - G\|_\infty \lesssim  \sqrt{\frac{\log n}{ (n p)^{1/\kappa}}} + \sqrt{ \frac{\log^2 n}{ np }}
$$
almost surely.  Moreover, \cite{Beck85} showed that that bound is sharp up to possible log factors when $\kappa \geq 2$. 
Combining this bound with the  anti-concentration inequality in Lemma \ref{lem: anticoncentration} implies
\begin{equation}\label{eq: rgn}
\varrho \lesssim \sqrt{ \log d} \left( \sqrt{\frac{\log n}{ (n p)^{1/\kappa}}} + \sqrt{ \frac{\log^2 n}{ np }} \right)
\end{equation}
When $d\geq n^{1/3}$, which is the most relevant case, this bound is better than that in \eqref{combo} by a $(\log n)^2$ factor for $\kappa=1$ but worse for $\kappa\geq 2$. For $\kappa\geq 3$, this bound is much worse than that in \eqref{combo} by a polynomial-in-$n$ factor regardless of $d$.

%For $\kappa =1$, this bound reduces to
%$$
%\varrho\lesssim \sqrt{ \frac{d(\log n)^2 \log d}{ n }},
%$$
%which is better than \eqref{combo} by a $(\log n)^2$ factor when $d \geq n^{1/3}$, which is the most relevant case. For $\kappa \geq 2$, \eqref{eq: rgn} is worse than \eqref{combo} when $d\geq n^{1/3}$. % Analogous comments applies for $\kappa =2$, where
%the HC bound is better by log factors. As soon as $\kappa \geq 3$, the bound is worse, for $\kappa = 3$:
%$$
%\sqrt{ \log d} \left( \sqrt{\frac{d^{1/3}\log n}{n^{1/3}}} + \sqrt{ \frac{d \log^2 n}{ n}} \right),
%$$
%which has $(d/n)^{-1/6}$ dependency.

\section{Mixed Smoothing Inequalities}\label{sec: smoothing inequalities}
Let $\phi>0$, $\epsilon\in[0,1]$, and $A=\prod_{j=1}^d(a_j,b_j]\in\mathcal R$. Also, let $\Sigma$ be a $d\times d$ symmetric positive definite matrix with unit diagonal entries, and let $\sigma_*>0$ be the square root of the smallest eigenvalue of $\Sigma$. Consider functions
$g^{\phi}\colon\mathbb{R}\to\mathbb{R}$, $m^{A,\phi}\colon\mathbb{R}^{d}\to\mathbb{R}$,
and $\rho^{A,\phi,\epsilon,\Sigma}\colon\mathbb{R}^{d}\to\mathbb{R}$ by
\[
g^{\phi}(t):=\begin{cases}
1 & \text{if }t\leq0,\\
1-\phi t & \text{if }0<t<1/\phi,\\
0 & \text{if }t\geq1/\phi,
\end{cases}
\]
\begin{equation}\label{eq: m function}
m^{A,\phi}(w):=g^{\phi}\left(\max_{1\leq j\leq d}[(w_{j}-b_{j})\vee(a_j-w_j)]\right),\quad w\in\mathbb{R}^{d},
\end{equation}
and
\begin{equation}\label{eq: rho function}
\rho^{A,\phi,\epsilon,\Sigma}(w):=\mathbb{E}m^{A,\phi}(w+\epsilon Z),\quad w\in\mathbb{R}^{d},
\end{equation}
where $Z$ is a centered normal random vector in $\mathbb{R}^{d}$ with covariance matrix $\Sigma$.
For large $\phi$ and small $\epsilon$, the function $\rho^{A,\phi,\epsilon,\Sigma}(\cdot)$
provides a good approximation to the indicator function $1_A$
but, in contrast to the indicator function, is smooth. In particular,
we will prove the following inequalities, which play a key role in obtain sharp bounds for the Gaussian approximation.
\begin{lemma}\label{lem: smoothing inequality mixed 1}
Let $v\in\mathbb Z$ and $K\in\mathbb R$ be such that $v\geq1$ and $K>0$. Set $\eta = \eta_d = K/\sqrt{\log d}$. Then
\begin{equation}
\sup_{A\in\mathcal R}\sup_{w\in\mathbb R^d}\sum_{j_{1},\dots,j_{v}=1}^{d}\sup_{y\in R(0,\epsilon\sigma_*\eta)}|\partial_{j_{1},\dots,j_{v}}\rho^{A,\phi,\epsilon,\Sigma}(w+y)|\leq C\frac{\phi(\log d)^{(v-1)/2}}{(\epsilon\sigma_*)^{v-1}},\label{eq: lemma}
\end{equation}
where $C>0$ is a constant depending only on $v$ and $K$.
\end{lemma}

\begin{lemma}\label{lemma:fk2.2}
Let $v\in\mathbb{Z}$ and $K\in\mathbb{R}$ be such that $v\geq1$ and $K>0$. Set $\eta=\eta_d=K/\sqrt{\log d}$. 
Then
\begin{equation*}%\label{eq:AHT}
\sup_{A\in\mathcal R}\sup_{w\in\mathbb R^d}\sum_{j_1,\dots,j_v=1}^d\sup_{y\in R(0,\epsilon\sigma_*\eta)}|\partial_{j_1,\dots, j_v}\rho^{A,\phi,\epsilon,\Sigma} (w+y)|
\leq C\frac{(\log d)^{v/2}}{(\epsilon\sigma_*)^{v}},
\end{equation*}
where $C>0$ is a constant depending only on $v$ and $K$. 
\end{lemma}

\begin{lemma}\label{lemma:vanish}
Let $A=\prod_{j=1}^d(a_j,b_j]\in\mathcal R$ and $v\in\mathbb Z$ be such that $v\geq 1$. Then for all $\kappa,\eta>0$ with $\kappa>\eta$,
\begin{equation*}%\label{eq:AHT}
\sup_{w\in (A^{2\epsilon\kappa+ \phi^{-1}}\setminus A^{-2\epsilon\kappa})^c}\sum_{j_1,\dots,j_v=1}^d\sup_{y\in R(0,\epsilon\sigma_*\eta)}|\partial_{j_1,\dots, j_v}\rho^{A,\phi,\epsilon,\Sigma} (w+y)|
\leq C\frac{d^v}{(\epsilon\sigma_*)^{v}}e^{-(\kappa-\eta)^2/4},
\end{equation*}
where $C>0$ is a constant depending only on $v$. 
\end{lemma}
\begin{remark}[Relation to Previous Work]
All three lemmas here are new. Their proofs are inspired by the original ideas of \cite{B90}, who derived Lemma \ref{lemma:fk2.2} without smoothing (with $\phi=\infty$) and $\eta=0$. See also \cite{FK20} who extended the result of \cite{B90} to allow for $\eta>0$ in Lemma \ref{lemma:fk2.2} using related methods of \cite{AHT98}. Having $\eta \neq 0$ is important for establishing the optimal dependence on the envelopes.
\qed
\end{remark}

\section{Proofs for Section \ref{sec: main results}}\label{sec: main proofs}
\begin{proof}[Proof of Theorem \ref{t1}]
For all $i=1,\dots,n$, we denote $\xi_i := X_i/\sqrt n$, so that $W=\sum_{i=1}^n \xi_i$ and $\|\xi_i\|_{\infty}\leq\delta$. 
Working with $\xi_i$'s is a little more convenient than working with $X_i$'s.
Also, we assume, without loss of generality, that $W$ and $Z$ are independent.  In addition, since $\varrho_{\Sigma}\leq1$, we assume, again without loss of generality, that
\ben{\label{wlog}
 \frac{\Delta_1}{\log d} + \frac{\delta^2\log d}{\sigma_*^2} \leq \frac{1}{3}.
}
Further, we write $\varrho'=\varrho_\Sigma$ for brevity. 
%Further, observe that for any $w\in\mathbb R^d$ and $A = \prod_{j=1}^d(a_j,b_j]\in\mathcal R$, $\Pr ( W \in A) = \Pr (W \leq b)-
%\Pr (W \leq a)$, where $a=(a_1,\dots,a_d)^T\in\mathbb R^d$ and $b=(b_1,\dots,b_d)^T\in\mathbb R^d$.  Therefore, it suffices to prove the asserted claim with $\varrho_{\Sigma}$ replaced by
%$$
%\varrho' = \varrho'_{\Sigma} := \sup_{r\in \mathbb R^d} |\Pr (W\leq r)- \Pr (Z\leq r)|,\qquad Z\sim N(0,\Sigma),
%$$
%which is what we do below.

Now, for any bounded measurable function $h:\mathbb{R}^d\to\mathbb{R}$ and $t\in[0,1]$, define $T_t h:\mathbb{R}^d\to\mathbb{R}$ by
\be{
T_t h(w)=\Ep h (\sqrt{1-t} w+\sqrt{t}Z)-\Ep h(Z),\qquad w\in\mathbb{R}^d.
}
Also, note that we have $\Pr(V\in A)=\Pr(V^\diamond\leq r)$ for any $A=\prod_{j=1}^d(a_j,b_j]\in\mathcal R$ and random vector $V$ in $\mathbb R^d$, where $V^\diamond=(V^T,-V^T)^T$ and $r=(b_1,\dots,b_d,-a_1,\dots,-a_d)^T$. 
Thus, by Lemmas \ref{lem: rao} and \ref{lem: anticoncentration} and the fact that $\Pr(\|Z\|_{\infty}>\sqrt{4\log d})\leq 1/2$, we have
\begin{equation}\label{3}
\varrho'\lesssim 
\sup_{A\in \mathcal{R}}|\Ep T_t 1_A (W)|
+\sqrt{\frac{t}{1-t}} \log d.
\end{equation}
By taking the value of $t$ appropriately, we will deduce a recursive inequality for $\varrho'$; see \eqref{somewhere} below. 
In particular, we set
\ben{\label{t}
t:= \frac{\Delta_1}{\log d}+\frac{\delta^2\log d}{\sigma_*^2}.
}
Note here that because of \eqref{wlog}, $|\log t|\geq 1$ and $1/\sqrt{1-t}\leq 2$. %1/root{1 -1/3} = 1/root{2/3}

Further, fix $\phi>0$, to be chosen below in \eqref{t0}, and for any $A\in\mathcal R$, consider the smoothed indicator function $m^{A,\phi}\colon \mathbb R^d\to\mathbb R$ as in \eqref{eq: m function} of Section \ref{sec: smoothing inequalities}. Denoting $\tilde W = \sqrt{1-t} W + \sqrt{t}Z$, we have by Lemma \ref{lem: anticoncentration} that
\begin{align*}
\Pr (\tilde W\in A)&\leq \Ep m^{A,\phi}(\tilde W)= \Ep m^{A,\phi}( Z) + \Ep m^{A,\phi}(\tilde W) - \Ep m^{A,\phi}(Z)\\
&\leq \Pr (Z\in A^{1/\phi}) + \Ep m^{A,\phi}(\tilde W) - \Ep m^{A,\phi}(Z)\\
&\leq \Pr (Z\in A) + C\sqrt{\log d}/\phi + \Ep m^{A,\phi}(\tilde W) - \Ep m^{A,\phi}(Z)
\end{align*}
and, similarly,
\begin{align*}
&\Pr (Z\in A)\leq \Pr (Z \in A^{-1/\phi}) + C\sqrt{\log d}/\phi\leq \Ep  m^{A^{-1/\phi},\phi}(Z) + C\sqrt{\log d}/\phi\\
&\quad = \Ep  m^{A^{-1/\phi},\phi}(\tilde W) + \Ep  m^{A^{-1/\phi},\phi}(Z) - \Ep  m^{A^{-1/\phi},\phi}(\tilde W ) + C\sqrt{\log d/}\phi\\
&\quad\leq \Pr  (\tilde W\in A) + \Ep  m^{A^{-1/\phi},\phi}(Z) - \Ep  m^{A^{-1/\phi},\phi}(\tilde W) + C\sqrt{\log d/}\phi,
\end{align*}
where $C>0$ is a universal constant. 
Hence,
$$
\sup_{A\in\mathcal R}\left|\Pr (\tilde W\in A) - \Pr (Z\in A)\right| \lesssim \sup_{A\in\mathcal R}\left|\Ep  m^{A,\phi}(\tilde W) - \Ep  m^{A,\phi}(Z)\right| + \sqrt{\log d}/\phi,
$$
and so,
\begin{equation}\label{eq: some reduction}
\sup_{A\in\mathcal R}|\Ep  T_t1_A(W)|\lesssim \sup_{A\in\mathcal R}|\Ep  T_tm^{A,\phi}(W)|+\sqrt{\log d}/\phi.
\end{equation}
Given \eqref{3} and \eqref{eq: some reduction}, we need to bound $\sup_{A\in\mathcal R}|\Ep  T_tm^{A,\phi}(W)|$.

To do so, fix any $A\in\mathcal{R}$ (we will take the supremum in \eqref{somewhere}), write $h:=m^{A,\phi}$, and proceed to bound $|\Ep T_t h(W)|$. By the fundamental theorem of calculus and the fact that $\Ep T_1h(W)=0$,
$$
\Ep T_th(W)=-\frac{1}{2}\int_t^1\Ep\left\langle \nabla h(\sqrt{1-s}W+\sqrt s Z),\frac{Z}{\sqrt s} - \frac{W}{\sqrt{1-s}} \right\rangle ds,
$$
and so, using Lemma \ref{lem: stein version},
\begin{equation}\label{eq: annoying lemma gone}
\Ep T_th(W) = -\frac{1}{2}\int_t^1\Ep\left[ \langle\Sigma,\nabla^2 h_s(\sqrt{1-s}W)\rangle - \left\langle\frac{W}{\sqrt{1-s}},\nabla h_s(\sqrt{1-s}W)\right\rangle\right]ds,
\end{equation}
where for all $s\in[t,1]$, the funciton $h_s\colon \mathbb R^d\to\mathbb R$ is given by 
$$
h_s(w)=\Ep h(w+\sqrt s Z),\quad w\in\mathbb R^d.
$$ 
Here, it is useful to note that $h_s = \rho^{A,\phi,\sqrt s,\Sigma}$, where $\rho^{A,\phi,\epsilon,\Sigma}$ with $\epsilon = \sqrt s$ is the function appearing in \eqref{eq: rho function} of Section \ref{sec: smoothing inequalities}. In particular, $h_s$ is infinitely differentiable, with derivatives satisfying bounds in Lemmas \ref{lem: smoothing inequality mixed 1}, \ref{lemma:fk2.2}, and \ref{lemma:vanish}. These bounds will be used below.

To bound the integral in \eqref{eq: annoying lemma gone}, we employ the exchangeable pair approach in Stein's method for multivariate normal approximation by \cite{ChMe08} and \cite{ReRo09} along with a symmetry argument by \cite{FK20,FaKo20b} (cf.~\eqref{n08}--\eqref{n09} below). 
Define $\xi = (\xi_i)_{i=1}^n$ and let $\xi' = (\xi'_i)_{i=1}^n$ be an independent copy of $\xi$. Also, let $I$ be a random index uniformly chosen from $\{1,\dots, n\}$ and independent of $\xi$ and $\xi'$.
In addition, define $Y_i:=\xi_i'-\xi_i$ and $W':=W+Y_I$. It is then easy to verify that $(W, W')$ has the same distribution as $(W', W)$ (exchangeability) and 
\ben{\label{n02}
\Ep[W'-W\mid W']=\frac{W'}{n},\quad \Ep[W'-W\mid W] = -\frac{W}{n}.
} 
Therefore, denoting $D:=W'-W$, we have
\begin{align}
\Ep\left\langle \frac{W}{\sqrt{1-s}}, \nabla h_s(\sqrt{1-s} W)]  \right\rangle \nonumber
& = \Ep \left\langle \frac{W'}{\sqrt{1-s}}, \nabla h_s(\sqrt{1-s} W')\right\rangle\nonumber\\
& = \Ep \left\langle \frac{nD}{\sqrt{1-s}}, \nabla h_s(\sqrt{1-s} W')\right\rangle.\label{eq: exch initial}
\end{align}
Express the right-hand side of this chain of identities, using Taylor's expansion around $W$ with exact integral remainder,  as:
\begin{multline}
\Ep \left\langle \frac{nD}{\sqrt{1-s}}, \nabla  h_s(\sqrt{1-s} W)\right\rangle + \Ep \langle  nDD^{T}, \Hess h_s(\sqrt{1-s} W) \rangle\\
 + n\sum_{j,k,l=1}^d \sqrt{1-s} \Ep\Big[U D_jD_kD_l\partial_{jkl} h_s\Big(\sqrt{1-s}(W+ (1-U)D)\Big)\Big],\label{eq: exch final}
\end{multline}
where $U$ is a uniform random variable on $[0,1]$ independent of everything else, and note also that by \eqref{n02},
\begin{equation}\label{eq: exch ura}
\Ep\left\langle \frac{nD}{\sqrt{1-s}}, \nabla h_s(\sqrt{1-s} W )  \right\rangle = - \Ep\left\langle \frac{W}{\sqrt{1-s}}, \nabla h_s(\sqrt{1-s} W)  \right\rangle.
\end{equation}
Therefore, substituting \eqref{eq: exch final} and \eqref{eq: exch ura} into \eqref{eq: exch initial} and rearranging terms, we obtain
\begin{align*}
&\Ep\left\langle \frac{W}{\sqrt{1-s}}, \nabla h_s(\sqrt{1-s} W)  \right\rangle
=\frac{n}{2}\Ep \langle  DD^{T}, \Hess h_s(\sqrt{1-s} W) \rangle \\
&\qquad\qquad + \frac{n}{2}\sum_{j,k,l=1}^d \sqrt{1-s} \Ep\Big[ U D_jD_kD_l\partial_{jkl} h_s\Big(\sqrt{1-s}(W +  (1-U)D)\Big)\Big],
\end{align*}
and so, by \eqref{eq: annoying lemma gone},
\begin{equation}\label{eq: intermediate step nice}
\Ep T_th(W) = \Ep T_t m^{A,\phi}(W)  = -\frac{1}{2} \int_{t}^1 (R_1(s) - R_2(s)) ds ,
\end{equation}
where
\ba{
R_1(s) &:=\sum_{j,k=1}^d \Ep\left[\Ep\left[\Sigma_{jk}-\frac{n}{2}D_{j}D_{k}\mid\xi\right]\partial_{jk} h_s(\sqrt{1-s}W )\right],\\
R_2(s) &:=\frac{n}{2}\sum_{j,k,l=1}^d \sqrt{1-s}  \Ep \Big[U D_jD_kD_l \partial_{jkl} h_s\Big(\sqrt{1-s}(W  +  (1-U)D)\Big)\Big].
}
We bound $R_1(s)$ and $R_2(s)$ in turn. Regarding $R_1(s)$, we have by Lemma \ref{lemma:fk2.2} that
\ba{
\sup_{w\in\mathbb{R}^d}\sum_{j,k=1}^d\left|\partial_{jk}h_s(w)\right|
\lesssim \frac{\log d}{\sigma^2_* s}.
}
Hence,
\ba{
\int_{t}^1 |R_1(s)| ds \lesssim \frac{(\log d)|\log t|}{\sigma_*^2}\Ep \left \| \Sigma -  \frac{n}{2} \Ep[DD^T\mid \xi] \right \|_\infty.
}
Here, recalling that $D = W'-W = Y_I$, one can deduce %(cf.~Eq.(22) of \cite{CCK14})
\ba{%\label{v-eq}
\Sigma -  \frac{n}{2} \Ep[DD^T\mid \xi] =
(\Sigma-\Sigma_W)-\frac{1}{2}\sum_{i=1}^n(\xi_{i}\xi_{i}^T-\Ep[\xi_{i}\xi_{i}^T]),
}
and so
$$
\Ep \left \| \Sigma -  \frac{n}{2} \Ep[DD^T\mid \xi] \right \|_\infty
 \leq\|\Sigma-\Sigma_W\|_\infty+\frac{1}{2}  \Ep \left \| \sum_{i=1}^n(\xi_{i}\xi_{i}^T-\Ep[\xi_{i}\xi_{i}^T]) \right \|_\infty.
 $$
Also, by Lemma \ref{lem: maximal}, the second term is bounded, up to an absolute
constant, by 
\begin{multline}
\max_{1\leq j,k\leq d}\sqrt{\sum_{i=1}^n\Ep[\xi_{ij}^2\xi_{ik}^2]}\sqrt{\log d}
+\sqrt{\Ep\left[\max_{1\leq i\leq n}\max_{1\leq j,k\leq d}\xi_{ij}^2\xi_{ik}^2\right]}\log d\\
\lesssim \max_{1\leq j\leq d}\sqrt{\sum_{i=1}^n\Ep[\xi_{ij}^4]}\sqrt{\log d}
+\delta^2\log d.\label{eq: maximal r1}
\end{multline}
Consequently, we obtain
\ben{\label{r1}
 \frac{1}{2} \int_{t}^1 |R_1(s) |ds \lesssim |\log t|\times \left(
\Delta_0
+\sqrt{\Delta_1\log d}
+\frac{(\delta\log d)^2}{\sigma_*^2}
\right).
}
Next, we bound $R_2(s)$. We have
\begin{align}
&\Ep\Big[UD_jD_kD_l   \partial_{jkl}h_s\Big(\sqrt{1-s}(W+ (1-U) D)\Big)\Big]\nonumber\\
&\qquad =-\Ep\Big[U D_jD_kD_l \partial_{jkl}h_s\Big(\sqrt{1-s}(W'- (1-U)D)\Big)\Big]\nonumber\\
&\qquad =-\Ep\Big[U D_jD_kD_l \partial_{jkl}h_s\Big(\sqrt{1-s}(W+UD)\Big)\Big],\label{n08}
\end{align}
where the first equality holds by exchangeability and the second by $W'=W+D$. Hence, using Taylor's expansion one more time, we obtain
\begin{align}
\frac{|R_2(s)|}{\sqrt{1-s}} &= \bigg|\frac{n}{4}\sum_{j,k,l=1}^d\Ep\Big[U D_jD_kD_l \partial_{jkl}h_s\Big(\sqrt{1-s}(W+(1-U)D)\Big)\Big]\nonumber\\
&\qquad - \frac{n}{4}\sum_{j,k,l=1}^d\Ep\Big[UD_jD_kD_l\partial_{jkl}h_s\Big(\sqrt{1-s}(W+UD)\Big)\Big]\bigg|\nonumber\\
&\leq\frac{n\sqrt{1-s}}{4} \sum_{j,k,l,m=1}^d\Ep\Big[\Big|D_jD_kD_lD_m \partial_{jklm}h_s\Big(\sqrt{1-s}(W+UD+U'(1-2U)D)\Big)\Big|\Big]\nonumber\\
&=\frac{\sqrt{1-s}}{4} \sum_{i=1}^n\sum_{j,k,l,m=1}^d\Ep\Big[\Big|Y_{ij}Y_{ik}Y_{il}Y_{im}\partial_{jklm}h_s\Big(\sqrt{1-s}(W^{(i)}+\tilde \xi_i)\Big)\Big|\Big],
\label{n09}
\end{align}
where $U'$ is a uniform random variable on $[0,1]$ independent of everything else, $W^{(i)}:=W-\xi_i$, and $\tilde \xi_i:=\xi_i+UY_i+U'(1-2U)Y_i$. Here, note that $|U+U'(1-2U)|\leq U\vee(1-U)\leq1$ and thus $\|\tilde \xi_i\|_\infty\leq\|\xi_i\|_\infty+\|Y_i\|_{\infty}\leq3\delta$. 
Therefore, given that $\|\xi_i\|_\infty\leq\delta$ and that $Y_i$ is independent of $W^{(i)}$, we have
that $4 |R_2(s)|/(1-s)$ is bounded by:
\ba{
&   \sum_{i=1}^n\sum_{j,k,l,m=1}^d \Ep\left[|Y_{ij}Y_{ik}Y_{il}Y_{im}|\sup_{\|y\|_\infty\leq 3\delta}\Big|\partial_{jklm}h_s\Big(\sqrt{1-s}(W^{(i)}+y)\Big)\Big|\right]\\
&\qquad =  \sum_{i=1}^n\sum_{j,k,l,m=1}^d \Ep[|Y_{ij}Y_{ik}Y_{il}Y_{im}|]\Ep\left[\sup_{\|y\|_\infty\leq3\delta}\Big|\partial_{jklm}h_s\Big(\sqrt{1-s}(W^{(i)}+y)\Big)\Big|\right]\\
&\qquad \leq   \sum_{i=1}^n\sum_{j,k,l,m=1}^d \Ep[|Y_{ij}Y_{ik}Y_{il}Y_{im}|]\Ep\left[\sup_{\|y\|_\infty\leq4\delta}\Big|\partial_{jklm}h_s\Big(\sqrt{1-s}(W+y)\Big)\Big|\right]\\
&\qquad =  \sum_{j,k,l,m=1}^d \left(\sum_{i=1}^n\Ep[|Y_{ij}Y_{ik}Y_{il}Y_{im}|]\right)\Ep\left[\sup_{\|y\|_\infty\leq4\delta}\Big|\partial_{jklm}h_s\Big(\sqrt{1-s}(W+y)\Big)\Big|\right]\\
&\qquad \lesssim  \max_{1\leq j\leq d}\sum_{i=1}^n\Ep\xi_{ij}^4 \Ep\left[\sum_{j,k,l,m=1}^d\sup_{\|y\|_\infty\leq4\delta}\Big|\partial_{jklm}h_s\Big(\sqrt{1-s}(W+y)\Big)\Big|\right].
}
Further, let
\begin{equation}\label{eq: kappa definition}
\eta:= 4/\sqrt{\log d},\quad \kappa:=\sqrt{16\log d-2\log (1-\sqrt{1-t})}+\eta,
\end{equation}
so that for any $t \in (0,1)$,
\begin{equation}\label{kappa bound}
e^{-(\kappa-\eta)^2/4}=\frac{\sqrt{1-\sqrt{1-t}}}{d^4}\leq\frac{\sqrt{t}}{d^4}.
\end{equation}
Also, for any $s\in[t,1]$,
\begin{equation}\label{eq: eta bound}
4\delta\sqrt{1-s}
\leq4\delta
\leq\eta\sigma_*\sqrt{s},
\end{equation}
where the second inequality follows from \eqref{t}. Then using \eqref{eq: eta bound} and denoting
\begin{equation}\label{eq: epsilon s}
\eps_s:=2\sqrt{s}\kappa + 1/\phi,
\end{equation}
we have
\begin{align}
&\sum_{j,k,l,m=1}^d\sup_{\|y\|_\infty\leq4\delta}\Big|\partial_{jklm}h_s\Big(\sqrt{1-s}(W+y)\Big)\Big|\nonumber\\
&\qquad \qquad\lesssim \frac{\phi(\log d)^{3/2}1_{\{\sqrt{1-s}W\in (\partial A)^{\eps_s}\}}}{(\sigma_*\sqrt{s})^3}+\frac{\sqrt t}{d^4}\frac{d^4}{(\sigma_*\sqrt{s})^4},\label{eq: smooth appl}
\end{align}
where the first term on the right-hand side appears from bounding the left-hand side by Lemma \ref{lem: smoothing inequality mixed 1} and the second term appears from bounding the left-hand side by Lemma \ref{lemma:vanish} and using \eqref{kappa bound} (here, for any $x>0$, we write $(\partial A)^x :=  A^{x}\setminus A^{-x}$). Hence,
\begin{equation}
 |R_2(s)| \lesssim (1-s)\frac{\Delta_1\sigma_*^4}{(\log d)^2}  \Ep \bigg[\frac{\phi(\log d)^{3/2}1_{\{\sqrt{1-s}W\in (\partial A)^{\eps_s}\}}}{(\sigma_*\sqrt{s})^3} +\frac{\sqrt{t}}{(\sigma_*\sqrt{s})^4} \bigg].\label{r2-1}
\end{equation}
Next,
\begin{align}
\Pr(\sqrt{1-s}W\in(\partial A)^{\varepsilon_s}) &\leq \Pr(\sqrt{1-s}Z\in(\partial A)^{\varepsilon_s}) + 2\varrho'\lesssim \frac{\eps_s\sqrt{\log d}}{\sqrt{1-s}}+\varrho'\label{eq: ququ}
\end{align}
by the definition of $\varrho'$ and Lemma \ref{lem: anticoncentration}, using that $\Sigma$
has unit diagonal entries.

Inserting this bound into \eqref{r2-1}, we deduce 
$$
|R_2(s) |\lesssim \frac{\Delta_1\sigma_*^4}{(\log d)^2}\left(\frac{\phi(\log d)^{3/2}}{(\sigma_*\sqrt s)^3}(\eps_s\sqrt{\log d} +\varrho') + \frac{\sqrt t}{(\sigma_*\sqrt s)^4}\right).
$$
Thus, using \eqref{eq: epsilon s}, we have
\begin{align}
\int_t^1 |R_{2}(s)| ds & \lesssim \frac{\Delta_1\sigma_*^4}{(\log d)^2} \left(\frac{\kappa\phi(\log d)^{2}|\log t|}{\sigma_*^3} + \frac{(\log d)^2}{\sigma_*^3\sqrt t} +\frac{\varrho'\phi(\log d)^{3/2}}{\sigma_*^3\sqrt t} + \frac{1}{\sigma_*^4\sqrt t}\right)\nonumber\\
& \lesssim \Delta_1\left(\kappa\phi|\log t| + \frac{1}{\sqrt t} +\frac{\varrho'\phi}{\sqrt{t\log d}}\right)\label{eq: r2 bound final}
\end{align}
since $\sigma_*\leq 1$ (recall that all diagonal entries of $\Sigma$ are equal to one).
From \eqref{3}, \eqref{eq: some reduction}, \eqref{eq: intermediate step nice}, %\eqref{n18}, 
\eqref{r1} and \eqref{eq: r2 bound final}, we obtain
\bmn{\label{somewhere}
\varrho'
\leq c\left[
|\log t|\left(
\Delta_0
+\sqrt{\Delta_1\log d}
+\frac{(\delta\log d)^2}{\sigma_*^2}
\right)\right.\\
+\left.\Delta_1\left(\kappa\phi|\log t| + \frac{1}{\sqrt t} +\frac{\varrho'\phi}{\sqrt{t\log d}}\right)
+ \sqrt{t} \log d + \frac{\sqrt{\log d}}{\phi}
\right], 
}
where $c>0$ is a universal constant (recall that $1/\sqrt{1-t} \leq 2$). Now we are ready to specify the value of $\phi$:
\begin{equation}\label{t0}
\phi := \frac{1}{2c\sqrt{\Delta_1|\log t|}}.
\end{equation}
Combining this choice with the choice of $t$ in \eqref{t} yields
\begin{equation}\label{eq: rho half}
c\Delta_1 \frac{\varrho'\phi}{\sqrt{t\log d}}\leq \frac{\varrho'}{2}.
\end{equation}
Also, since $t\leq 1/3$ by \eqref{wlog}, $\kappa$ defined in \eqref{eq: kappa definition} satisfies
\begin{equation}\label{eq: kappa final}
\kappa \lesssim \sqrt{\log d} + \sqrt{|\log t|}.
\end{equation}
The asserted claim now follows by substituting \eqref{t}, \eqref{t0}, \eqref{eq: rho half}, and \eqref{eq: kappa final} into \eqref{somewhere}.
\end{proof}

% \subsection{Proof of thm 2.2}
\begin{proof}[Proof of Theorem \ref{thm: general unbounded}]
The proof is a modification of the proof of Theorem \ref{t1}. Here, we describe the changes, keeping all unmentioned notations the same as those in the proof of Theorem \ref{t1}. In a nutshell, we only need to change the values of $t$ and $\phi$ and use a truncation argument in the bound for $R_2(s)$. 

First, we now set
\begin{equation}\label{eq: t new}
t:=\frac{\Delta_1}{\log d} + \frac{M(\psi)\log d}{n\sigma_*^4} + \frac{\psi^2\log d}{n\sigma_*^2}
\end{equation}
instead of using $t$ in \eqref{t}. As in the proof of Theorem \ref{t1}, it is without loss of generality to assume here that $t\leq 1/3$; compare with \eqref{wlog} and \eqref{t}. Moreover, since all diagonal entries of $\Sigma$ are equal to one, it follows that $\sigma_*\leq 1$. 
%and so we can assume that $n^{-1}\sum_{i=1}^n \Ep X_{ij}^2\geq 1/2$ for all $j=1,\dots,d$ since otherwise $\Delta_0\geq 1/2$ and the asserted claim of the theorem follows trivially. 
%The latter, however, implies via Jensen's inequality that $\Delta_1/\log d\geq 1/(4n)$, and so $|\log t|\lesssim \log n$. 
In addition, for a while, we assume $n^{-1}\sum_{i=1}^n \Ep X_{ij}^2\geq 1/2$ for all $j=1,\dots,d$. This implies via Jensen's inequality that $\Delta_1/\log d\geq 1/(4n)$, and so $|\log t|\lesssim \log n$. 

Next, note that \eqref{3}, \eqref{eq: some reduction}, and \eqref{eq: intermediate step nice} hold under our current assumptions by the same arguments as those in the proof of Theorem \ref{t1}. Thus, we only need to bound
$$
\int_t^1|R_1(s)|ds\quad\text{and}\quad\int_t^1|R_2(s)|ds.
$$
Regarding the former, we proceed as in the proof of Theorem \ref{t1} but we change the second line in \eqref{eq: maximal r1} by
$$
\max_{1\leq j\leq d}\sqrt{\sum_{i=1}^n\Ep[\xi_{ij}^4]}\sqrt{\log d}
+\frac{\mathcal M^2\log d}{n}
$$
so that similarly to \eqref{r1}, we obtain
\ben{\label{r1 new}
 \frac{1}{2} \int_{t}^1 |R_1(s) |ds \lesssim |\log t|\times \left(
\Delta_0
+\sqrt{\Delta_1\log d}
+\frac{(\mathcal M\log d)^2}{n\sigma_*^2}
\right).
}
%where we used $\Delta_0\leq\Lambda_0$.

Further, by \eqref{n09}, we have
$$
R_2(s)\lesssim (1-s)\sum_{i=1}^n\sum_{j,k,l,m=1}^d\Ep\Big[\Big|Y_{ij}Y_{ik}Y_{il}Y_{im}\partial_{jklm}h_s\Big(\sqrt{1-s}(W^{(i)}+\tilde \xi_i)\Big)\Big|\Big] \leq \mathcal I_{1}(s)+\mathcal I_{2}(s),
$$
where
$$
\mathcal I_{1}(s):=(1-s) \sum_{i=1}^n\sum_{j,k,l,m=1}^d\Ep\Big[\Big|\varsigma_iY_{ij}Y_{ik}Y_{il}Y_{im}\partial_{jklm}h_s\Big(\sqrt{1-s}(W^{(i)}+\tilde \xi_i)\Big)\Big|\Big],
$$
$$
\mathcal I_{2}(s):= (1-s)\sum_{i=1}^n\sum_{j,k,l,m=1}^d\Ep\Big[\Big|(1-\varsigma_i)Y_{ij}Y_{ik}Y_{il}Y_{im}\partial_{jklm}h_s\Big(\sqrt{1-s}(W^{(i)}+\tilde \xi_i)\Big)\Big|\Big],
$$
and
$\varsigma_i := 1\{\|\xi_i\|_{\infty}\vee\|\xi_i'\|_{\infty} \leq 2\psi/\sqrt n\}$ for all $i =1,\dots,n$. 
%At this step, $\int_t^1\mathcal I_{1}(s)ds$ is bounded in the same way as $\int_t^1|R_2(s)|ds$ in the proof of Theorem \ref{t1}; namely,
%$$
%\int_t^1\mathcal I_1(s)ds\lesssim \Delta_1\left(\kappa\phi|\log t|+\frac{1}{\sqrt t} + \frac{\varrho'\phi}{\sqrt t\log d}\right);
%$$
%compare with \eqref{eq: r2 bound final}. 
We first focus on $\mathcal I_1(s)$. 
Given that $\|\tilde \xi_i\|_\infty\leq\|\xi_i\|_\infty+\|Y_i\|_{\infty}\leq2\|\xi_i\|_\infty+\|\xi_i'\|_\infty$ and that $Y_i$ is independent of $W^{(i)}$, we have
that $|\mathcal I_1(s)|$ is bounded by:
\ba{
&   (1-s)\sum_{i=1}^n\sum_{j,k,l,m=1}^d \Ep\left[\varsigma_i|Y_{ij}Y_{ik}Y_{il}Y_{im}|\sup_{\|y\|_\infty\leq 6\psi/\sqrt n}\Big|\partial_{jklm}h_s\Big(\sqrt{1-s}(W^{(i)}+y)\Big)\Big|\right]\\
&\qquad =  (1-s)\sum_{i=1}^n\sum_{j,k,l,m=1}^d \Ep[\varsigma_i|Y_{ij}Y_{ik}Y_{il}Y_{im}|]\Ep\left[\sup_{\|y\|_\infty\leq6\psi/\sqrt n}\Big|\partial_{jklm}h_s\Big(\sqrt{1-s}(W^{(i)}+y)\Big)\Big|\right]\\
&\qquad \leq  \mathcal{I}_{11}(s)+\mathcal{I}_{12}(s),
}
where
\ba{
\mathcal{I}_{11}(s)&:=(1-s)\sum_{i=1}^n\sum_{j,k,l,m=1}^d \Ep[|Y_{ij}Y_{ik}Y_{il}Y_{im}|]\Ep\left[\varsigma_i\sup_{\|y\|_\infty\leq6\psi/\sqrt n}\Big|\partial_{jklm}h_s\Big(\sqrt{1-s}(W^{(i)}+y)\Big)\Big|\right],\\
\mathcal{I}_{12}(s)&:=(1-s)\sum_{i=1}^n\sum_{j,k,l,m=1}^d \Ep[| Y_{ij}Y_{ik}Y_{il}Y_{im}|]\Ep\left[(1-\varsigma_i)\sup_{\|y\|_\infty\leq6\psi/\sqrt n}\Big|\partial_{jklm}h_s\Big(\sqrt{1-s}(W^{(i)}+y)\Big)\Big|\right].
}
At this step, $\int_t^1\mathcal I_{11}(s)ds$ is bounded in the same way as $\int_t^1|R_2(s)|ds$ in the proof of Theorem \ref{t1}; namely,
$$
\int_t^1\mathcal I_{11}(s)ds\lesssim \Delta_1\left(\kappa\phi|\log t|+\frac{1}{\sqrt t} + \frac{\varrho'\phi}{\sqrt t\log d}\right);
$$
compare with \eqref{eq: r2 bound final}. 
Meanwhile, using the independence between $W^{(i)}$ and $Y_i$, we obtain
\ba{
\mathcal{I}_{12}(s)&\leq(1-s)\sum_{i=1}^n\sum_{j,k,l,m=1}^d\Ep\|Y_{i}\|_\infty^4\Ep[1-\varsigma_i] \Ep\left[\sup_{\|y\|_\infty\leq6\psi/\sqrt n}\Big|\partial_{jklm}h_s\Big(\sqrt{1-s}(W^{(i)}+y)\Big)\Big|\right]\\
&\lesssim(1-s)\sum_{i=1}^n\Ep(1-\varsigma_i)(\|\xi_{i}\|_\infty\vee\|\xi_i'\|_\infty)^4\Ep\left[\sum_{j,k,l,m=1}^d\sup_{\|y\|_\infty\leq6\psi/\sqrt n}\Big|\partial_{jklm}h_s\Big(\sqrt{1-s}(W^{(i)}+y)\Big)\Big|\right],
}
where the last inequality follows from Chebyshev's association inequality; see Theorem 2.14 in \cite{BLM13}. Then, applying \eqref{eq: smooth appl} with replacing $W$ and $4\delta$ by $W^{(i)}$ and $6\psi/\sqrt n$ respectively, we deduce
\ba{
\mathcal{I}_{12}(s)
&\lesssim(1-s)\sum_{i=1}^n\Ep(1-\varsigma_i)(\|\xi_{i}\|_\infty\vee\|\xi_i'\|_\infty)^4 \Ep\left[\frac{\phi(\log d)^{3/2}1_{\{\sqrt{1-s}W^{(i)}\in (\partial A)^{\eps_s}\}}}{(\sigma_*\sqrt{s})^3} +\frac{\sqrt{t}}{(\sigma_*\sqrt{s})^4} \right].
}
Now, as in the proof of \eqref{eq: r2 bound final}, we obtain
\ba{
\mathcal{I}_{12}(s) \lesssim\sum_{i=1}^n\Ep(1-\varsigma_i)(\|\xi_{i}\|_\infty\vee\|\xi_i'\|_\infty)^4 \left(\frac{\phi(\log d)^{3/2}}{(\sigma_*\sqrt s)^3}(\eps_s\sqrt{\log d} +\varrho^{\{i\}}) + \frac{\sqrt t}{(\sigma_*\sqrt s)^4}\right),
}
where, for any subset $\mathcal I\subset\{1,\dots,n\}$, 
$$
\varrho^{\mathcal I}:=\sup_{A\in\mathcal R}\left|\Pr\left(\frac{1}{\sqrt n}\sum_{i\in\mathcal I^c}X_i\in A\right) - \Pr(Z\in A)\right|,\quad Z\sim N(0,\Sigma),\quad i=1,\dots,n.
$$
In addition, for all $i=1,\dots,n$,
\begin{align}
\Ep(1-\varsigma_i)(\|\xi_{i}\|_\infty\vee\|\xi_i'\|_\infty)^4
& \leq \Ep(1-\varsigma_i)(\|\xi_i\|_{\infty}^4 + \|\xi_i'\|_{\infty}^4)
\nonumber\\
&\lesssim \Ep(1\{\|\xi_i\|_{\infty}>\psi/\sqrt n\} + 1\{\|\xi_i'\|_{\infty}>\psi/\sqrt n\})(\|\xi_i\|_{\infty}^4 + \|\xi_i'\|_{\infty}^4)
\nonumber\\
&\lesssim \Ep1\{\|\xi_i\|_{\infty}>\psi/\sqrt n\}\|\xi_i\|_{\infty}^4
\leq M(\psi)/n^2,\label{y-tail-est}
\end{align}
where the penultimate inequality holds by Chebyshev's association inequality. Thus,
\ba{
\mathcal{I}_{12}(s)&\lesssim \frac{M(\psi)}{n} \left(\frac{\phi(\log d)^{3/2}}{(\sigma_*\sqrt s)^3}((\sqrt s\kappa+1/\phi)\sqrt{\log d} +\overline\varrho) + \frac{\sqrt t}{(\sigma_*\sqrt s)^4}\right)\\
&\lesssim \frac{M(\psi)}{n} \left(\frac{\kappa\phi(\log d)^2}{\sigma_*^3s} + \frac{(\log d)^2}{\sigma_*^3s^{3/2}} + \frac{\overline\varrho\phi(\log d)^{3/2}}{\sigma_*^3s^{3/2}}+ \frac{\sqrt t}{\sigma_*^4  s^2}\right),
}
where $\overline\varrho := n^{-1}\sum_{i=1}^n\varrho^{\{i\}}$. 
Therefore, we conclude
$$
\int_t^1\mathcal I_{12}(s)ds \lesssim \frac{M(\psi)(\log d)^2}{n\sigma_*^4}\left(\kappa\phi|\log t|+\frac{1}{\sqrt t} + \frac{\overline\varrho\phi}{\sqrt{t\log d}}\right).
$$
Turning to $\mathcal I_2(s)$, we have by the law of iterated expectations,
$$
\mathcal I_2(s) \lesssim (1-s)\sum_{i=1}^n\Ep\left[ (1-\varsigma_i)\|\xi_i-\xi_i'\|_{\infty}^4\sum_{j,k,l,m=1}^d\Ep\Big[ \Big| \partial_{jklm}h_s\Big(\sqrt{1-s}(W^{(i)}+\tilde\xi_i)\Big) \Big| \mid \tilde\xi_i \Big] \right].
$$
Here, we bound the internal sum as in \eqref{eq: kappa definition} with $\eta=0$, \eqref{eq: epsilon s}, \eqref{eq: smooth appl}, and  \eqref{eq: ququ}; namely, by Lemmas \ref{lem: smoothing inequality mixed 1} and \ref{lemma:vanish},
\begin{align*}
&(1-s)\sum_{j,k,l,m=1}^d\Ep \Big[ \Big| \partial_{jklm}h_s\Big(\sqrt{1-s}(W^{(i)}+\tilde\xi_i)\Big) \Big| \mid \tilde\xi_i \Big]\\
&\qquad \qquad \lesssim \frac{\phi(\log d)^{3/2}}{(\sigma_*\sqrt s)^3}\Big((\sqrt s\kappa + 1/\phi)\sqrt{\log d}+\varrho^{(i)}\Big) + \frac{\sqrt t}{(\sigma_*\sqrt s)^4}\\
&\qquad\qquad\lesssim \frac{\kappa\phi(\log d)^2}{\sigma_*^3s} + \frac{(\log d)^2}{\sigma_*^3s^{3/2}} + \frac{\varrho^{(i)}\phi(\log d)^{3/2}}{\sigma_*^3s^{3/2}}+ \frac{\sqrt t}{\sigma_*^4  s^2}.
\end{align*} 
%where
%$$
%\varrho^{(i)}:=\sup_{r\in\mathbb R^d}\left|\Pr\left(\sqrt{\frac{n}{n-1}}W^{(i)}\leq r\right) - \Pr(Z\leq r)\right|,\quad Z\sim N(0,\Sigma),\quad i=1,\dots,n.
%$$
%In addition, for all $i=1,\dots,n$,
%\begin{align*}
%\Ep(1-\varsigma_i)\|\xi_i-\xi_i'\|_{\infty}^4
%& \lesssim \Ep(1-\varsigma_i)(\|\xi_i\|_{\infty}^4 + \|\xi_i'\|_{\infty}^4)\\
%&\lesssim \Ep(1\{\|\xi_i\|_{\infty}>\psi/\sqrt n\} + 1\{\|\xi_i'\|_{\infty}>\psi/\sqrt n\})(\|\xi_i\|_{\infty}^4 + \|\xi_i'\|_{\infty}^4)\\
%&\lesssim \Ep1\{\|\xi_i\|_{\infty}>\psi/\sqrt n\}\|\xi_i\|_{\infty}^4 \leq M(\psi)/n^2,
%\end{align*}
%where the penultimate inequality holds by Chebyshev's association inequality; see Theorem 2.14 in \cite{BLM13}.
Thus, by \eqref{y-tail-est},
$$
\int_t^1\mathcal I_2(s)ds \lesssim \frac{M(\psi)(\log d)^2}{n\sigma_*^4}\left(\kappa\phi|\log t|+\frac{1}{\sqrt t} + \frac{\overline\varrho\phi}{\sqrt{t\log d}}\right).
$$
%where $\overline\varrho := n^{-1}\sum_{i=1}^n\varrho^{(i)}$.

Combining all terms and recalling that $|\log t|\lesssim \log n$, we now have
\begin{multline}\label{somewhere new}
\varrho'
\leq c\bigg[
(\log n)\left(
\Delta_0+\sqrt{\Delta_1\log d}+\frac{(\mathcal M\log d)^2}{n\sigma_*^2}
\right)\\
+\left(\Delta_1 + \frac{M(\psi)(\log d)^2}{n\sigma_*^4}\right)\left(\kappa\phi\log n + \frac{1}{\sqrt t} +\frac{(\varrho'+\overline\varrho)\phi}{\sqrt{t\log d}}\right)
+ \sqrt{t} \log d + \frac{\sqrt{\log d}}{\phi}
\bigg], 
\end{multline}
where $c>0$ is a universal constant; compare with \eqref{somewhere}. Here, we set
$$
\phi := \frac{1}{2ec\sqrt{\left(\Delta_1 + \frac{M(\psi)(\log d)^2}{n\sigma_*^4}\right)\log n}}
$$
and use $t$ defined in \eqref{eq: t new} to obtain
\begin{multline}
\varrho' \leq C'\bigg[(\log n)\left(
\Delta_0+\sqrt{\Delta_1\log d}+\frac{(\mathcal M\log d)^2}{n\sigma_*^2}
\right) \\
+ \sqrt{\left(\Delta_1 + \frac{M(\psi)(\log d)^2}{n\sigma_*^4}\right)(\log n)\log(dn)} + \frac{\psi(\log d)^{3/2}}{\sigma_*\sqrt n}\bigg] + \frac{\varrho' + \overline\varrho}{2e}
\end{multline}
since $\kappa\lesssim \sqrt{\log(dn)}$ by \eqref{eq: kappa final}, where $C'$ is a universal constant. Hence, rearranging the terms and substituting the definition of $\Lambda_1$,
$$
\varrho' \leq 2C'\left[(\log n)\left(\Delta_0 + \sqrt{\Delta_1\log d} + \frac{(\mathcal M \log d)^2}{n\sigma_*^2}\right) + \sqrt{\frac{\Lambda_1M(\psi)}{n\sigma_*^4}} + \frac{\psi(\log d)^{3/2}}{\sigma_*\sqrt n}\right] + \frac{\overline\varrho}{e}.
$$
Now we note that this bound is valid even when $n^{-1}\sum_{i=1}^n\Ep X_{ij}^2<1/2$ for some $j$; in fact, we have $\Delta_0\geq(\log d/\sigma_*^2)|\Sigma_{jj}-n^{-1}\sum_{i=1}^n\Ep X_{ij}^2|>1/2$ in this case because $\Sigma_{jj}=1$ and $\sigma_*\leq1$. 
As a result, this bound holds without the restriction $n^{-1}\sum_{i=1}^n\Ep X_{ij}^2\geq1/2$ for all $j$ imposed at the beginning of the proof. 
We now iterate this bound to obtain inequalities for each $\varrho^{\{i\}}$ and repeat the procedure $[\log n] + 1$ times, dropping one observation at a time. 
Here, given a proper subset $\mathcal I\subset\{1,\dots,n\}$, we apply this bound to $\varrho^{\mathcal I}$ with replacing $(X_{i})_{i=1}^n$ and $\psi$ by $(\sqrt{\frac{n-|\mathcal I|}{n}}X_{i})_{i\in\mathcal I^c}$ and $\sqrt{\frac{n-|\mathcal I|}{n}}\psi$, respectively. Then, denoting the corresponding $\Delta_0,\Delta_1,\mathcal{M},M(\psi)$ by $\Delta_0^{\mathcal I},\Delta_1^{\mathcal I},\mathcal{M}^{\mathcal I},M^{\mathcal I}(\sqrt{\frac{n-|\mathcal I|}{n}}\psi)$ respectively, we have
\begin{align*}
\Delta_0^{\mathcal I}&=\frac{\log d}{\sigma_*^2}\max_{1\leq j,k\leq d}\bigg|\Sigma_{jk}-\frac{1}{n-|\mathcal I|}\sum_{i\in\mathcal I^c}\frac{n-|\mathcal I|}{n}\Ep X_{ij}X_{ik}\bigg|
\leq\Delta_0+\frac{\log d}{n\sigma_*^2}\max_{1\leq j\leq d}\sum_{i\in\mathcal I}\Ep X_{ij}^2\\
&\leq\Delta_0+\frac{\log d}{n\sigma_*^2}\max_{1\leq j\leq d}\sqrt{|\mathcal I|\sum_{i\in\mathcal I}\Ep X_{ij}^4}
\leq\Delta_0+\sqrt{|\mathcal I|\Delta_1\log d},\\
\Delta_1^{\mathcal I}&=\frac{(\log d)^2}{(n-|\mathcal I|)^2\sigma_*^4}\max_{1\leq j\leq d}\sum_{i\in\mathcal I^c}\frac{(n-|\mathcal I|)^2}{n^2}\Ep X_{ij}^4\leq\Delta_1,
\end{align*}
and
\begin{align*}
\frac{(\mathcal{M}^{\mathcal I})^2}{n-|\mathcal I|}&\leq\frac{\mathcal{M}^2}{n},\qquad
\frac{M^{\mathcal I}(\sqrt{\frac{n-|\mathcal I|}{n}}\psi)}{n-|\mathcal I|}\leq\frac{M(\psi)}{n}.
\end{align*}
Note that $\sum_{i=1}^\infty e^{-i}\sqrt{i}<\infty$. 
We thus obtain 
$$
\varrho'\lesssim (\log n)\left(\Delta_0 + \sqrt{\Delta_1\log d} + \frac{(\mathcal M \log d)^2}{n\sigma_*^2}\right) + \sqrt{\frac{\Lambda_1M(\psi)}{n\sigma_*^4}} + \frac{\psi(\log d)^{3/2}}{\sigma_*\sqrt n} + \frac{\widehat\varrho}{e^{\log n}},
$$
where
$$
\widehat\varrho := \frac{1}{|\widehat{\mathcal N}|}\sum_{\mathrm N\in \widehat{\mathcal N}}\sup_{A\in\mathcal R}\left|\Pr\left(\frac{1}{\sqrt{n}}\sum_{i\in\mathrm N} X_i \in A\right) - \Pr(Z\in A)\right|,\quad Z\sim N(0,\Sigma),
$$
and $\widehat{\mathcal N}:=\{\mathrm N\subset\{1,\dots,n\}\colon |\mathrm N|=n-[\log n]-2\}$.
Since $\widehat\varrho\leq 1$ and $\Delta_1/\log d\geq 1/(4n)$, as discussed above, the asserted claim follows.
\end{proof}

\begin{proof}[Proof of Corollary \ref{thm: simple}]
Since 
\[
\sup_{A\in \mathcal{R}} |\Pr(W\in A)- \Pr(Z\in A)|=\sup_{A\in \mathcal{R}} |\Pr(SW\in A)- \Pr(SZ\in A)|
\]
for any $d\times d$ diagonal matrix $S$, we assume, without loss of generality, that $\sigma_j=1$ for all $j=1,\dots,d$. Then $\sigma_{*,W}$ is the square root of the smallest eigenvalue of $\Sigma_W=\Ep WW^T$. To prove the asserted claims, we will apply Theorems \ref{t1} and \ref{thm: general unbounded} with $\Sigma = \Sigma_W$, so that $\sigma_* = \sigma_{*,W}$.

Consider first the case when (E.1) holds. By Jensen's inequality, $\Delta_1 \geq (\log d)^2/(n\sigma_{*,W}^4) \geq (\log d)/n$, and so 
$$
\left|\log\left(\frac{\Delta_1}{\log d} + \frac{\delta^2\log d}{\sigma_*^2}\right)\right|\leq \log n.
$$ 
Also, since (E.1) implies (M),
\begin{equation}\label{eq: m delta bound}
(\log n)\sqrt{\Delta_1\log d} \leq \frac{B_n(\log d)^{3/2}(\log n)}{\sqrt n\sigma_{*,W}^2}.
\end{equation}
Combining these inequalities and using Theorem \ref{t1} with $\delta = B_n / \sqrt n$ and $\Delta_0=0$ gives the asserted claim under condition (E.1).

Next, consider the case when (M) and (E.2) hold. Without loss of generality, we assume that
\begin{equation}\label{eq: as proof}
\frac{B_n(\log d)^{3/2}\log n}{\sqrt n \sigma_{*,W}}\leq 1
\end{equation}
since otherwise the asserted claims are trivial. 
%Further, note that (E.2) implies that $\Ep X_{ij}^2\lesssim B_n^2$ for all $i=1,\dots,n$ and $j=1,\dots,d$, and so
%\begin{align*}
%\Lambda_0
%&\leq \frac{\log d}{\sigma_{*,W}^2}\max_{\mathrm N\in\mathcal N}\left(\left\| \Sigma_W - \frac{1}{n}\sum_{i\in\mathrm N}\Ep X_iX_i' \right\|_{\infty} + \left\|\frac{1}{n}\sum_{i\in\mathrm N}\Ep X_iX_i' - \frac{1}{|\mathrm N|}\sum_{i\in\mathrm N}\Ep X_iX_i'\right\|_{\infty}\right)\\
%&\lesssim \frac{\log d}{\sigma_{*,W}^2}\left(\frac{B_n^2\log n}{n} + \frac{\log n}{n-\log n}\right)
%\lesssim \frac{B_n^2(\log d)(\log n)}{n\sigma_{*,W}^2}.
%\end{align*}
%Thus, by \eqref{eq: as proof},
%\begin{equation}\label{eq: lambda bound proof}
%\Lambda_0\log n \lesssim \frac{B_n^2(\log d)(\log n)^2}{n\sigma_{*,W}^2}
%\lesssim \frac{B_n(\log d)^{3/2}\log n}{\sqrt n\sigma_{*,W}^2}.
%\end{equation}
\eqref{eq: m delta bound} holds by condition (M). In addition, $\mathcal M\lesssim B_n\sqrt{\log(d n)}$ by Lemma 2.2.2 and discussion on page 95 of \cite{VW96}, and so
\begin{equation}\label{eq: third e1}
\frac{(\log n)(\mathcal M\log d)^2}{n\sigma_{*,W}^2}\lesssim \frac{B_n^2(\log d)^2(\log n)\log(dn)}{n\sigma_{*,W}^2}\lesssim \frac{B_n(\log d)^{3/2}\log n}{\sqrt n\sigma_{*,W}^2}
\end{equation}
by \eqref{eq: as proof}. Moreover, since for any $i=1,\dots,n$ and $\psi>0$,
$$
\Ep\|X_i\|_{\infty}^41\{\|X_i\|_{\infty} > \psi\}\leq\left(\Ep[\|X_i\|_{\infty}^8]\Pr(\|X_i\|_{\infty}>\psi)\right)^{1/2}\lesssim B_n^4\sqrt{\Pr(\|X_i\|_{\infty}>\psi)},
$$
setting $\psi = CB_n\sqrt{\log(d n)}$ for a sufficiently large but universal constant $C$, we have
\begin{align}
\sqrt{\frac{\Lambda_1M(\psi)}{n\sigma_{*,W}^4}} + \frac{\psi(\log d)^{3/2}}{\sqrt n\sigma_{*,W}}
&\lesssim \frac{B_n(\log d)^{3/2}\log n}{\sqrt n\sigma_{*,W}^2} + \frac{B_n(\log d)^{3/2}\sqrt{\log(dn)}}{\sqrt n\sigma_{*,W}}\nonumber\\
&\lesssim \frac{B_n(\log d)^{3/2}\log n}{\sqrt n\sigma_{*,W}^2} + \frac{B_n(\log d)^{2}}{\sqrt n\sigma_{*,W}},\label{eq: last two e1}
\end{align}
where the last inequality follows from $\sigma_{*,W}\leq 1$. Combining $\Delta_0=0$, \eqref{eq: m delta bound}, \eqref{eq: third e1}, and \eqref{eq: last two e1} and applying Theorem \ref{thm: general unbounded} gives the asserted claim under conditions (M) and (E.2).

Now consider the case when (M) and (E.3) hold. In this case, 
$$
\mathcal M\leq \left( \Ep\left[ \sum_{i=1}^n \max_{1\leq j\leq d}|X_{i j}|^q \right] \right)^{1/q}\leq n^{1/q}B_n,
$$
and so
\begin{equation}\label{eq: third e2}
\frac{(\log n)(\mathcal M\log d)^2}{n\sigma_{*,W}^2} \leq \frac{B_n^2(\log d)^2\log n}{n^{1-2/q}\sigma_{*,W}^2}.
\end{equation}
Also, \eqref{eq: m delta bound} holds by the same arguments as those in the previous case. In addition, since for any $i=1,\dots,n$ and $\psi>0$,
$$
\Ep\|X_i\|_{\infty}^41\{\|X_i\|_{\infty} > \psi\}\leq \Ep\|X_i\|_{\infty}^q/\psi^{q-4} \leq B_n^q/\psi^{q-4},
$$
setting
$$
\psi=\left(\frac{B_n^q(\log n)\log(dn)}{\sigma_{*,W}^2\log d}\right)^{1/(q-2)},
$$
we obtain
\begin{equation}\label{eq: last e2}
\sqrt{\frac{\Lambda_1M(\psi)}{n\sigma_{*,W}^4}} + \frac{\psi(\log d)^{3/2}}{\sqrt n\sigma_{*,W}}\lesssim\left(\frac{B_n^q(\log d)^{3q/2-4}(\log n)\log(d n)}{n^{q/2-1}\sigma_{*,W}^q}\right)^{\frac{1}{q-2}}.
\end{equation}
Combining $\Delta_0=0$, \eqref{eq: m delta bound}, \eqref{eq: third e2}, and \eqref{eq: last e2} and applying Theorem \ref{thm: general unbounded} gives the asserted claim under conditions (M) and (E.3) and completes the proof of the theorem.
\end{proof}

\begin{proof}[Proof of Proposition \ref{prop:lower-bound}]
For every $n\geq1$, let $(X_{n,ij})_{i,j=1}^\infty$ be an array of i.i.d.~variables such that
\[
\Pr\left(X_{n,ij}=a_n\right)=1-\Pr\left(X_{n,ij}=b_n\right)=p_n,
\]
where
\[
p_n:=\frac{1}{B_n^2},\qquad
a_n:=\sqrt{\frac{1-p_n}{p_n}},\qquad
b_n:=-\sqrt{\frac{p_n}{1-p_n}}.
\]
Since $B_n\geq2$, we have $|X_{n,ij}|\leq \max\{B_n,1/\sqrt3\}\leq B_n$. 
Also, it is straightforward to check that $\Ep[X_{n,ij}]=0$ and $\Ep[X_{n,ij}^2]=1$. 
Therefore, we complete the proof once we show that there is a sequence $(x_n)_{n=1}^\infty$ of real numbers such that
\[
\rho:=\liminf_{n\to\infty}\frac{\sqrt{n}}{B_n\log^{3/2}d}\left|\Pr\left(\max_{1\leq j\leq d}W_{n,j}\leq x_n\right)-\Pr\left(\max_{1\leq j\leq d}Z_{j}\leq x_n\right)\right|>0,
\]
where
\[
W_{n,j}:=\frac{1}{\sqrt n}\sum_{i=1}^nX_{n,ij}.
\]

%We denote by $\varphi_1$ and $\Phi_1$ the density and distribution function of the standard normal distribution, respectively.
For every $n$, we define $x_n\in\mathbb{R}$ as the solution of the equation $\Phi_1(x)^{d}=e^{-1}$, i.e.~$x_n:=\Phi_1^{-1}(e^{-1/d})$. 
Then we have $x_n/\sqrt{2\log d}\to1$ as $n\to\infty$ (cf.~the proof of Proposition 2.1 in \cite{K19}). We also have
\ben{\label{est:normal-tail}
d(1-\Phi_1(x_n))=d(1-e^{-1/d})=1+O(d^{-1})\qquad\text{as }n\to\infty.
}
Now, applying Theorem 1 in \cite{AGG89} with $I=\{1,\dots,d\}$, $B_\alpha=\{\alpha\}$ and $X_\alpha=1_{\{W_{n,\alpha}>x_n\}}$ in their notation, we obtain
\begin{align*}
\left|\Pr\left(\max_{1\leq j\leq d}W_{n,j}\leq x_n\right)-e^{-\lambda_n}\right|
\leq d\Pr\left(W_{n,1}>x_n\right)^2,
\end{align*}
where $\lambda_n:=d\Pr\left(W_{n,1}>x_n\right)$. 
Meanwhile, we have by definition
\[
\Pr\left(\max_{1\leq j\leq d}Z_{j}\leq x_n\right)=\Phi_1(x_n)^d=e^{-1}.
\]
Hence we obtain
\begin{equation}\label{eq:chen-stein}
\left|\Pr\left(\max_{1\leq j\leq d}W_{n,j}\leq x_n\right)-\Pr\left(\max_{1\leq j\leq d}Z_j\leq x_n\right)\right|
\geq|e^{-\lambda_n}-e^{-1}|-\frac{\lambda_n^2}{d}.
\end{equation}
To evaluate $\lambda_n$, we apply Theorem 2.1 in \cite{FLS20} with $m=n$, $X_\alpha=X_{n,\alpha1}$, $\mathcal{I}_i=\{i\}$, $\xi_i=X_{n,i1}/\sqrt n$, $\delta=\delta_n:=B_n/\sqrt n$ and $s=d=1$ in their notation. %\footnote{It seems impossible to apply a more standard moderate deviation result here because such a one requires that the moment generating function of $X_{n,11}$ is bounded on a neighborhood of the origin uniformly in $n$.} 
We have by assumption
\ben{\label{est:delta}
n^2\delta_n^5x_n^2=O\left(\frac{B_n^5\log d}{\sqrt n}\right)=o\left(\frac{B_n\log^{3/2} d}{\sqrt{n}}\right)=o(1).
}
Hence there is a constant $C_0>0$ such that $\sqrt{n^2\delta_n^5}x_n\leq C_0$ for all $n$. 
Therefore, Theorem 2.1 in \cite{FLS20} yields
\begin{equation}\label{eq:mdp}
\left|\frac{\Pr\left( W_{n,1}>x_n\right)}{(1-\Phi_1(x_n))e^{\gamma_n x_n^3/6}}-1\right|\leq Cn^2\delta_n^5(1+x_n^2)\qquad\text{for all }n,
\end{equation}
where $\gamma_n:=\Ep[W_{n,1}^3]$ and $C$ is a constant depending only on $C_0$. 
Since
\[
\gamma_n=\frac{\Ep[X_{n,11}^3]}{\sqrt n}=\frac{1-2p_n}{\sqrt{np_n(1-p_n)}},
\]
we have
\ben{\label{est:skewness}
\liminf_{n\to\infty}\frac{\gamma_nx_n^3}{B_n(\log d)^{3/2}/\sqrt{n}}
\geq\liminf_{n\to\infty}\frac{1}{2B_n\sqrt{p_n}}\left(\frac{x_n}{\sqrt{\log d}}\right)^3
=\sqrt{2}.
}
In particular, $\gamma_nx_n^3=O(B_n(\log d)^{3/2}/\sqrt{n})=o(1)$ by assumption. Combining this estimate with \eqref{est:normal-tail} and \eqref{est:delta}, we deduce from \eqref{eq:mdp}
\begin{align*}
\lambda_n&=\frac{\Pr\left( W_{n,1}>x_n\right)}{(1-\Phi_1(x_n))e^{\gamma_n x_n^3/6}}\cdot d(1-\Phi_1(x_n))e^{\gamma_n x_n^3/6}\\
&=d(1-\Phi_1(x_n))e^{\gamma_n x_n^3/6}+o\left(\frac{B_n\log^{3/2} d}{\sqrt{n}}\right)\\
&=e^{\gamma_n x_n^3/6}+O\left(\frac{1}{d}\right)+o\left(\frac{B_n\log^{3/2} d}{\sqrt{n}}\right).
\end{align*}
Using the Maclaurin expansion of the exponential function and \eqref{est:skewness}, we obtain
\[
\lambda_n=1+\frac{\gamma_n x_n^3}{6}+O\left(\frac{1}{d}\right)+o\left(\frac{B_n\log^{3/2} d}{\sqrt{n}}\right)
\]
and thus
\[
e^{-\lambda_n+1}=1-\frac{\gamma_n x_n^3}{6}+O\left(\frac{1}{d}\right)+o\left(\frac{B_n\log^{3/2} d}{\sqrt{n}}\right).
\]
Note that we particularly have $\lambda_n=O(1)$. Thus, \eqref{eq:chen-stein} yields
\ba{
\rho\geq\liminf_{n\to\infty}\frac{\sqrt{n}}{B_n\log^{3/2}d}|e^{-\lambda_n}-e^{-1}|
=\liminf_{n\to\infty}e^{-1}\frac{\sqrt{n}}{B_n\log^{3/2}d}\frac{\gamma_n x_n^3}{6},
}
where we used the assumption $d^{-1}=o(B_n\log^{3/2}d/\sqrt n)$. Hence we obtain by \eqref{est:skewness}
\[
\rho\geq e^{-1}\frac{\sqrt{2}}{6}>0.
\]
This completes the proof. 
\end{proof}

\begin{remark}
In the above proof, it seems impossible to use a more traditional moderate deviation result instead of Theorem 2.1 in \cite{FLS20}. This is because such a one requires that the moment generating function of $X_{n,11}$ is bounded uniformly in $n$ on a neighborhood of the origin (see Lemma 4.1 in \cite{JKV88} for instance).
\qed
\end{remark}

\section{Proofs for Section \ref{sec: bootstrap approximation}}

\begin{proof}[Proof of Theorem \ref{thm: multiplier bootstrap}]
The asserted claim is an immediate consequence of Theorem 1.1 in \cite{FK20} (restated as Lemma \ref{FK:GC} in this paper) once we note that (i) conditional on $X$, the random vector $W^{\xi}=n^{-1}\sum_{i=1}^n\xi_i(X_i-\bar X)(X_i-\bar X)^T$ is centered Gaussian with covariance matrix $n^{-1}\sum_{i=1}^n(X_i-\bar X)(X_i-\bar X)^T$ and (ii) for centered Gaussian random vectors, the Stein kernel is equal to the covariance matrix.
\end{proof}

\begin{proof}[Proof of Corollary \ref{thm: multiplier bootstrap simple}]
Like in the proof of Corollary \ref{thm: simple}, we assume, without loss of generality, that $\sigma_j=1$ for all $j=1,\dots,d$, so that $\sigma_{*,W}$ is the square root of the smallest eigenvalue of $\Sigma_W=\Ep WW^T$. To prove the asserted claims, we will apply Theorem \ref{thm: multiplier bootstrap} with $\Sigma=\Sigma_W$, so that $\sigma_*= \sigma_{*,W}$. This requires bounding $\Delta_0'$ with $\Sigma=\Sigma_W$. We do so separately for each case.

Consider first the case when (E.1) holds. We assume, without loss of generality, that
\begin{equation}\label{eq: mult boot simple e1}
\frac{B_n(\log d)(\log n)\sqrt{\log(d/\alpha)}}{\sqrt n\sigma_{*,W}^2}\leq 1,
\end{equation}
since otherwise the asserted claim is trivial. Using this assumption, we will now prove that there exists a universal constant $C'\geq 1$ such that
\begin{equation}\label{eq: e1need to proof}
\Pr\left(\Delta_0' > \frac{C' B_n(\log d)\sqrt{\log(d/\alpha)}}{\sqrt n\sigma_{*,W}^2}\right)\leq \alpha.
\end{equation}
This derivation is similar to the proof of Proposition 4.1 in \cite{CCK17}.

Note that 
\begin{equation}\label{eq: decomposition m boot proof}
\sigma_{*,W}^2\Delta_0'/\log d\leq \Delta^{(1)} + \Delta^{(2)},
\end{equation}
where 
\begin{equation}\label{eq: delta 1 delta 2}
\Delta^{(1)} := \left\| \frac{1}{n}\sum_{i=1}^n(X_iX_i^T - \Ep X_iX_i^T) \right\|_{\infty},\quad \Delta^{(2)} := \|\bar X\bar X^T\|_{\infty} = \|\bar X\|_{\infty}^2.
\end{equation}
We first bound $\Delta^{(1)}$. To do so, since (E.1) implies (M),
$$
\sigma_n^2:=\max_{1\leq j,k\leq d}\sum_{i=1}^n \Ep(X_{ij}X_{ik}-\Ep[X_{ij}X_{ik}])^2 \leq \max_{1\leq j,k\leq d}\sum_{i=1}^n \Ep(X_{ij}X_{ik})^2 \leq n B_n^2.
$$
Also,
$$
\left\| \max_{1\leq i\leq n}\max_{1\leq j,k\leq d}|X_{ij}X_{ik}| \right\|_{\psi_1}\lesssim B_n^2\log(dn),
$$
so that $M_n:= \max_{1\leq i\leq n}\max_{1\leq j,k\leq d}|X_{ij}X_{ik} - \Ep[X_{ij}X_{ik}]|$ satisfies
$$
\sqrt{\Ep M_n^2}\lesssim \|M_n\|_{\psi_1} \lesssim B_n^2\log(d n).
$$
Hence, by Lemma \ref{lem: maximal},
\begin{align*}
\Ep\Delta^{(1)} &\lesssim n^{-1}\left(\sqrt{\sigma_n^2\log d} + \sqrt{\Ep M_n^2}\log d\right)\\
&\lesssim \sqrt{n^{-1}B_n^2\log d} + n^{-1}B_n^2(\log d)\log(d n) \lesssim \sqrt{n^{-1}B_n^2\log d},
\end{align*}
where the last inequality follows from \eqref{eq: mult boot simple e1}. Thus, applying Lemma \ref{lem: fuk-nagaev}(i) with $\beta=\eta=1$, we have for all $t>0$ that
$
\Delta^{(1)} \lesssim \sqrt{n^{-1}B_n^2\log d}+t
$
with probability at least
$$
1 - \exp\left(-\frac{nt^2}{3B_n^2}\right) - 3\exp\left(-\frac{cnt}{B_n^2\log(d n)}\right),
$$ 
where $c>0$ is a universal constant. Setting here
$$
t = \frac{B_n\sqrt{3\log(4/\alpha)}}{\sqrt n} + \frac{B_n^2\log(d n)\log(12/\alpha)}{c n}
$$
and recalling \eqref{eq: mult boot simple e1}, we conclude that 
$$
\Delta^{(1)} \lesssim \sqrt{\frac{B_n^2\log(d/\alpha)}{n}}
$$
with probability at least $1 - \alpha/2$, and, by the same argument, we can also find that
$$
\Delta^{(2)} \lesssim \sqrt{\frac{B_n^2\log(d/\alpha)}{n}}
$$
again with probability at least $1 - \alpha/2$. Combining these inequalities and recalling \eqref{eq: decomposition m boot proof}, we obtain \eqref{eq: e1need to proof}.

Now, observe that the function $f\colon(0,1)\to\mathbb R$ defined by $f(x):=x(1\vee|\log x|)$ for all $x\in(0,1)$ is increasing. Also, note that by \eqref{eq: mult boot simple e1},
$$
\left|\log\left(\frac{C'B_n\sqrt{\log(d/\alpha)}}{\sqrt n\sigma_{*,W}^2}\right)\right| \lesssim \log n.
$$
Therefore, the asserted claim under condition (E.1) follows from combining Theorem \ref{thm: multiplier bootstrap} and \eqref{eq: e1need to proof}.

Further, the asserted claim in the case when (M) and (E.2) hold can be proven using the exactly same calculations as those in the case of (E.1). 

Now consider the case when (M) and (E.3) hold. In this case, we assume, again without loss of generality, that
\begin{equation}\label{eq: e2 mult almost there}
\frac{(\log d)(\log n)}{\sigma_{*,W}^2}\left(\frac{B_n\sqrt{\log(d/\alpha)}}{\sqrt n} + \frac{B_n^2(\log  d + \alpha^{-2/q})}{n^{1-2/q}}\right)\leq 1.
\end{equation} 
Then, defining $\sigma_n^2$ and $M_n$ as above, we have $\sigma_n^2\leq nB_n^2$ and
$
\sqrt{\Ep M_n^2}\lesssim \|M_n\|_{L_{q/2}}\lesssim n^{2/q}B_n^2.
$
Hence, by Lemma \ref{lem: maximal}, $\Delta^{(1)}$ defined in \eqref{eq: delta 1 delta 2} satisfies
\begin{align*}
\Ep\Delta^{(1)} &\lesssim n^{-1}\left(\sqrt{\sigma_n^2\log d} + \sqrt{\Ep M_n^2}\log d\right)\lesssim \sqrt{n^{-1}B_n^2\log d} + n^{-1+2/q}B_n^2\log d.
\end{align*}
Thus, applying Lemma \ref{lem: fuk-nagaev}(ii) with $\eta=1$ and $s=q/2$, we have for all $t>0$ that
$$
\Delta^{(1)}\lesssim \sqrt{n^{-1}B_n^2\log d} + n^{-1+2/q}B_n^2\log d + t
$$
with probability at least
$$
1-\exp\left(-\frac{nt^2}{3B_n^2}\right) - \frac{cnB_n^{q}}{n^{q/2}t^{q/2}},
$$
where $c>0$ is a universal constant. Setting here
$$
t=\frac{B_n\sqrt{3\log(4/\alpha)}}{\sqrt n} + \frac{(4c/\alpha)^{2/q}B_n^2}{n^{1-2/q}},
$$
we conclude that
$$
\Delta^{(1)}\lesssim \frac{B_n\sqrt{\log(d/\alpha)}}{\sqrt n} + \frac{B_n^2(\log d+\alpha^{-2/q})}{n^{1-2/q}}
$$
with probability at least $1 - \alpha/2$, and, by the same argument, we can also find, for $\Delta^{(2)}$ defined in \eqref{eq: delta 1 delta 2}, that
$$
\Delta^{(2)}\lesssim \frac{B_n\sqrt{\log(d/\alpha)}}{\sqrt n} + \frac{B_n^2(\log d+\alpha^{-2/q})}{n^{1-2/q}}
$$
again with probability at least $1 - \alpha/2$. Combining these inequalities and using \eqref{eq: decomposition m boot proof}, we obtain
$$
\Pr\left(\Delta_0' > \frac{C'\log d}{\sigma_{*,W}^2}\left(\frac{B_n\sqrt{\log(d/\alpha)}}{\sqrt n} + \frac{B_n^2(\log d + \alpha^{-2/q})}{n^{1-2/q}}\right)\right)\leq 1-\alpha,
$$
where $C'\geq 1$ is a universal constant. Here,
$$
\left|\log\left( \frac{C'}{\sigma_{*,W}^2}\left(\frac{B_n\sqrt{\log(d/\alpha)}}{\sqrt n} + \frac{B_n^2(\log d + \alpha^{-2/q})}{n^{1-2/q}}\right) \right)\right|\lesssim\log n
$$
by \eqref{eq: e2 mult almost there}. The proof can now be completed by applying Theorem \ref{thm: multiplier bootstrap} as we did in the case of (E.1).
\end{proof}

\begin{proof}[Proof of Theorem \ref{thm: empirical bootstrap}]
The asserted claim follows as a direct application of Theorem \ref{thm: general unbounded} once we note that $\Delta_0$, $\Delta_1$, and $M(\psi)$ are equal to $\Delta_0'$, $\Delta_1'$, and $M^*(\psi)$ and $\mathcal M$ is bounded from above by $\mathcal M^*$ if we substitute $X_1^*-\bar X,\dots,X_n^* - \bar X$ instead of $X_1,\dots,X_n$ in Theorem \ref{thm: general unbounded}. 
\end{proof}

\begin{proof}[Proof of Corollary \ref{cor: emp boot simple}]
Like in the proof of Corollary \ref{thm: simple}, we assume, without loss of generality, that $\sigma_j=1$ for all $j=1,\dots,d$, so that $\sigma_{*,W}$ is the square root of the smallest eigenvalue of $\Sigma_W=\Ep WW^T$. To prove the asserted claims, we will apply Theorem \ref{thm: empirical bootstrap} with $\Sigma=\Sigma_W$, so that $\sigma_*= \sigma_{*,W}$. This requires bounding all terms appearing in Theorem \ref{thm: empirical bootstrap}. We do so separately for each case.

Consider first the case when (E.1) holds. We assume, without loss of generality, that
\begin{equation}\label{eq: emp boot simple e1}
\frac{B_n(\log d)(\log n)\sqrt{\log(d/\alpha)}}{\sqrt n\sigma_{*,W}^2}\leq 1,
\end{equation}
since otherwise the asserted claim is trivial. Then by the proof of Corollary \ref{thm: multiplier bootstrap simple},
\begin{equation}\label{eq: delta0 prime from m boot}
\Delta_0'\log n\lesssim \frac{B_n(\log d)(\log n)\sqrt{\log(d/\alpha)}}{\sqrt n\sigma_{*,W}^2}
\end{equation}
with probability at least $1-\alpha/2$. Also, by Lemma \ref{lem: maximal ineq nonnegative} and \eqref{eq: emp boot simple e1},
$$
\Ep\max_{1\leq j\leq d}\sum_{i=1}^n|X_{ij}|^4 \lesssim nB_n^2+B_n^4\log d\lesssim nB_n^2.
$$
Thus, by Lemma \ref{lem: deviation ineq nonnegative}(i) with $\eta=\beta=1$, we have for any $t>0$ that
$
\max_{1\leq j\leq d}\sum_{i=1}^n |X_{ij}|^4\lesssim nB_n^2 + t
$
with probability at least
$$
1-3\exp\left(-\frac{t}{cB_n^4}\right),
$$
where $c>0$ is a universal constant. Setting here $t=cB_n^4\log(6/\alpha)$, using Jensen's inequality, and recalling \eqref{eq: emp boot simple e1}, we have that
$$
\max_{1\leq j\leq d}\sum_{i=1}^n (X_{ij} - \bar X_j)^4\lesssim \max_{1\leq j\leq d}\sum_{i=1}^n |X_{ij}|^4\lesssim nB_n^2+B_n^4\log(6/\alpha)\lesssim nB_n^2
$$
with probability at least $1 - \alpha/2$, and so
$$
(\log n)\sqrt{\Delta_1'\log d} \lesssim \frac{B_n(\log d)^{3/2}\log n}{\sqrt n \sigma_{*,W}^2}
$$
with the same probability. Further, $\mathcal M^*\lesssim B_n$, and so
$$
\frac{(\mathcal M^*\log d)^2\log n}{n\sigma_{*,W}^2}\lesssim \frac{B_n^2(\log d)^2\log n}{n\sigma_{*,W}^2}\lesssim \frac{B_n(\log d)^{3/2}\log n}{\sqrt n \sigma_{*,W}^2}
$$
by \eqref{eq: emp boot simple e1}.
Finally, setting $\psi=2B_n$ gives $M^*(\psi)=0$ and
$$
\frac{\psi(\log d)^{3/2}}{\sqrt n\sigma_{*,W}}\lesssim \frac{B_n(\log d)^{3/2}}{\sqrt n\sigma_{*,W}}.
$$
Combining presented inequalities and applying Theorem \ref{thm: empirical bootstrap} gives the asserted claim under condition (E.1).

Next, consider the case when (M) and (E.2) hold. We assume, without loss of generality, that 
\begin{equation}\label{eq: emp boot simple e2}
\frac{B_n(\log d)(\log n)\sqrt{\log(d/\alpha)}}{\sqrt n\sigma_{*,W}^2}+\frac{B_n(\log(dn))^2\sqrt{\log(1/\alpha)}}{\sqrt n\sigma_{*,W}}\leq 1
\end{equation}
since otherwise the asserted claim is trivial. Then, again by the proof of Corollary \ref{thm: multiplier bootstrap simple}, $\Delta_0'$ satisfies \eqref{eq: delta0 prime from m boot} with probability at least $1 - \alpha/4$. Further, by Jensen's inequality, \eqref{eq: emp boot simple e2}, and Lemmas \ref{lem: maximal ineq nonnegative} and \ref{lem: deviation ineq nonnegative}(i) with $\eta=1$ and $\beta=1/2$,
\begin{align*}
\max_{1\leq j\leq d}\sum_{i=1}^n (X_{ij} - \bar X_j)^4
&\lesssim nB_n^2 + B_n^4(\log(d n))^2\log d + B_n^4(\log(d n))^2(\log(1/\alpha))^2\\
&\lesssim nB_n^2 + B_n^4(\log(d n))^2(\log(1/\alpha))^2,
\end{align*}
with probability at least $1-\alpha/4$, and so
\begin{align*}
(\log n)\sqrt{\Delta_1'\log d}
&\lesssim \frac{B_n(\log d)^{3/2}\log n}{\sqrt n\sigma_{*,W}^2} + \frac{B_n^2(\log d)^{3/2}\log(dn)\log(1/\alpha)\log n}{n\sigma_{*,W}^2}\\
&\lesssim \frac{B_n(\log d)^{3/2}\log n}{\sqrt n\sigma_{*,W}^2} + \frac{B_n(\log(dn))^2\sqrt{\log(1/\alpha)}}{\sqrt n\sigma_{*,W}}
\end{align*}
with the same probability by \eqref{eq: emp boot simple e2}. In addition, $\mathcal M^*\lesssim B_n\sqrt{\log(dn)}\sqrt{\log(1/\alpha)}$ with probability at least $1-\alpha/4$, and so
$$
\frac{(\mathcal M^*\log d)^2\log n}{n\sigma_{*,W}^2}\lesssim \frac{B_n^2(\log(d n))^4\log(1/\alpha)}{n\sigma_{*,W}^2}\lesssim \frac{B_n(\log(dn))^2\sqrt{\log(1/\alpha)}}{\sqrt n\sigma_{*,W}}
$$
with the same probability by \eqref{eq: emp boot simple e2}. Finally, setting $\psi=C'B_n\sqrt{\log(dn)}\sqrt{\log(1/\alpha)}$ with a sufficiently large but universal constant $C'>0$, we have $M^*(\psi)=0$ with probability at least $1 - \alpha/4$, and
$$
\frac{\psi(\log d)^{3/2}}{\sqrt n\sigma_{*,W}}\lesssim \frac{B_n(\log(d n))^2\sqrt{\log(1/\alpha)}}{\sqrt n\sigma_{*,W}}.
$$
Combining presented inequalities and applying Theorem \ref{thm: empirical bootstrap} gives the asserted claim under conditions (M) and (E.2).

Now consider the case when (M) and (E.3) hold. We assume, without loss of generality, that
\begin{equation}\label{eq: emp boot simple e3}
\frac{B_n(\log d)(\log n)\sqrt{\log(d/\alpha)}}{\sqrt n\sigma_{*,W}^2}+\frac{B_n\sqrt{\log(d n)}\log d}{n^{1/2-1/q}\alpha^{1/q}\sigma_{*,W}}\leq1
\end{equation}
since otherwise the asserted claim is trivial. Using this assumption, by the proof of Corollary \ref{thm: multiplier bootstrap simple},
\begin{align*}
\Delta_0'\log n &\lesssim \frac{(\log d)(\log n)}{\sigma_{*,W}^2}\left(\frac{B_n\sqrt{\log(d/\alpha)}}{\sqrt n} + \frac{B_n^2(\log  d + \alpha^{-2/q})}{n^{1-2/q}}\right)\\
&\lesssim \frac{B_n(\log d)(\log n)\sqrt{\log(d/\alpha)}}{\sqrt n\sigma_{*,W}^2}+\frac{B_n\sqrt{\log(d n)}\log d}{n^{1/2-1/q}\alpha^{1/q}\sigma_{*,W}}
\end{align*}
with probability at least $1 - \alpha/4$. Also, given that
$$
\Ep\max_{1\leq i\leq n}\max_{1\leq j\leq d}|X_{i j}|^4 \leq \left(\Ep\max_{1\leq i\leq n}\max_{1\leq j\leq d}|X_{i j}|^q\right)^{4/q}\leq n^{4/q}B_n^4,
$$
we have by Lemma \ref{lem: maximal ineq nonnegative} that
$$
\Ep\max_{1\leq j\leq d}\sum_{i=1}^n |X_{ij}|^4\lesssim n B_n^2 + n^{4/q}B_n^4\log d.
$$
Thus, by Lemma \ref{lem: deviation ineq nonnegative}(ii) with $\eta=1$ and $s=q/4$, we have for any $t>0$ that
$$
\max_{1\leq j\leq d}\sum_{i=1}^n |X_{ij}|^4\lesssim nB_n^2 + n^{4/q}B_n^4\log d + t
$$
with probability at least $1-(cnB_n^q)/t^{q/4}$, where $c>0$ is a universal constant. Setting here $t=(4cn/\alpha)^{4/q}B_n^4$ and using Jensen's inequality, we have that
$$
\max_{1\leq j\leq d}\sum_{i=1}^n (X_{ij} - \bar X_j)^4\lesssim \max_{1\leq j\leq d}\sum_{i=1}^n |X_{ij}|^4\lesssim nB_n^2 + n^{4/q}B_n^4(\log d + \alpha^{-4/q})
$$
with probability at least $1-\alpha/4$, and so
\begin{align*}
(\log n)\sqrt{\Delta_1'\log d}&\lesssim \frac{B_n(\log d)^{3/2}\log n}{\sqrt n\sigma_{*,W}^2} + \frac{B_n^2(\log d)^{3/2}(\log n)(\sqrt{\log d} + \alpha^{-2/q})}{n^{1-2/q}\sigma_{*,W}^2}\\
&\lesssim \frac{B_n(\log d)(\log n)\sqrt{\log(d/\alpha)}}{\sqrt n\sigma_{*,W}^2}+\frac{B_n\sqrt{\log(d n)}\log d}{n^{1/2-1/q}\alpha^{1/q}\sigma_{*,W}}
\end{align*}
with the same probability by \eqref{eq: emp boot simple e3}. In addition, by Markov's inequality, for any $t>0$,
$$
\Pr\left(\max_{1\leq j\leq d}\max_{1\leq i\leq n}|X_{i j}|>t\right) \leq t^{-q}\Ep\max_{1\leq j\leq d}\max_{1\leq i\leq n}|X_{i j}|^q\leq nB_n^q/t^q,
$$
and so
$\mathcal M^*\lesssim n^{1/q}B_n/\alpha^{1/q}$ with probability at least $1-\alpha/4$, so that
$$
\frac{(\mathcal M^*\log d)^2\log n}{n\sigma_{*,W}^2}\lesssim \frac{B_n^2(\log d)^2\log n}{n^{1-2/q}\alpha^{2/q}\sigma_{*,W}^2}\lesssim \frac{B_n\sqrt{\log(d n)}\log d}{n^{1/2-1/q}\alpha^{1/q}\sigma_{*,W}}
$$
with the same probability by \eqref{eq: emp boot simple e3}. Finally, setting $\psi = C'n^{1/q}B_n/\alpha^{1/q}$ for a sufficiently large but universal constant $C'>0$, we have $M^*(\psi)=0$ with probability at least $1-\alpha/4$ and
$$
\frac{\psi(\log d)^{3/2}}{\sqrt n\sigma_{*,W}}\lesssim \frac{B_n(\log d)^{3/2}}{n^{1/2-1/q}\alpha^{1/q}\sigma_{*,W}}.
$$
Combining all presented inequalities and applying Theorem \ref{thm: empirical bootstrap} gives the asserted claim under conditions (M) and (E.3).
\end{proof}

\section{Proofs for Section \ref{sec: special cases}}

% \subsection{Proof of thm t-QG}
\begin{proof}[Proof of Theorem \ref{t-QG}]
%First we prove the second asserted claim subject to the first one. 
%The claim immediately follows if $\delta\sqrt{\log d}\leq \sigma_{*,0}$, so we consider the case $\delta\sqrt{\log d}>\sigma_{*,0}$. If $\sigma_{*,0}\leq1$, the claim holds as long as $C'\geq1$ since $\tilde\varrho_{\Sigma_W}\leq1$. 
%Finally, if $\sigma_{*,0}>1$, we have $n^{-1}\sum_{i=1}^n \Ep (\sigma_{*,0}^{-1}X_{ij})^2\leq1$ for all $j=1,\dots,d$ by assumption. Since
%\[
%\tilde\varrho_{\Sigma} = \sup_{A\in\mathcal R}|\Pr (\sigma_{*,0}^{-1}\tilde W\in A)-\Pr (\sigma_{*,0}^{-1}\tilde Z\in A)|,
%\]
%the desired claim follows from that in the case $\sigma_{*,0}=1$. 

%We turn to the proof of the first asserted claim. 
%Observe that for any $w\in\mathbb R^d$ and $A = \prod_{j=1}^d(a_j,b_j]\in\mathcal R$, $\Pr ( \tilde W \in A) = \Pr (\tilde W \leq b)-\Pr (\tilde W \leq a)$, where $a=(a_1,\dots,a_d)^T\in\mathbb R^d$ and $b=(b_1,\dots,b_d)^T\in\mathbb R^d$. Therefore, it suffices to prove the asserted claim with $\tilde\varrho_{\Sigma}$ replaced by
%$$
%\tilde\varrho' = \tilde\varrho'_{\Sigma}:= \sup_{A\in \mathcal R} |\Pr (\tilde W\leq r)- \Pr (\tilde Z\leq r)|,\qquad \tilde Z\sim N(0,\tilde \Sigma),
%$$
%which is what we do below.
 
For all $i=1,\dots,n$, denote $\xi_i := X_i/\sqrt n$, so that $W = \sum_{i=1}^n \xi_i$. Also, let $Z\sim N(0,\Sigma)$ be independent of everything else. Then$$
\tilde W = W+ G\quad\text{and}\quad  \tilde Z = Z + G.
$$
In addition, for any $A\in\mathcal R$, let $h^A\colon\mathbb R^d\to\mathbb R$ be the function defined by
 $$h^A(x) := \Ep 1_A(x + G).$$  For brevity of notations, we suppress the dependence on $A$ in what follows.

By Lemma \ref{lemma:fk2.2} with $\phi = \infty$ and $\epsilon=1$,  the function $h$ is infinitely differentiable and each derivative is bounded by a constant
that only depends on the order of the derivative; in particular,
\begin{equation}\label{FKbounds-t2}
\sup_{x\in\mathbb{R}^d}\sum_{j,k=1}^d\left|\partial_{jk} h(x)\right|
\lesssim \frac{\log d}{\sigma^2_{*,0} }, \quad \sup_{x\in\mathbb{R}^d} \sum_{j,k,l=1}^d \sup_{y \in R(0, \sigma_{*,0} \eta)}\left|\partial_{jkl} h(x+y)\right| \lesssim \frac{(\log d)^{3/2}}{\sigma^{3}_{*,0}},
\end{equation}
where $\eta:=2c/\sqrt{\log d}$. The second property, local stability of the derivative, is important to obtain good dependence on $\delta$. Since $h$ is infinite differentiable and has bounded derivates, we can freely interchange differentiation and integration below, without further announcement.  

Now, write 
\begin{equation}\label{eq: starting point qg}
\Pr(\tilde W \in A) - \Pr(\tilde Z \in A) = \Ep h (W) - \Ep h(Z).
%\Pr(\tilde W \leq r) - \Pr(\tilde Z \leq r) = \Ep h (W) - \Ep h(Z).
\end{equation}
Also, define the Slepian interpolant
$$F(s):=\sqrt{1-s} F+\sqrt{s} Z,\quad s\in[0,1],$$
 for any random vector $F$ in $\mathbb{R}^d$.  Using fundamental theorem of calculus and integration by parts, write \eqref{eq: starting point qg} further as:
\begin{align}
\Ep h(W) -  \Ep h(Z)&=- \frac{1}{2}\int_0^1\Ep\left\langle \nabla h(W(s)),\frac{Z}{\sqrt s} - \frac{W}{\sqrt{1-s}} \right\rangle ds
\nonumber\\
&= - \frac{1}{2}\int_0^1\Ep\left[\langle\Sigma,\nabla^2 h(W(s))\rangle - \left\langle\frac{W}{\sqrt{1-s}},\nabla h(W(s))\right\rangle\right]ds.\label{slepian-ip}
\end{align}
%To bound the integral here, we proceed like in the proof of Theorem \ref{t1}, i.e. employing the exchangeable pair approach.
To bound the integral here, we employ Stein's leave-one-out trick. 

First, we have 
\begin{equation*}
\Ep\left[\left\langle\frac{W}{\sqrt{1-s}},\nabla h(W(s))\right\rangle\right]
=\frac{1}{\sqrt{1-s}}\sum_{i=1}^n\sum_{j=1}^d\Ep[\xi_{ij}\partial_jh(W(s))].
\end{equation*}
Let $W^{(i)}:=W-\xi_i$ for $i=1,\dots,n$. Taylor expanding $\partial_jh(W(s))$ around $W^{(i)}(s)=W(s)-\sqrt{1-s}\xi_i$ for each $i$, we obtain
\begin{align*}
&\Ep\left[\left\langle\frac{W}{\sqrt{1-s}},\nabla h(W(s))\right\rangle\right]\\
&=\frac{1}{\sqrt{1-s}}\sum_{i=1}^n\sum_{j=1}^d\Ep[\xi_{ij}\partial_jh(W^{(i)}(s))]
+\sum_{i=1}^n\sum_{j,k=1}^d\Ep[\xi_{ij}\xi_{ik}\partial_{jk}h(W^{(i)}(s))]
+R_1(s),
\end{align*}
where
\[
R_1(s)=\sqrt{1-s}\sum_{i=1}^n\sum_{j,k,l=1}^d\Ep[(1-U)\xi_{ij}\xi_{ik}\xi_{il}\partial_{jkl}h(W^{(i)}(s)+\sqrt{1-s}U\xi_i)],
\]
and $U$ is a uniform random variable on $[0,1]$ independent of everything else. 
Using the independence between $W^{(i)}$ and $\xi_i$ as well as $\Ep[\xi_{ij}]=0$, we deduce
\begin{equation}\label{r1-loo}
\Ep\left[\left\langle\frac{W}{\sqrt{1-s}},\nabla h(W(s))\right\rangle\right]
=\sum_{i=1}^n\sum_{j,k=1}^d\Ep[\xi_{ij}\xi_{ik}]\Ep[\partial_{jk}h(W^{(i)}(s))]+R_1(s).
\end{equation}
Next, we decompose $\Ep\left[\langle\Sigma,\nabla^2 h(W(s))\rangle\right]$ as
\besn{\label{cov-loo}
&\Ep\left[\langle\Sigma,\nabla^2 h(W(s))\rangle\right]\\
&=\Ep\left[\langle\Sigma-\Sigma_W,\nabla^2 h(W(s))\rangle\right]
+\Ep\left[\langle\Sigma_W,\nabla^2 h(W(s))\rangle\right].
}
We have by the first inequality in \eqref{FKbounds-t2}
\begin{equation}\label{r0-loo}
\left|\Ep\left[\langle\Sigma-\Sigma_W,\nabla^2 h(W(s))\rangle\right]\right|
\lesssim\tilde\Delta_0.
\end{equation}
Meanwhile, we further decompose $\Ep\left[\langle\Sigma_W,\nabla^2 h(W(s))\rangle\right]$ in the following. 
Since $\Sigma_W=\sum_{i=1}^n\Ep[\xi_{i}\xi_{i}^T]$, we have
\begin{equation*}
\Ep\left[\langle\Sigma_W,\nabla^2 h(W(s))\rangle\right]
=\sum_{i=1}^n\sum_{j,k=1}^d\Ep[\xi_{ij}\xi_{ik}]\Ep\left[\partial_{jk}h(W(s))\right].
\end{equation*}
Taylor expanding $\partial_{jk}h(W(s))$ around $W^{(i)}(s)$, we obtain
\begin{equation}\label{r2-loo}
\Ep\left[\langle\Sigma_W,\nabla^2 h(W(s))\rangle\right]
=\sum_{i=1}^n\sum_{j,k=1}^d\Ep[\xi_{ij}\xi_{ik}]\Ep[\partial_{jk}h(W^{(i)}(s))]
+R_2(s),
\end{equation}
where
\[
R_2(s)=\sqrt{1-s}\sum_{i=1}^n\sum_{j,k,l=1}^d\Ep[\xi_{ij}\xi_{ik}]\Ep[\xi_{il}\partial_{jkl}h(W^{(i)}(s)+\sqrt{1-s}U\xi_i)].
\] 
%Using the independence between $W^{(i)}$ and $\xi_i$ as well as $\Ep[\xi_i]=0$ again, we deduce
%\begin{equation}\label{r2-loo}
%\Ep\left[\langle\Sigma_W,\nabla^2 h(W(s))\rangle\right]
%=\sum_{i=1}^n\sum_{j,k=1}^d\Ep[\xi_{ij}\xi_{ik}]\Ep[\partial_{jk}h(W^{(i)}(s))]
%+R_2(s).
%\end{equation}
From \eqref{eq: starting point qg}--\eqref{r2-loo}, we obtain
\begin{equation}\label{eq:loo}
|\Pr(\tilde W\in A)-\Pr(\tilde Z\in A)|
\lesssim\tilde\Delta_0+\int_0^1|R_1(s)|ds+\int_0^1|R_2(s)|ds.
\end{equation}
Therefore, we complete the proof once we show
\begin{equation}\label{remain-loo}
\int_0^1|R_1(s)|ds\lesssim\tilde\Delta_1\quad
\text{and}\quad
\int_0^1|R_2(s)|ds\lesssim\tilde\Delta_1.
\end{equation}

Since $\|\xi_i\|_\infty\leq\delta\leq\sigma_{*,0}\eta/2$ by assumption, we have
\begin{align*}
|R_{1}(s)|&\leq\sum_{i=1}^n\sum_{j,k,l=1}^d\Ep\left[|\xi_{ij}\xi_{ik}\xi_{il}|\sup_{y\in R(0,\sigma_{*,0}\eta/2)}|\partial_{jkl}h(W^{(i)}(s)+y)|\right]\\
&=\sum_{i=1}^n\sum_{j,k,l=1}^d\Ep\left[|\xi_{ij}\xi_{ik}\xi_{il}|\right]\Ep\left[\sup_{y\in R(0,\sigma_{*,0}\eta/2)}|\partial_{jkl}h(W^{(i)}(s)+y)|\right],
\end{align*}
where the last line follows from the independence between $W^{(i)}$ and $\xi_i$. 
Using $\|\xi_i\|_\infty\leq\sigma_{*,0}\eta/2$ again, we conclude
\begin{align*}
|R_{1}(s)|&\leq\sum_{i=1}^n\sum_{j,k,l=1}^d\Ep\left[|\xi_{ij}\xi_{ik}\xi_{il}|\right]\Ep\left[\sup_{y\in R(0,\sigma_{*,0}\eta)}|\partial_{jkl}h(W(s)+y)|\right]\\
&\leq\left(\max_{1\leq j\leq d}\sum_{i=1}^n\Ep[|\xi_{ij}|^3]\right)\sum_{j,k,l=1}^d\Ep\left[\sup_{y\in R(0,\sigma_{*,0}\eta)}|\partial_{jkl}h(W(s)+y)|\right].
\end{align*}
Therefore, we obtain by the second inequality in \eqref{FKbounds-t2}
\begin{equation}\label{r1-loo-est}
|R_{1}(s)|\lesssim\max_{1\leq j\leq d}\sum_{i=1}^n\Ep[|\xi_{ij}|^3]\frac{\log^{3/2} d}{\sigma_{*,0}^3}=\tilde\Delta_1.
\end{equation}
This yields the first inequality in \eqref{remain-loo}; since we can similarly prove the second one, the desired result follows from combining the bounds and noting that these bounds do not depend on $A$.
\end{proof}

\begin{proof}[Proof of Theorem \ref{t-QG2}]
The proof is a modification of the proof of Theorem \ref{t-QG}, inspired by \cite{Ga20}. 
Here, we describe the changes, keeping all unmentioned notations the same as those in the proof of Theorem \ref{t-QG}. 
In a nutshell, we only need to consider an additional Taylor expansion when bounding $R_1(s)$ and $R_2(s)$. 

%The proof of the second asserted claim is the same as in the proof of Theorem \ref{t-QG}, so we focus on the first asserted claim. 
%As in the proof of Theorem \ref{t-QG}, it suffices to prove the asserted claim with $\tilde\varrho_{\Sigma}$ replaced by $\tilde\varrho'_\Sigma$. 
Note that \eqref{eq:loo} holds under our current assumptions by the same arguments as those in the proof of Theorem \ref{t-QG}. Thus, we only need to bound
$$
\int_0^1|R_1(s)|ds\quad\text{and}\quad\int_0^1|R_2(s)|ds.
$$
To bound the former, we consider the Taylor expansion of $\partial_{jkl}h(W^{(i)}(s)+\sqrt{1-s}U\xi_i)$ around $W^{(i)}(s)$ and rewrite $R_1(s)$ as 
\[
R_1(s)=\sqrt{1-s}\sum_{i=1}^n\sum_{j,k,l=1}^d\Ep[(1-U)\xi_{ij}\xi_{ik}\xi_{il}\partial_{jkl}h(W^{(i)}(s))]
+R_1'(s),
\]
where
\[
R_1'(s)=(1-s)\sum_{i=1}^n\sum_{j,k,l,q=1}^d\Ep[U(1-U)\xi_{ij}\xi_{ik}\xi_{il}\xi_{iq}\partial_{jklq}h(W^{(i)}(s)+\sqrt{1-s}UU'\xi_i)],
\]
and $U'$ is a uniform random variable on $[0,1]$ independent of everything else. 
Since $W^{(i)},\xi_i$ and $U$ are independent, we obtain
\[
R_1(s)=\sqrt{1-s}\sum_{i=1}^n\sum_{j,k,l=1}^d\Ep[1-U]\Ep[\xi_{ij}\xi_{ik}\xi_{il}]\Ep[\partial_{jkl}h(W^{(i)}(s))]
+R_1'(s).
\]
Thus, we conclude $R_1(s)=R_1'(s)$ by the assumption \eqref{zero-skew}. 
Now, note that we have by Lemma \ref{lemma:fk2.2} with $\phi = \infty$ and $\epsilon=1$
\[
\sup_{x\in\mathbb{R}^d} \sum_{j,k,l,q=1}^d \sup_{y \in R(0, \sigma_{*,0} \eta)}\left|\partial_{jklq} h(x+y)\right| \lesssim \frac{(\log d)^{2}}{\sigma^{4}_{*,0}}.
\]
Using this inequality instead of the second one in \eqref{FKbounds-t2}, we can prove $|R_1'(s)|\lesssim\tilde\Delta_2$ analogously to the proof of \eqref{r1-loo-est}. A similar argument also yields $|R_2'(s)|\lesssim\tilde\Delta_2$, completing the proof. 
\end{proof}

\begin{proof}[Proof of Proposition \ref{prop:lower-bound2}]
The proof is almost the same as that of \cite[Proposition 1.1]{FK20}, except that we use Theorem 3 in \cite[Chapter VIII]{Petrov75} instead of Eq.(2.41) in \cite[Chapter VIII]{Petrov75}. 

It suffices to show that there is a sequence $(x_n)_{n=1}^\infty$ of real numbers such that
\[
\rho:=\liminf_{n\to\infty}\frac{n}{\log^{2}d}\left|\Pr\left(\max_{1\leq j\leq d}W_{n,j}\leq x_n\right)-\Pr\left(\max_{1\leq j\leq d}Z_{j}\leq x_n\right)\right|>0,
\]
where
\[
W_{n,j}:=\frac{1}{\sqrt n}\sum_{i=1}^nX_{n,ij}.
\]

We define the sequence $(x_n)_{n=1}^\infty$ in the same way as in the proof of Proposition \ref{prop:lower-bound}. 
Then we can prove \eqref{est:normal-tail} and \eqref{eq:chen-stein} by the same arguments as in the proof of Proposition \ref{prop:lower-bound}. 
Moreover, since $x_n=O(\sqrt{\log d})=o(n^{1/4})$ by assumption, Theorem 3 in \cite[Chapter VIII]{Petrov75} implies
\begin{align*}%\label{eq:mdp}
\frac{P\left( W_1>x_n\right)}{1-\Phi_1(x_n)}&=e^{\frac{\gamma}{24n}x_n^4}\left\{1+O\left(\frac{x_n+1}{\sqrt{n}}\right)\right\}.
%=1+\frac{\gamma}{24n}x_n^4+o\left(\frac{x_n^4}{n}\right)+O\left(\frac{x_n+1}{\sqrt{n}}\right),
\end{align*}
%where the second line follows from the Maclaurin expansion of the exponential function. 
Combining this with \eqref{est:normal-tail}, we obtain
\begin{align*}
\lambda_n&=\frac{P\left( W_1>x_n\right)}{1-\Phi_1(x_n)}\cdot d(1-\Phi_1(x_n))\\
&=e^{\frac{\gamma}{24n}x_n^4}\left\{1+O\left(\frac{x_n+1}{\sqrt{n}}\right)\right\}+O\left(\frac{1}{d}\right).
\end{align*}
Using the Maclaurin expansion of the exponential function, we obtain
\[
\lambda_n=1+\frac{\gamma}{24n}x_n^4+o\left(\frac{x_n^4}{n}\right)+O\left(\frac{x_n+1}{\sqrt{n}}\right)+O\left(\frac{1}{d}\right)
\]
and thus
\[
e^{-\lambda_n+1}=1-\frac{\gamma}{24n}x_n^4+o\left(\frac{x_n^4}{n}\right)+O\left(\frac{x_n+1}{\sqrt{n}}\right)+O\left(\frac{1}{d}\right).
\]
Note that we particularly have $\lambda_n=O(1)$. Thus, \eqref{eq:chen-stein} yields
\ba{
\rho\geq\liminf_{n\to\infty}\frac{n}{\log^{2}d}|e^{-\lambda_n}-e^{-1}|
=\liminf_{n\to\infty}e^{-1}\frac{n}{\log^{2}d}\frac{|\gamma| x_n^4}{24}
=e^{-1}\frac{|\gamma|}{6}
}
because $x_n/\sqrt{2\log d}\to1$ as well as $d^{-1}=o((\log d)^{-1})=o(n^{-1}\log^2d)$ and $n/\log^3d\to0$ by assumption. 
This completes the proof. 
\end{proof}

\section{Proofs for Section \ref{sec: smoothing inequalities}}\label{sec: proofs smoothing inequalities}
Throughout this section, for brevity of notations, we often drop super-indices $\phi$, $\epsilon$, $A$, and $\Sigma$ in the functions $g^{\phi}(\cdot)$, $m^{A,\phi}(\cdot)$,
and $\rho^{A,\phi,\epsilon,\Sigma}(\cdot)$ and simply write $g(\cdot)$,
$m(\cdot)$, and $\rho(\cdot)$ instead. Also, we use $\lesssim$ to denote inequalities that hold up to a universal constant. 

We also introduce some additional notations used throughout this section. 
Let $\varphi\colon\mathbb{R}^{d}\to\mathbb{R}$ denote the pdf
of the standard normal distribution on $\mathbb{R}^{d}$. In addition,
let $\varphi_{1}\colon\mathbb{R}\to\mathbb{R}$ and $\Phi_{1}\colon\mathbb{R}\to\mathbb{R}$
denote the pdf and the cdf of the standard normal distribution on
$\mathbb{R}$.
Moreover, for an integer $\nu\geq0$, the $\nu$-th Hermite polynomial is denoted by $H_\nu$: $H_\nu(t)=(-1)^\nu\varphi_1(t)^{-1}\varphi_1^{(\nu)}(t)$. 
We denote by $t_\nu$ the maximum root of $H_\nu$ when $\nu\geq1$. For example, $t_1=0,t_2=1,t_3=\sqrt{3}$. It is evident that $H_\nu$ is positive and strictly increasing on $(t_\nu,\infty)$. 
We also have 
\ben{\label{h-root}
t_1<t_2<\cdots;
} 
see e.g.~\cite[Theorem 3.3.2]{Sz39}. 
When $\nu\geq1$, we define the function $h_\nu$ on $\mathbb{R}$ by $h_\nu(t)=H_{\nu-1}(t)\varphi_1(t),~t\in\mathbb{R}$. 
In addition, set $M_{\nu}:=\max_{0\leq t\leq t_\nu}|H_{\nu-1}(t)|<\infty$ and define the function $\tilde h_\nu:[0,\infty)\to(0,\infty)$ by
\[
\tilde h_\nu(t)=M_{\nu}\varphi_1(t)1_{[0,t_\nu]}(t)+h_{\nu}(t)1_{(t_\nu,\infty)}(t),\qquad t\in[0,\infty).
\]
A simple computation shows $h_\nu'(t)=-h_{\nu+1}(t)$; hence $h_\nu$ is strictly decreasing on $[t_\nu,\infty)$. Moreover, since $h_\nu$ is either even or odd, we have $|h_\nu(-t)|=|h_\nu(t)|$ for all $t\in\mathbb{R}$. These facts imply the following properties of $\tilde h_\nu$:
\begin{align}
&\tilde h_\nu \text{ is decreasing on }[0,\infty).\label{h-decrease}\\
&|h_\nu(t)|\leq\tilde h_\nu(|t|)\text{ for all }t\in\mathbb R.\label{h-bound}
\end{align}
%\ben{\label{even}
%|h_\nu(-t)|=|h_\nu(t)|\qquad\text{for any }t\in\mathbb{R}.
%}
Finally, for every $u\in\{1,\dots,v\}$, we set 
\ba{
\mathcal{N}^u(v)&:=\{(\nu_1,\dots,\nu_u)\in\mathbb{Z}^u:\nu_1,\dots,\nu_u\geq1,\nu_1+\cdots+\nu_u=v\},\\
\mathcal{J}^u(d)&:=\{(j_1,\dots,j_u)\in\{1,\dots,d\}^u:j_1,\dots,j_u\text{ are mutually different}\}.
}

\begin{proof}[Proof of Lemma \ref{lem: smoothing inequality mixed 1}]
First, note that the asserted claim for general $\epsilon$ follows from the asserted claim for $\epsilon=1$. Indeed, since $m^{A,\phi}(\epsilon w)=m^{\epsilon^{-1}A,\epsilon\phi}(w)$ by definition, we have
\ben{\label{rho-epsilon}
\rho^{A,\phi,\epsilon,\Sigma}(w)
=\Ep m^{A,\phi}(\epsilon(w/\epsilon + Z))
=\rho^{\epsilon^{-1}A,\epsilon\phi,1,\Sigma}(w/\epsilon).
}
Hence
\ba{
&\sup_{A\in\mathcal R}\sup_{w\in\mathbb R^d}\sum_{j_{1},\dots,j_{v}=1}^{d}\sup_{y\in R(0,\epsilon\sigma_*\eta)}|\partial_{j_{1},\dots,j_{v}}\rho^{A,\phi,\epsilon,\Sigma}(w+y)| \\
&\qquad =\frac{1}{\epsilon^v}\sup_{A\in\mathcal R}\sup_{w\in\mathbb R^d}\sum_{j_{1},\dots,j_{v}=1}^{d}\sup_{y\in R(0,\epsilon\sigma_*\eta)}|\partial_{j_{1},\dots,j_{v}}\rho^{\epsilon^{-1}A,\epsilon\phi,1,\Sigma}((w+y)/\epsilon)| \\
&\qquad =\frac{1}{\epsilon^v}\sup_{A\in\mathcal R}\sup_{w\in\mathbb R^d}\sum_{j_{1},\dots,j_{v}=1}^{d}\sup_{y\in R(0,\sigma_*\eta)}|\partial_{j_{1},\dots,j_{v}}\rho^{A,\epsilon\phi,1,\Sigma}(w+y)|.
}
Similarly, the asserted claim for general $\Sigma$ follows from the asserted claim for $\Sigma=I_d$. Indeed, define $\Sigma^1 = \Sigma - \sigma_*^2 I_d$ and $\Sigma^2 = \sigma_*^2 I_d$ and let $Z^1$ and $Z^2$ be independent random vectors in $\mathbb R^d$ such that $Z^1\sim N(0,\Sigma^1)$ and $Z^2\sim N(0,\Sigma^2)$. Then $Z\sim N(0,\Sigma)$ is equal in distribution to $Z^1 + Z^2$. Hence,
$$
\rho^{A,\phi,\epsilon,\Sigma}(w)=\Ep m^{A,\phi}(w+\epsilon Z^1 + \epsilon Z^2) = \Ep \rho^{A,\phi,\epsilon\sigma_*,I_d}(w+\epsilon Z^1),
$$
and so, by Jensen's inequality,
\begin{align*}
&\sup_{w\in\mathbb R^d}\sum_{j_{1},\dots,j_{v}=1}^{d}\sup_{y\in R(0,\epsilon\sigma_*\eta)}|\partial_{j_{1},\dots,j_{v}}\rho^{A,\phi,\epsilon,\Sigma}(w+y)| \\
&\qquad \leq \Ep\left[\sup_{w\in\mathbb R^d}\sum_{j_{1},\dots,j_{v}=1}^{d}\sup_{y\in R(0,\epsilon\sigma_*\eta)}|\partial_{j_{1},\dots,j_{v}}\rho^{A,\phi,\epsilon\sigma_*,I_d}(w + \epsilon Z^1 +y)|\right]\\
&\qquad = \sup_{w\in\mathbb R^d}\sum_{j_{1},\dots,j_{v}=1}^{d}\sup_{y\in R(0,\epsilon\sigma_*\eta)}|\partial_{j_{1},\dots,j_{v}}\rho^{A,\phi,\epsilon\sigma_*,I_d}(w +y)|.
\end{align*}
Therefore, in what follows, we set $\epsilon=1$ and $\Sigma=I_d$. 

Next, we prepare some notation. Let
\[
\mathbb{R}_{-}^{d}=\{w\in\mathbb{R}^{d}\colon w\leq0\}.
\]
Then we set $\tilde m=\tilde m^\phi:=m^{\mathbb R_-^d,\phi}$ and $\tilde\rho=\tilde\rho^{\phi}:=\rho^{\mathbb R_-^d,\phi,1,I_d}$. That is,
\ba{
\tilde m(w)=g^\phi\left(\max_{1\leq j\leq d}w_j\right)\quad
\text{and}\quad
\tilde\rho(w)=\Ep\tilde m(w+Z)\quad
\text{for }w\in\mathbb R^d. 
}
Also, for all $j=1,\dots,d$, let $\pi_{j}\colon\mathbb{R}^{d}\to\mathbb{R}$
be the function defined by
\[
\pi_{j}(w)=1\left\{ j=\arg\max_{1\leq k\leq d}w_{k}\right\} ,\quad w\in\mathbb{R}^{d},
\]
where $\arg\max_{1\leq k\leq d}w_{k}$ is equal to the smallest $l=1,\dots,d$
such that $w_{l}=\max_{1\leq k\leq d}w_{k}$. Here, it is useful to
note that the functions $\pi_{j}(\cdot)$ satisfy
\begin{equation}
\sum_{j=1}^{d}\pi_{j}(w)=1\label{eq: pi sum}
\end{equation}
for all $w\in\mathbb{R}^{d}$. Finally, let $\Psi\colon\mathbb{R}^{d}\to\mathbb{R}$
be the function defined by
\[
\Psi(w)=\int_{-\infty}^{0}\prod_{j=1}^{d}\Phi_{1}(t-w_{j})dt,\quad w\in\mathbb{R}^{d}.
\]
It is straightforward to check that
\begin{equation}
\partial_{j_{1},\dots, j_{u}}\Psi(w)=(-1)^{u}\int_{\mathbb{R}_{-}^{d}}\pi_{j_{1}}(s)\partial_{j_{2},\dots,j_{u}}\varphi\left(s-w\right)ds\label{eq: Phi reduction}
\end{equation}
for all $u=1,\dots,v$, $j_{1},\dots,j_{u}=1,\dots,d$, and $w\in\mathbb{R}^{d}$. Indeed, for $u=1$, we have
\begin{align*}
\partial_{j_1}\Psi(w) 
&= -\int_{-\infty}^0\varphi_1(t-w_{j_1})\prod_{j:j\neq j_1}\Phi_1(t - w_j)dt\\
&=-\int_{-\infty}^0\varphi_1(t-w_{j_1})\left(\prod_{j:j\neq j_1}\int_{-\infty}^0 \varphi_1(s_j - w_j)ds_j \right)dt
 =-\int_{\mathbb R_{-}^{d}}\pi_{j_1}(s)\varphi(s-w)ds,
\end{align*}
yielding \eqref{eq: Phi reduction}, and for $u\geq 2$, \eqref{eq: Phi reduction} follows immediately from the case $u=1$.

For the rest of the proof, we proceed in three steps. 
In the first step, we prove that
\begin{align}
&\sup_{A\in\mathcal R}\sup_{w\in\mathbb{R}^{d}}\sum_{j_1,\dots,j_v=1}^d\sup_{y\in R(0,\eta)}|\partial_{j_1,\dots,j_v}\rho(w+y)|\nonumber\\
&\qquad\lesssim\phi+\sum_{u=1}^v(\log d)^{(v-u)/2}\sup_{w\in\mathbb{R}^{d}}\sum_{(j_{1},\dots,j_{u})\in\mathcal{J}^u(d)}\sup_{y\in R(0,\eta)}|\partial_{j_{1}\dots j_{u}}\tilde\rho(w+y)|.\label{eq: Step 0}
\end{align}
In the second step, we prove that
\begin{align}
&\sup_{w\in\mathbb{R}^{d}}\sum_{j_1,\dots,j_v=1}^d\sup_{y\in R(0,\eta)}|\partial_{j_{1}\dots j_{v}}\tilde\rho(w+y)|
\nonumber\\
&\qquad\leq2\phi\sup_{w\in\mathbb{R}^{d}}\sum_{j_1,\dots,j_v=1}^d\sup_{y\in R(0,\eta)}|\partial_{j_{1}\dots j_{v}}\Psi(w+y)|.\label{eq: Step 1}
\end{align}
In the third step, we prove that
\begin{equation}
\sup_{w\in\mathbb{R}^{d}}\sum_{j_{1},\dots,j_{v}=1}^d\sup_{y\in R(0,\eta)}|\partial_{j_{1}\dots j_{v}}\Psi(w+y)|\lesssim (\log d)^{(v-1)/2}.\label{eq: Step 2}
\end{equation}
Combining these steps, with $u$ replacing $v$ in \eqref{eq: Step 1} and \eqref{eq: Step 2}, gives the asserted claim of the lemma. 

We will use the following elementary identity in the first step. 
\begin{lemma}\label{m-id}
For any random vector $W$ in $\mathbb{R}^d$ and $A=\prod_{j=1}^d(a_j,b_j]\in\mathcal R$, we have
\ba{
\Ep m(W)=\Ep m^{A,\phi}(W)=\phi\int_0^{\phi^{-1}}\Pr(W\in A^s)ds.
}
\end{lemma}

\begin{proof}
For any random variable $\tau$, we have, with $\Pr^\tau$ being the law of $\tau$,
\ba{
\Ep [g(\tau)]&=\int_{\mathbb{R}} g(t)\Pr^\tau(dt)
=-\int_{\mathbb{R}}\left(\int_t^{\infty} g'(s)ds\right)\Pr^\tau(dt)\\
&=-\int_{-\infty}^\infty\left(\int_{\mathbb{R}}1_{(-\infty,s]}(t)\Pr^\tau(dt)\right)g'(s)ds\\
&=-\int_{-\infty}^\infty\Pr(\tau\leq s)g'(s)ds
=\phi\int_0^{\phi^{-1}}\Pr(\tau\leq s)ds.
}
Applying this identity with $\tau=\max_{1\leq j\leq d}[(W_j-b_j)\vee[(a_j-W_j)]$, we obtain 
\ba{
\Ep[m(W)]=\phi\int_0^{\phi^{-1}}\Pr\left(\max_{1\leq j\leq d}[(W_j-b_j)\vee[(a_j-W_j)]\leq s\right)ds
=\phi\int_0^{\phi^{-1}}\Pr(W\in A^s)ds.
}
This completes the proof. 
\end{proof}

\medskip
\textbf{Step 1}. Here, we prove (\ref{eq: Step 0}). 
First, note that 
\ba{
&\sup_{A\in\mathcal R}\sup_{w\in\mathbb{R}^{d}}\sum_{j_{1},\dots,j_{v}=1}^{d}\sup_{y\in R(0,\eta)}|\partial_{j_{1},\dots,j_{v}}\rho(w+y)|\\
&\qquad\lesssim \sum_{u=1}^v\sum_{(\nu_1,\dots,\nu_u)\in\mathcal{N}^u(v)}\sup_{A\in\mathcal R}\sup_{w\in\mathbb{R}^{d}}\sum_{(j_1,\dots,j_u)\in\mathcal{J}^u(d)}\sup_{y\in R(0,\eta)}|\partial_{j_{1}}^{\nu_1}\cdots\partial_{j_{u}}^{\nu_u}\rho(w+y)|.
}
For each $u=1,\dots,v$, the cardinality of the set $\mathcal{N}^u(v)$ is bounded by a constant depending only on $v$. Therefore, it suffices to show that
\ban{
&\sup_{A\in\mathcal R}\sup_{w\in\mathbb{R}^{d}}\sum_{(j_1,\dots,j_u)\in\mathcal{J}^u(d)}\sup_{y\in R(0,\eta)}|\partial_{j_{1}}^{\nu_1}\cdots\partial_{j_{u}}^{\nu_u}\rho(w+y)|
\nonumber\\
&\qquad\lesssim\phi+(\log d)^{(v-u)/2}\sup_{w\in\mathbb{R}^{d}}\sum_{(j_{1},\dots,j_{u})\in\mathcal{J}^u(d)}\sup_{y\in R(0,\eta)}|\partial_{j_{1}\dots j_{u}}\tilde\rho(w+y)|\label{step0:aim}
}
for any (fixed) $u\in\{1,\dots,v\}$ and $(\nu_1,\dots,\nu_u)\in\mathcal{N}^u(v)$.

Let $A=\prod_{j=1}^d(a_j,b_j]\in\mathcal R$ be fixed. Using Lemma \ref{m-id}, we can rewrite $\rho(w)$ as 
\ba{
\rho(w)&=\phi\int_0^{\phi^{-1}}\Pr(w+Z\in A^s)ds
=\phi\int_0^{\phi^{-1}}\left\{\int_{\mathbb R^d}1_{A^s}(w+z)\varphi(z)dz\right\}ds\\
&=\phi\int_0^{\phi^{-1}}\left\{\int_{A^s}\varphi(z-w)dz\right\}ds.
}
Thus we have
\ben{\label{rho-deriv}
\partial_{j_1}^{\nu_1}\cdots\partial_{j_u}^{\nu_u}\rho(w)=(-1)^v\phi\int_0^{\phi^{-1}}\left\{\int_{A^s}\partial_{j_1}^{\nu_1}\cdots\partial_{j_u}^{\nu_u}\varphi(z-w)dz\right\}ds.
} 
Next, we have for any $w\in \mathbb R^d$ and $s\in[0,\phi^{-1}]$ 
\bm{
\left|\int_{A^s}\partial_{j_1}^{\nu_1}\cdots\partial_{j_u}^{\nu_u}\varphi(z-w)dz\right|\\
=\left(\prod_{q=1}^u\left|h_{\nu_q}(b^w_{j_q}+s)-h_{\nu_q}(a^w_{j_q}-s)\right|\right)
\prod_{k:k\neq j_1,\dots,j_u}\left\{\Phi_1(b^w_k+s)-\Phi_1(a^w_k-s)\right\},
}
where $a^w_j:=a_j-w_j$ and $b^w_j:=b_j-w_j$. 
Then, by \eqref{h-bound},
\bm{
\left|\int_{A^s}\partial_{j_1}^{\nu_1}\cdots\partial_{j_u}^{\nu_u}\varphi(z-w)dz\right|
\leq\left(\prod_{q=1}^u\left(\tilde h_{\nu_q}(|b^w_{j_q}+s|)+\tilde h_{\nu_q}(|-a^w_{j_q}+s|)\right)\right)\\
\times\prod_{k:k\neq j_1,\dots,j_u}\left\{\Phi_1(b^w_k+s)+\Phi_1(-a^w_k+s)-1\right\},
}
where we also use the identity $1-\Phi_1(t)=\Phi_1(-t)$. 
Now we set
\[
r^w_j:=b^w_j\wedge(-a^w_j),\qquad j=1,\dots,d.
\] 
Then, for any $j$, 
\ben{\label{Phi-bound}
\Phi_1(b^w_j+s)+\Phi_1(-a^w_j+s)-1
\leq\min\{\Phi_1(b^w_j+s),\Phi_1(-a^w_j+s)\}
=\Phi_1(r^w_j+s).
}
Also, 
\ben{\label{ba-small}
|b^w_j+s|\wedge|-a^w_j+s|\geq|r^w_j+s|.
}
In fact, if $b^w_j\geq0$ and $a^w_j\leq0$, then $|b^w_j+s|=b^w_j+s$, $|-a^w_j+s|=-a^w_j+s$ and $|r^w_j+s|=b^w_j\wedge(-a^w_j)+s$, so \eqref{ba-small} is evident. Otherwise, we have $b^w_j<0$ or $a^w_j>0$. In the first case, we have $-a^w_j>-b^w_j>0$, so $r^w_j=b^w_j$ and
\ba{
|-a^w_j+s|
=-a^w_j+s
\geq\max\{b^w_j+s,-b^w_j-s\}
=|b^w_j+s|.
}
Hence \eqref{ba-small} holds true. 
%In the second case, we have $b^w_j>a^w_j>0$, so $r^w_j=-a^w_j$ and
%\ba{
%|b^w_j+s|
%=b^w_j+s
%\geq\max\{-a^w_j+s,a^w_j-s\}
%=|-a^w_j+s|.
%}
%Hence \eqref{ba-small} holds true.
We can similarly prove \eqref{ba-small} in the second case. Combining \eqref{Phi-bound} and \eqref{ba-small} with \eqref{h-decrease}, we obtain
\ban{
&\left|\int_{A^s}\partial_{j_1}^{\nu_1}\cdots\partial_{j_u}^{\nu_u}\varphi(z-w)dz\right|
\nonumber\\
&\qquad\leq2^u\left(\prod_{q=1}^u\tilde h_{\nu_q}(|r^w_{j_q}+s|)\right)\prod_{k:k\neq j_1,\dots,j_u}\Phi_1(r^w_k+s)
\nonumber\\
&\qquad\lesssim \left(\prod_{q=1}^u\{1+|r^w_{j_q}+s|^{\nu_q-1}\}\varphi_1(r^w_{j_q}+s)\right)\prod_{k:k\neq j_1,\dots,j_u}\Phi_1(r^w_k+s).\label{phi-deriv-bound}
}
Let
\[
\mathbb S:=\left\{s\in[0,\phi^{-1}]:\max_{1\leq q\leq u}|r^w_{j_q}+s|\leq\sqrt{8v^2\log d}\right\},\qquad
\mathbb S^c:=[0,\phi^{-1}]\setminus\mathbb S.
\]
Then, we have by \eqref{phi-deriv-bound}
\ba{
&\int_{\mathbb S}\left|\int_{A^s}\partial_{j_1}^{\nu_1}\cdots\partial_{j_u}^{\nu_u}\varphi(z-w)dz\right|ds\\
&\qquad\lesssim(\log d)^{(v-u)/2}\int_0^{\phi^{-1}}\left(\prod_{q=1}^u\varphi_1(r^w_{j_q}+s)\right)\prod_{k:k\neq j_1,\dots,j_u}\Phi_1(r^w_k+s)ds
}
and
\ba{
\int_{\mathbb S^c}\left|\int_{A^s}\partial_{j_1}^{\nu_1}\cdots\partial_{j_u}^{\nu_u}\varphi(z-w)dz\right|ds
&\lesssim\int_{\mathbb S^c}\exp\left(-\frac{1}{4}\max_{1\leq q\leq u}(r^w_{j_q}+s)^2\right)ds\\
&\leq d^{-v}\int_{\mathbb S^c}\exp\left(-\frac{1}{8}\max_{1\leq q\leq u}(r^w_{j_q}+s)^2\right)ds\\
&\leq d^{-v}\int_{\mathbb S^c}\exp\left(-\frac{1}{8}(r^w_{j_1}+s)^2\right)ds\lesssim d^{-v}.
}
%Since $\max_{1\leq q\leq u}|r^w_{j_q}+s|\geq|s-\min_{1\leq q\leq u}|r^w_{j_q}||$ for all $s\geq0$, we obtain
%\ba{
%&\int_{\mathbb S^c}\left|\int_{A^s}\partial_{j_1}^{\nu_1}\cdots\partial_{j_u}^{\nu_u}\varphi(z-w)dz\right|ds\\
%&\lesssim d^{-v}\int_{\mathbb R}\exp\left(-\frac{1}{8}\left(s-\min_{1\leq q\leq u}|r^w_{j_q}-\eta|\right)^2\right)ds
%\lesssim d^{-v}.
%}
Combining these bounds with \eqref{rho-deriv}, we obtain for any $w\in\mathbb R^d$
\ba{
&\sup_{y\in R(0,\eta)}|\partial_{j_1}^{\nu_1}\cdots\partial_{j_u}^{\nu_u}\rho(w+y)|\\
&\lesssim (\log d)^{(v-u)/2}\sup_{y\in R(0,\eta)}\phi\int_0^{\phi^{-1}}\left(\prod_{q=1}^u\varphi_1(r^{w+y}_{j_q}+s)\right)\prod_{k:k\neq j_1,\dots,j_u}\Phi_1(r^{w+y}_k+s)ds
+\phi d^{-v}.
}
Now, since $|r^{w+y}_j-r^w_j|\leq|y_j|\leq\eta$ for all $j=1,\dots,d$ and $y\in R(0,\eta)$, we deduce
\ba{
&\sup_{y\in R(0,\eta)}|\partial_{j_1}^{\nu_1}\cdots\partial_{j_u}^{\nu_u}\rho(w+y)|\\
&\lesssim (\log d)^{(v-u)/2}\sup_{y\in R(0,\eta)}\phi\int_0^{\phi^{-1}}\left(\prod_{q=1}^u\varphi_1(r^{w}_{j_q}-y_{j_q}+s)\right)\prod_{k:k\neq j_1,\dots,j_u}\Phi_1(r^{w}_k-y_k+s)ds
+\phi d^{-v}\\
&=(\log d)^{(v-u)/2}\sup_{y\in R(0,\eta)}\phi\int_0^{\phi^{-1}}\left\{\int_{(\mathbb R_-^d)^s}\partial_{j_1\dots j_u}\varphi(z+r^w-y)dz\right\}ds
+\phi d^{-v},
}
where $r^w:=(r^w_1,\dots,r^w_d)^T$. Hence we conclude by \eqref{rho-deriv}
\be{
\sup_{y\in R(0,\eta)}|\partial_{j_1}^{\nu_1}\cdots\partial_{j_u}^{\nu_u}\rho(w+y)|
\lesssim (\log d)^{(v-u)/2}\sup_{y\in R(0,\eta)}|\partial_{j_1\dots j_u}\tilde\rho(-r^w+y)|+\phi d^{-v}.
}
This gives \eqref{step0:aim} and hence the asserted claim of this step.

\medskip
\textbf{Step 2}. Here, we prove (\ref{eq: Step 1}). Fix
any $w\in\mathbb{R}^{d}$ and $j_{1},\dots,j_{v}=1,\dots,d$, and
observe that
\[
\partial_{j_{1}}\tilde\rho(w)=\int_{\mathbb{R}^{d}}\partial_{j_{1}}\tilde m(w+z)\varphi(z)dz=\int_{\mathbb{R}^{d}}\partial_{j_{1}}\tilde m(s)\varphi\left(s-w\right)ds,
\]
where the second equality holds by the change of variables $z\mapsto s=w+ z$.
Therefore,
\begin{align*}
\partial_{j_{1},\dots,j_{v}}\tilde \rho(w) & =(-1)^{v-1}\int_{\mathbb{R}^{d}}\partial_{j_1}\tilde m(s)\partial_{j_{2},\dots,j_{v}}\varphi\left(s-w\right)ds\\
 & =(-1)^{v-1}\int_{\mathbb{R}^{d}}\partial_{j_1}\tilde m(w+z)\partial_{j_{2},\dots,j_{v}}\varphi(z)dz,
\end{align*}
where the second equality holds by the reverse change of variables
$s\mapsto z=s-w$. In addition,
\[
\partial_{j_1}\tilde m(w+z)=g'\left(\max_{1\leq j\leq d}(w_{j}+z_{j})\right)\pi_{j_{1}}(w+z)
\]
for almost all $z$ with respect to the Lebesgue measure on $\mathbb{R}^{d}$. Thus, given that
\[
g'(t)=\begin{cases}
\phi & \text{if }t\in(0,1/\phi),\\
0 & \text{if }t\notin(0,1/\phi),
\end{cases}
\]
denoting
\[
A^w_{1}=\left\{ z\in\mathbb{R}^{d}\colon w+ z\leq0\right\} ,\quad A^w_{2}=\left\{ z\in\mathbb{R}^{d}\colon w+ z\leq1/\phi\right\} ,
\]
we have
\[
\partial_{j_{1},\dots,j_{v}}\tilde\rho(w)=(-1)^{v-1}\phi\int_{A^w_{2}\setminus A^w_{1}}\pi_{j_{1}}(w+z)\partial_{j_{2},\dots,j_{v}}\varphi(z)dz,
\]
and so
\begin{align*}
|\partial_{j_{1},\dots,j_{v}}\tilde\rho(w)| & \leq\phi\left|\int_{A^w_{2}}\pi_{j_{1}}(w+ z-1/\phi)\partial_{j_{2},\dots,j_{v}}\varphi(z)dz\right|\\
 & \quad+\phi\left|\int_{A^w_{1}}\pi_{j_{1}}(w+ z)\partial_{j_{2},\dots,j_{v}}\varphi(z)dz\right|,
\end{align*}
where we used $\pi_{j_{1}}(w+ z)=\pi_{j_{1}}(w+ z-1/\phi)$. Therefore,
\begin{align}
&\sup_{w\in\mathbb{R}^{d}}\sum_{j_{1},\dots,j_{v}=1}^d\sup_{y\in R(0,\eta)}|\partial_{j_{1}\dots j_{v}}\tilde\rho(w+y)|\nonumber\\
&\quad \leq2\phi\sup_{w\in\mathbb{R}^{d}}\sum_{j_{1},\dots,j_{v}=1}^d\sup_{y\in R(0,\eta)}\left|\int_{A^{w+y}_{1}}\pi_{j_{1}}(w+y+z)\partial_{j_{2},\dots,j_{v}}\varphi(z)dz\right|.\label{eq: step 1 ineq 1}
\end{align}
Moreover, 
\begin{align}
&\int_{A^{w+y}_{1}}\pi_{j_{1}}(w+y+ z)\partial_{j_{2},\dots,j_{v}}\varphi(z)dz 
\nonumber\\
& \qquad=\int_{\mathbb{R}_{-}^{d}}\pi_{j_{1}}( s)\partial_{j_{2},\dots,j_{v}}\varphi(s-(w+y))ds\nonumber\\
% & \qquad=\int_{\mathbb{R}_{-}^{d}}\pi_{j_{1}}(s)\partial_{j_{2},\dots,j_{v}}\varphi(s-(w+y)/\epsilon)ds\nonumber\\
 & \qquad =(-1)^{v}\partial_{j_{1},\dots,j_{v}}\Psi(w+y),\label{eq: step 1 ineq 2}
\end{align}
where the first equality holds by the change of variables $z\mapsto s=z+(w+y)$
and the third by (\ref{eq: Phi reduction}). Combining \eqref{eq: step 1 ineq 1} and \eqref{eq: step 1 ineq 2} gives the asserted claim of this step. 

\medskip
\textbf{Step 3}. Here, we prove (\ref{eq: Step 2}). To
do so, we proceed by induction on $v$. For $v=1$, we have for all $w\in\mathbb R^d$, $y\in R(0,\eta)$, and $j=1,\dots,d$ that
$$
|\partial_j\Psi(w+y)|=\int_{-\infty}^0 \left(\prod_{l\colon l\neq j}\Phi_1(t-w_l-y_l)\right)\varphi_1(t-w_j-y_j)dt.
$$
To bound the integral on the right-hand side here, consider the partition
$$
(-\infty,0] = \mathbb T_j\cup \mathbb T_j^c,
$$
where
$$
\mathbb T_j = \Big\{t\in(-\infty,0]\colon |t-w_j|\leq (2\log d)^{1/2}+\eta\Big\},\quad \mathbb T_j^c = (-\infty,0]\setminus \mathbb T_j.
$$
Then
\begin{align*}
&\int_{\mathbb T_j} \left(\prod_{l\colon l\neq j}\Phi_1(t-w_l-y_l)\right)\varphi_1(t-w_j-y_j)dt\\
&\qquad \leq \int_{\mathbb T_j} \left(\prod_{l\colon l\neq j}\Phi_1(t-w_l+\eta)\right)\varphi_1(t-w_j+\eta)\frac{\varphi_1(t-w_j-y_j)}{\varphi_1(t-w_j+\eta)}dt\\
&\qquad \lesssim \int_{\mathbb T_j} \left(\prod_{l\colon l\neq j}\Phi_1(t-w_l+\eta)\right)\varphi_1(t-w_j+\eta)dt\\
&\qquad \leq \int_{-\infty}^0 \left(\prod_{l\colon l\neq j}\Phi_1(t-w_l+\eta)\right)\varphi_1(t-w_j+\eta)dt = - \partial_j\Psi(w-\eta),
\end{align*}
where the second inequality holds for all $y\in R(0,\eta)$ because $\eta\leq K/\sqrt{\log d}$. Also,
\begin{align*}
&\int_{\mathbb T_j^c} \left(\prod_{l\colon l\neq j}\Phi_1(t-w_l-y_l)\right)\varphi_1(t-w_j-y_j)dt\\
&\qquad \leq \int_{\mathbb T_j^c}\varphi_1(t-w_j-y_j)dt \lesssim \int_{(2\log d)^{1/2}}^{+\infty}\varphi_1(t)dt \lesssim \varphi_1\left( (2\log d)^{1/2} \right)\lesssim d^{-1}.
\end{align*}
Combining these bounds, we obtain
\begin{align*}
\sum_{j=1}^d\sup_{y\in R(0,\eta)}|\partial_j\Psi(w+y)|
& \lesssim 1 - \sum_{j=1}^d \partial_j\Psi(w-\eta)\\
& = 1 + \sum_{j=1}^d\int_{\mathbb R^d_{-}}\pi_j(s)\varphi(s-w+\eta)ds\\
& = 1 + \int_{\mathbb R^d_{-}}\varphi(s-w+\eta)ds\leq 2,
\end{align*}
where the last line follows from (\ref{eq: pi sum}). This gives \eqref{eq: Step 2} for $v=1$.

Now, fix $v\geq 2$. By induction, we can assume that
\begin{equation}
\max_{1\leq u\leq v-1}\sup_{w\in\mathbb{R}^{d}}\sum_{j_{1},\dots,j_{u}=1}^{d}\sup_{y\in R(0,\eta)}\left|\partial_{j_{1},\dots,j_{u}}\Psi(w+y)\right|\lesssim(\log d)^{(u-1)/2}.\label{eq: induction}
\end{equation}
Also, define
\[
\mathcal{J}=\mathcal{J}^v(d)=\left\{ (j_{1},\dots,j_{v})\in\left\{ 1,\dots,d\right\} ^{v}\colon\text{all }j_{1},\dots,j_{v}\text{ are different}\right\}
\]
and 
$
\mathcal{J}^{c}=\left\{ 1,\dots,d\right\} ^{v}\setminus\mathcal{J}.
$
Like in the $v=1$ case, we can check that for all $w\in\mathbb R^d$, $y\in R(0,\eta)$, and $(j_1,\dots,j_v)\in\mathcal J$, we have
$$
|\partial_{j_1\dots j_v}\Psi(w+y)|\lesssim d^{-v} + (-1)^v \partial_{j_1\dots j_v}\Psi(w-\eta).
$$
Therefore,
\begin{align*}
&\sum_{j_{1},\dots,j_{v}=1}^{d}\sup_{y\in R(0,\eta)}|\partial_{j_{1}\dots j_{v}}\Psi(w+y)|\\
 & \lesssim 1 + \sum_{(j_{1},\dots,j_{v})\in\mathcal{J}^{c}}\sup_{y\in R(0,\eta)}|\partial_{j_{1}\dots j_{v}}\Psi(w+y)|+(-1)^v\sum_{j_{1},\dots,j_{v}=1}^{d}\partial_{j_{1}\dots j_{v}}\Psi(w-\eta).
\end{align*}
Here, for all $w\in\mathbb R^d$,
\begin{align*}
\left|\sum_{j_{1},\dots,j_{v}=1}^{d}\partial_{j_{1}\dots j_{v}}\Psi(w)\right| & =\sum_{j_{1},\dots,j_{v}=1}^{d}\int_{\mathbb{R}_{-}^{d}}\pi_{j_{1}}(s)\partial_{j_{2},\dots,j_{v}}\varphi\left(s-w\right)ds\\
 & =\sum_{j_{2},\dots,j_{v}=1}^{d}\int_{\mathbb{R}_{-}^{d}}\partial_{j_{2},\dots,j_{v}}\varphi\left(s-w\right)ds\lesssim(\log d)^{(v-1)/2}
\end{align*}
by (\ref{eq: pi sum}) and Lemma 2.2 in \cite{FK20}. 

Hence, it remains to prove that
\begin{equation}\label{eq: remaining piece}
\sum_{(j_{1},\dots,j_{v})\in\mathcal{J}^{c}}\sup_{y\in R(0,\eta)}|\partial_{j_{1}\dots j_{v}}\Psi(w+y)| \lesssim (\log d)^{(v-1)/2}.
\end{equation}
To do so, for all $(j_{1},\dots,j_{v})\in\mathcal{J}^{c}$,
let $N(j_{1},\dots,j_{v})$ denote the number of different indices
among $v$ indices $j_{1},\dots,j_{v}$. Then
\[
\mathcal{J}^{c}=\mathcal{J}_{1}\cup\dots\cup\mathcal{J}_{v-1},
\]
where
\[
\mathcal{J}_{u}=\left\{ (j_{1},\dots,j_{v})\in\mathcal{J}^{c}\colon N(j_{1},\dots,j_{v})=u\right\} ,\quad u=1,\dots,v-1.
\]
Thus,
\[
\sum_{(j_{1},\dots,j_{v})\in\mathcal{J}^{c}}\sup_{y\in R(0,\eta)}|\partial_{j_{1}\dots j_{v}}\Psi(w+y)|=\sum_{u=1}^{v-1}\sum_{(j_{1},\dots,j_{v})\in\mathcal{J}_{u}}\sup_{y\in R(0,\eta)}|\partial_{j_{1}\dots j_{v}}\Psi(w+y)|.
\]
Next, fix any $u=1,\dots,v-1$ and consider the corresponding sum on the right-hand side of the equality above.
Fix any $(j_{1},\dots,j_{v})\in\mathcal{J}_{u}$. By the definition of
$\mathcal{J}_{u}$, there are exactly $u$ different indices among
$v$ indices $j_{1},\dots,j_{v}$. Denote them by $o_{1},\dots,o_u$ and assume that they appear $k_{1},\dots,k_u$ times, respectively, where $k_{1}+\dots+k_{u}=v$. Then, denoting
$$
o=(o_1,\dots,o_u),\quad\mathcal J^o=\{1,\dots,d\}\setminus\{o_1,\dots,o_u\},
$$
we have for all $w\in\mathbb R^d$ that
\begin{align*}
|\partial_{j_{1}\dots j_{v}}\Psi(w)| &= \left| \int_{-\infty}^0 \prod_{j\in\mathcal J^o}\Phi_1(t-w_j)\prod_{i=1}^u \partial^{k_i}\Phi_1(t-w_{o_i})dt \right|\\
&\lesssim \int_{-\infty}^0 \prod_{j\in\mathcal J^o}\Phi_1(t-w_j)\prod_{i=1}^u \left((|t-w_{o_i}|^{k_i-1}+1)\varphi_1(t-w_{o_i})\right)dt.
\end{align*}
To bound the integral on the right-hand side here, consider the partition
$$
(-\infty,0] = \mathcal T_0 \cup \mathcal T_1 \cup \dots \cup \mathcal T_u,
$$
where
$$
\mathcal T_0 = \left\{t\in(-\infty,0]\colon \vee_{i=1}^u |t-w_{o_i}|\leq (4v^2\log d)^{1/2}\right\}
$$
and
$$
\mathcal T_i = \left\{t\in(-\infty,0]\setminus\mathcal T_0\colon i=\arg\max_{1\leq k\leq u}|t-w_{o_k}|\right\},\quad i=1,\dots,u.
$$
Then
\begin{multline*}
\int_{\mathcal T_0} \prod_{j\in\mathcal J^o}\Phi_1(t-w_j)\prod_{i=1}^u \left((|t-w_{o_i}|^{k_i-1}+1)\varphi_1(t-w_{o_i})\right)dt 
\lesssim (\log d)^{(v-u)/2}|\partial_{o_1\dots o_u}\Psi(w)|
\end{multline*}
and, for all $i=1,\dots,u$,
\begin{align*}
&\int_{\mathcal T_i} \prod_{j\in\mathcal J^o}\Phi_1(t-w_j)\prod_{i=1}^u \left((|t-w_{o_i}|^{k_i-1}+1)\varphi_1(t-w_{o_i})\right)dt \\
&\qquad \lesssim \int_{\mathcal T_i}|t-w_{o_i}|^{v-u}\varphi_1(t-w_{o_i})dt \lesssim \int_{(4v^2\log d)^{1/2}}^{+\infty} t^v\varphi_1(t)dt \\
&\qquad \lesssim \int_{(4v^2\log d)^{1/2}}^{+\infty}\exp(v\log t - t^2/2)dt \lesssim \int_{(4v^2\log d)^{1/2}}^{+\infty}\exp(-t^2/4)dt \\
&\qquad = \sqrt2\int_{(2v^2\log d)^{1/2}}^{+\infty}\exp(-t^2/2)dt \lesssim \varphi_1\left((2v^2\log d)^{1/2}\right) \lesssim d^{-{v}}.
\end{align*}
Combining these bounds, we obtain
\begin{align*}
&\sum_{(j_1,\dots,j_v)\in\mathcal J^u}\sup_{y\in R(0,\eta)}|\partial_{j_{1}\dots j_{v}}\Psi(w+y)|\\
&\qquad\quad  \lesssim 1 + (\log d)^{(v-u)/2}\sum_{o_1,\dots,o_u=1}^d\sup_{y\in R(0,\eta)}|\partial_{o_1\dots o_u}\Psi(w+y)|\\
& \qquad \quad\lesssim 1 + (\log d)^{(v-u)/2}(\log d)^{(u-1)/2}\lesssim (\log d)^{(v-1)/2},
\end{align*}
where the third line follows from \eqref{eq: induction}. Therefore, given that $u=1,\dots,v-1$ here is arbitrary, \eqref{eq: remaining piece} follows, which gives the asserted claim of this step and completes the proof of the lemma. 
\end{proof}

\begin{proof}[Proof of Lemma \ref{lemma:fk2.2}]
As in the proof of Lemma \ref{lem: smoothing inequality mixed 1}, it suffices to consider the case with $\epsilon=1$ and $\Sigma=I_d$, which is what we do below. 
Then, similarly to the proof of \eqref{rho-deriv}, we obtain
\ba{
\partial_{j_1,\dots,j_v}\rho(w+y)
&=(-1)^v\phi\int_0^{\phi^{-1}}\left\{\int_{A^s}\partial_{j_1,\dots,j_v}\varphi(z-w-y)dz\right\}ds\\
&=(-1)^v\phi\int_0^{\phi^{-1}}\left\{\int_{A^s-w}\partial_{j_1,\dots,j_v}\varphi(z-y)dz\right\}ds.
}
Note that $A^s-w\in\mathcal R$. Thus, combining this identity with Lemma 2.2 in \cite{FK20} gives the asserted claim.
%
%Using Lemma \ref{m-id}, we can rewrite $\rho (w)$ as
%\ba{
%\rho (w)&=\phi\int_0^{\phi^{-1}}\Pr(w+\epsilon Z\leq s)ds\\
%&=\phi\int_0^{\phi^{-1}}\left\{\int_{\mathbb{R}^d}1_{\mathbb{R}^d_-}(w-s+\epsilon z)\varphi(z)dz\right\}ds\\
%&=\phi\int_0^{\phi^{-1}}\left\{\int_{\mathbb{R}^d}1_{\mathbb{R}^d_-}(\epsilon z)\varphi(z-\epsilon^{-1}(w-s))dz\right\}ds\\
%&=\phi\int_0^{\phi^{-1}}\left\{\int_{\mathbb{R}^d_-}\varphi(z-\epsilon^{-1}(w-s))dz\right\}ds.
%}
%Thus, we have
%\begin{align}
%&\partial_{j_1,\dots, j_v}\rho (w) =\frac{(-1)^v\phi}{\epsilon^v}\int_0^{\phi^{-1}}\left\{\int_{\mathbb{R}^d_-}\partial_{j_1,\dots, j_v}\varphi(z-\epsilon^{-1}(w-s))dz\right\}ds.\label{rho-deriv}
%\end{align}
%Combining this identity with Lemma 2.2 in \cite{FK20} gives the asserted claim.
\end{proof}

\begin{proof}[Proof of Lemma \ref{lemma:vanish}]
As in the beginning of the proof of Lemma \ref{lem: smoothing inequality mixed 1}, define $\Sigma^1 = \Sigma - \sigma_*^2 I_d$ and $\Sigma^2 = \sigma_*^2 I_d$ and let $Z^1$ and $Z^2$ be independent random vectors in $\mathbb R^d$ such that $Z^1\sim N(0,\Sigma^1)$ and $Z^2\sim N(0,\Sigma^2)$. Then
\begin{align*}
&\sup_{w\in(A^{2\epsilon\kappa+\phi^{-1}}\setminus A^{-2\epsilon\kappa})^c}\sum_{j_{1},\dots,j_{v}=1}^{d}\sup_{y\in R(0,\epsilon\sigma_*\eta)}|\partial_{j_{1},\dots,j_{v}}\rho^{A,\phi,\epsilon,\Sigma}(w+y)| \\
&\qquad \leq \Ep\left[\sup_{w\in(A^{2\epsilon\kappa+\phi^{-1}}\setminus A^{-2\epsilon\kappa})^c}\sum_{j_{1},\dots,j_{v}=1}^{d}\sup_{y\in R(0,\epsilon\sigma_*\eta)}|\partial_{j_{1},\dots,j_{v}}\rho^{A,\phi,\epsilon\sigma_*,I_d}(w + \epsilon Z^1 +y)|\right].
\end{align*}
Also, by the union and Chernoff's bounds,
$$
\Pr\left( \|Z^1\|_{\infty} >\kappa \right)\leq 2de^{-\kappa^2/2}\leq 2de^{-(\kappa-\eta)^2/4}.
$$
Thus, by Lemma \ref{lemma:fk2.2},
\begin{align*}
&\Ep\left[\sup_{w\in(A^{2\epsilon\kappa+\phi^{-1}}\setminus A^{-2\epsilon\kappa})^c}\sum_{j_{1},\dots,j_{v}=1}^{d}\sup_{y\in R(0,\epsilon\sigma_*\eta)}|\partial_{j_{1},\dots,j_{v}}\rho^{A,\phi,\epsilon\sigma_*,I_d}(w + \epsilon Z^1 +y)|\right]\\
&\qquad \lesssim \sup_{w\in(A^{\epsilon\kappa+\phi^{-1}}\setminus A^{-\epsilon\kappa})^c}\sum_{j_{1},\dots,j_{v}=1}^{d}\sup_{y\in R(0,\epsilon\sigma_*\eta)}|\partial_{j_{1},\dots,j_{v}}\rho^{A,\phi,\epsilon\sigma_*,I_d}(w +y)| + \frac{(\log d)^{v/2}}{(\epsilon\sigma_*)^v}\times de^{-(\kappa-\eta)^2/4}.
\end{align*}
Further, using \eqref{rho-epsilon}, we obtain
\ba{
&\sup_{A\in\mathcal R}\sup_{w\in(A^{\epsilon\kappa+\phi^{-1}}\setminus A^{-\epsilon\kappa})^c}\sum_{j_{1},\dots,j_{v}=1}^{d}\sup_{y\in R(0,\epsilon\sigma_*\eta)}|\partial_{j_{1},\dots,j_{v}}\rho^{r,\phi,\epsilon\sigma_*,I_d}(w +y)| \\
&\qquad=\frac{1}{\epsilon^v}\sup_{A\in\mathcal R}\sup_{w\in(A^{\epsilon\kappa+\phi^{-1}}\setminus A^{-\epsilon\kappa})^c}\sum_{j_{1},\dots,j_{v}=1}^{d}\sup_{y\in R(0,\epsilon\sigma_*\eta)}|\partial_{j_{1},\dots,j_{v}}\rho^{(\epsilon\sigma_*)^{-1}A,\epsilon\sigma_*\phi,1,I_d}((w +y)/(\epsilon\sigma_*))| \\
&\qquad\leq\frac{1}{\epsilon^v}\sup_{A\in\mathcal R}\sup_{w\in(A^{\kappa+(\epsilon\sigma_*\phi)^{-1}}\setminus A^{-\kappa})^c}\sum_{j_{1},\dots,j_{v}=1}^{d}\sup_{y\in R(0,\eta)}|\partial_{j_{1},\dots,j_{v}}\rho^{A,\epsilon\sigma_*\phi,1,I_d}(w +y)|,
}
where we use the inequality $\kappa/\sigma_*\geq\kappa$ to deduce the last line. 
Combining these inequalities shows that the asserted claim for general $\Sigma$ and $\epsilon$ follows from the asserted claim for $\Sigma = I_d$ and $\epsilon=1$ with replacing $A^{2\epsilon\kappa+\phi^{-1}}$ and $A^{-\epsilon\kappa}$ by $A_1:=A^{\kappa+\phi^{-1}}$ and $A_2:=A^{-\kappa}$, respectively. In what follows, we therefore set $\Sigma=I_d$ and $\epsilon=1$. 
%Similarly, again as in the proof of Lemma \ref{lem: smoothing inequality mixed 1}, we set $r=0$.

Further, note that identity \eqref{rho-deriv} derived in the proof of Lemma \ref{lem: smoothing inequality mixed 1} did not rely on any specific assumptions of Lemma \ref{lem: smoothing inequality mixed 1}, and so remains valid under current assumptions. We will use this identity below.

Next, note that 
\ba{
&\sup_{w\in(A_1\setminus A_2)^c}\sum_{j_{1},\dots,j_{v}=1}^{d}\sup_{y\in R(0,\eta)}|\partial_{j_{1},\dots,j_{v}}\rho(w+y)|\\
&\qquad\lesssim \sum_{u=1}^v\sum_{(\nu_1,\dots,\nu_u)\in\mathcal{N}^u(v)}\sup_{w\in(A_1\setminus A_2)^c}\sum_{(j_1,\dots,j_u)\in\mathcal{J}^u(d)}\sup_{y\in R(0,\eta)}|\partial_{j_{1}}^{\nu_1}\cdots\partial_{j_{u}}^{\nu_u}\rho(w+y)|.
}
Further, by \eqref{rho-deriv}, 
\ba{
|\partial_{j_{1}}^{\nu_1}\cdots\partial_{j_{u}}^{\nu_u}\rho(w+y)|
\leq\phi\int_0^{\phi^{-1}}\left|\int_{A^s}\partial_{j_1}^{\nu_1}\cdots\partial_{j_u}^{\nu_u}\varphi(z-w-y)dz\right|ds.
}
For each $u=1,\dots,v$, the cardinality of the set $\mathcal{N}^u(v)$ is bounded by a constant depending only on $v$. Therefore, it suffices to show that
\ben{\label{vanish1}
\sup_{w\in A_1^c}\sum_{(j_1,\dots,j_u)\in\mathcal{J}^u(d)}\sup_{y\in R(0,\eta)}\left|\int_{A^s}\partial_{j_1}^{\nu_1}\cdots\partial_{j_u}^{\nu_u}\varphi(z-w-y)dz\right|\lesssim d^ve^{-(\kappa-\eta)^2/4}
}
and 
\ben{\label{vanish2}
\sup_{w\in A_2}\sum_{(j_1,\dots,j_u)\in\mathcal{J}^u(d)}\sup_{y\in R(0,\eta)}\left|\int_{A^s}\partial_{j_1}^{\nu_1}\cdots\partial_{j_u}^{\nu_u}\varphi(z-w-y)dz\right|\lesssim d^ve^{-(\kappa-\eta)^2/4}
}
for any (fixed) $s\in[0,\phi^{-1}]$ and $(\nu_1,\dots,\nu_u)\in\mathcal{N}_u(v)$ with $u\in\{1,\dots,v\}$.

For any $w\in\mathbb R^d$, we have
\ba{
I(w)&:=\sum_{(j_1,\dots,j_u)\in\mathcal{J}^u(d)}\sup_{y\in R(0,\eta)}\left|\int_{A^s}\partial_{j_1}^{\nu_1}\cdots\partial_{j_u}^{\nu_u}\varphi(z-w-y)dz\right|\\
&=\sum_{(j_1,\dots,j_u)\in\mathcal{J}^u(d)}\sup_{y\in R(0,\eta)}\left(\prod_{q=1}^u\left|h_{\nu_q}(b^w_{j_q}+s-y_{j_q})-h_{\nu_q}(a^w_{j_q}-s-y_{j_q})\right|\right)\\
&\qquad\times\prod_{k:k\neq j_1,\dots,j_u}\left\{\Phi_1(b^w_k+s-y_k)-\Phi_1(a^w_k-s-y_k)\right\},
}
where $a^w_j:=a_j-w_j$ and $b^w_j:=b_j-w_j$. Since 
\ba{
\Phi_1(b^w_k+s-y_k)-\Phi_1(a^w_k-s-y_k)
&=\Phi_1(b^w_k+s-y_k)+\Phi_1(-a^w_k+s+y_k)-1\\
&\leq\Phi_1(b^w_k+s-y_k)\wedge\Phi_1(-a^w_k+s+y_k)
}
and
$|h_\nu(t)|\lesssim(1+|t|^{\nu-1})e^{-t^2/2}\lesssim e^{-t^2/4}$ for all $t\in\mathbb R$, we obtain
\bmn{\label{I-bound}
I(w)
\lesssim\sum_{(j_1,\dots,j_u)\in\mathcal{J}^u(d)}\sup_{y\in R(0,\eta)}\left(\prod_{q=1}^u\left(e^{-(b^w_{j_q}+s-y_{j_q})^2/4}+e^{-(a^w_{j_q}-s-y_{j_q})^2/4}\right)\right)\\
\times\prod_{k:k\neq j_1,\dots,j_u}\Phi_1(b^w_k+s-y_k)\wedge\Phi_1(-a^w_k+s+y_k).
}

Now, if $w\in A_1^c$, there is an $l\in\{1,\dots,d\}$ such that $w_l>b_l+\kappa+\phi^{-1}$ or $w_l\leq a_l-\kappa-\phi^{-1}$. When the former holds, then for any $y\in R(0,\eta)$,
\[
a^w_l-s-y_l<b^w_l+s-y_l<-\kappa-\phi^{-1}+s-y_l\leq-\kappa+\eta<0.
\] 
When the latter holds, then for any $y\in R(0,\eta)$,
\[
b^w_l+s-y_l>a^w_l-s-y_l\geq\kappa+\phi^{-1}-s-y_l\geq\kappa-\eta>0.
\]
Hence, in either case, we have by \eqref{I-bound}
\ba{
I(w)&\lesssim \sum_{(j_1,\dots,j_u)\in\mathcal{J}^u(d)}e^{-(\kappa-\eta)^2/4}\vee\Phi_1(-\kappa+\eta)
\leq d^ve^{-(\kappa-\eta)^2/4}\vee\Phi_1(-\kappa+\eta).
}
Since $\Phi_1(-\kappa+\eta)\leq e^{-(\kappa-\eta)^2/2}$ by Chernoff's bound, we obtain \eqref{vanish1}. 

Next, if $w\in A_2$, $a_j+\kappa<w_j\leq b_j-\kappa$ for all $j=1,\dots,d$. Hence $b^w_j+s-y_j\geq\kappa-\eta>0$ and $-a^w_j+s+y_j>\kappa-\eta>0$. Thus, by \eqref{I-bound},
\ba{
I(w)
&\lesssim\sum_{(j_1,\dots,j_u)\in\mathcal{J}^u(d)}\prod_{q=1}^ue^{-(\kappa-\eta)^2/4}
\lesssim d^ve^{-(\kappa-\eta)^2/4}.
}
Hence we obtain \eqref{vanish2} and complete the proof of the lemma.
\end{proof}

\section{Auxiliary Lemmas}\label{sec: auxiliary lemmas}
\begin{lemma}\label{lem: stein version}
Let $Z=(Z_1,\dots,Z_d)^T$ be a centered Gaussian random vector in $\mathbb R^d$ with a non-singular covariance matrix $\Sigma = (\Sigma_{jk})_{j,k=1}^d$. Then for any $j=1,\dots,d$, $\epsilon>0$, $w\in\mathbb R^d$, and bounded and measurable $h\colon\mathbb R^d\to\mathbb R$,
$$
\Ep h(w+\epsilon Z)Z_j = \epsilon\sum_{k=1}^d \partial_k h_{\epsilon}(w)\Sigma_{jk},
$$
where $h_{\epsilon}\colon\mathbb R^d\to\mathbb R$ is given by $h_{\epsilon}(w)=\Ep h(w+\epsilon Z)$ for all $w\in\mathbb R^d$.
\end{lemma}
\begin{remark}
This lemma is a version of Stein's identity suitable for non-differentiable functions $h$. Although the lemma seems to be rather well known, we provide its proof below for reader's convenience. \qed
\end{remark}
\begin{proof}
Observe that the asserted claim for general $\epsilon>0$ follows from the asserted claim for $\epsilon = 1$ by rescaling of the vector $Z$. Therefore, we only consider the case $\epsilon = 1$.

Now, fix any $j=1,\dots,d$, $w\in\mathbb R^d$, and bounded and measurable $h\colon\mathbb R^d\to\mathbb R$. Then, denoting $A=\Sigma^{-1/2}$, so that $V=AZ\sim N(0,I_d)$, we have
\begin{align*}
h_{\epsilon}(w)
&=\Ep h(w+Z) = \Ep h(w+A^{-1}V)\\
&=\int h(w+A^{-1}v)\varphi(v)dv = |A|\int h(s)\varphi(A(s-w))ds,
\end{align*}
where $\varphi$ is the pdf of the standard normal distribution on $\mathbb R^d$ and $|A|$ is the determinant of $A$. Thus, differentiating under the integral, which is allowed by Corollary A.10 in \cite{D99}, for all $k=1,\dots d$,
\begin{align*}
\partial_k h_{\epsilon}(w) 
& =- |A|\int h(s)\sum_{l=1}^d A_{kl} \partial_l  \varphi(A(s-w)) ds\\
&= - \int h(w+A^{-1}v)\sum_{l=1}^d A_{kl} \partial_l  \varphi(v) dv
= \int h(w+A^{-1}v)\sum_{l=1}^d A_{kl} v_l \varphi(v) dv.
\end{align*}
Hence,
\begin{align*}
\sum_{k=1}^d\partial_k h_{\epsilon}(w) \Sigma_{jk} &= \int h(w+A^{-1}v)\sum_{l=1}^d \sum_{k=1}^d \Sigma_{jk} A_{kl} v_l \varphi(v) dv \\
&= \int h(w+A^{-1}v)\sum_{l=1}^d (\Sigma^{1/2})_{jl} v_l \varphi(v) dv = \Ep h(w + Z)Z_j,
\end{align*}
where the last equality follows from $Z=\Sigma^{1/2}V$. The asserted claim follows.
\end{proof}

\begin{lemma}\label{lem: rao}
Let $X$, $Y$, and $Z$ be independent random vectors in $\mathbb R^d$. Denote
$$
\zeta:=\sup_{r\in\mathbb R^d}\Big| \Pr(X\leq r) - \Pr(Y\leq r) \Big|\quad\text{and}\quad
 \gamma:= \sup_{r\in\mathbb R^d}\Big| \Pr(X+Z\leq r) - \Pr(Y+Z\leq r) \Big|
$$
and let $\epsilon >0$ be such that 
$
\alpha:=\Pr(Z\in R(0,\epsilon)) > 1/2.
$
Then
$$
\zeta \leq \frac{\gamma+\alpha\tau}{2\alpha-1},
$$
where
$$
\tau:=\sup_{r\in\mathbb R^d}\Big| \Pr(Y\leq r+2\epsilon) - \Pr(Y\leq r) \Big|.
$$
\end{lemma}
\begin{remark}
This result is an adaptation of Lemma 2.4 from \cite{FK20} with hopefully easier to follow notations and is a version of Lemma 11.4 from \cite{BR76}. We provide a proof here for reader's convenience.
\qed
\end{remark}
\begin{proof}
Note that
$$
\zeta = \max\left(\sup_{r\in\mathbb R^d}\Big( \Pr(X\leq r) - \Pr(Y\leq r) \Big),\sup_{r\in\mathbb R^d}\Big( \Pr(Y\leq r) - \Pr(X\leq r) \Big)\right)
$$
and consider the case
\begin{equation}\label{eq: case 1 ini smo}
\zeta = \sup_{r\in\mathbb R^d}\Big( \Pr(X\leq r) - \Pr(Y\leq r) \Big).
\end{equation}
In this case, for any $r\in\mathbb R^d$, we have
\begin{equation}\label{eq: i1-i2}
\Pr(X+Z\leq r + \epsilon) - \Pr(Y+Z\leq r + \epsilon) =\mathcal I_{1,r} + \mathcal I_{2,r},
\end{equation}
where
$$
\mathcal I_{1,r} = \Ep\Big[\Big(1\{X+Z\leq r + \epsilon\} - 1\{Y+Z\leq r+\epsilon\}\Big)1\{Z\in R(0,\epsilon)\}\Big],
$$
$$
\mathcal I_{2,r} = \Ep\Big[\Big(1\{X+Z\leq r + \epsilon\} - 1\{Y+Z\leq r+\epsilon\}\Big)1\{Z\notin R(0,\epsilon)\}\Big].
$$
Here, denoting $\zeta_r:=\Pr(X\leq r) - \Pr(Y\leq r)$, we have
\begin{align*}
\mathcal I_{1,r}
&\geq \Ep\Big[\Big(1\{X\leq r\} - 1\{Y\leq r+2\epsilon\}\Big)1\{Z\in R(0,\epsilon)\}\Big]\\
& = \Ep\Big[1\{X\leq r\} - 1\{Y\leq r+2\epsilon\}\Big]\Ep\Big[1\{Z\in R(0,\epsilon)\}\Big] \geq (\zeta_r - \tau)\alpha
\end{align*}
and
\begin{align*}
\mathcal I_{2,r}
& = \Ep\Big[\Ep\Big[1\{X+Z\leq r + \epsilon\} - 1\{Y+Z\leq r+\epsilon\}\mid Z\Big]1\{Z\notin R(0,\epsilon)\}\Big]\geq - \zeta(1-\alpha).
\end{align*}
Therefore, taking the supremum over $r\in\mathbb R^d$ in \eqref{eq: i1-i2} and recalling \eqref{eq: case 1 ini smo}, we have
$$
\gamma \geq \zeta(2\alpha - 1) - \tau\alpha.
$$
Rearranging the terms in this inequality gives the asserted claim under \eqref{eq: case 1 ini smo}, and since the case
$$
\zeta = \sup_{r\in\mathbb R^d}\Big( \Pr(Y\leq r) - \Pr(X\leq r) \Big)
$$
is similar, the proof is complete.
\end{proof}

\begin{lemma}[Nazarov's inequality]\label{lem: anticoncentration}
Let $Z = (Z_1,\dots,Z_d)^T$ be a centered Gaussian random vector in $\mathbb R^d$ such that $\Ep Z_j^2\geq 1$ for all $j=1,\dots,d$ with $d\geq 3$. Then for any $z\in\mathbb R^d$ and any $\varepsilon>0$,
$$
\Pr(Z\leq z+\varepsilon)-\Pr(Z\leq z)\leq C\varepsilon\sqrt{\log d},
$$
where $C>0$ is a universal constant.
\end{lemma}
\begin{proof}
See Lemma A.1 in \cite{CCK17}.
\end{proof}

\begin{lemma}\label{lem: maximal}
Let $X_1,\dots,X_n$ be independent centered random vectors in $\mathbb R^d$ with $d\geq 2$. Define the following quantities: $Z:=\max_{1\leq j\leq d}|\sum_{i=1}^n X_{ij}|$, $M:=\max_{1\leq i\leq n}\max_{1\leq j\leq d}|X_{ij}|$, and $\sigma^2:=\max_{1\leq j\leq d}\sum_{i=1}^n\Ep[X_{ij}^2]$. Then
$$
\Ep[Z]\leq C\left( \sigma\sqrt{\log d} + \sqrt{\Ep[M^2]}\log d \right),
$$
where $C>0$ is a universal constant.
\end{lemma}
\begin{proof}
See Lemma 8 in \cite{CCK15}
\end{proof}

\begin{lemma}
\label{lem: fuk-nagaev}
Assume the setting of Lemma \ref{lem: maximal}. 
(i) For every $\eta > 0, \beta \in (0,1]$ and $t>0$,
\[
\Pr \{  Z \geq (1+\eta) \Ep [ Z ] + t \} \leq \exp \{ -t^{2}/(3\sigma^{2}) \} +3\exp \{ -( t/(C\| M \|_{\psi_{\beta}}))^{\beta} \},
\]
where $C >0$ is a constant depending only on $\eta, \beta$. (ii) For every $\eta >0, s \geq 1$ and $t>0$,
\[
\Pr \{ Z \geq (1+\eta) \Ep[Z] + t \} \leq \exp\{ -t^2/(3\sigma^2)\}+C' \Ep [ M^{s}]/t^s, 
\]
where $C' >0 $ is a constant depending only on $\eta$ and $s$. 
\end{lemma}
\begin{proof}
See Theorem 4 in \cite{A08} for case (i) and Theorem 2 in \cite{A10} for case (ii).
\end{proof}

\begin{lemma}
\label{lem: maximal ineq nonnegative}
Let $X_1,\dots,X_n$ be independent random vectors in $\R^d$ with $d\geq 2$ such that $X_{ij} \geq 0$ for all $i=1,\dots,n$ and $j=1,\dots,d$. Define $Z := \max_{1 \leq j \leq d} \sum_{i=1}^{n} X_{ij}$ and $M := \max_{1 \leq i \leq n} \max_{1 \leq j \leq d}  X_{ij} $. Then
\[
\Ep [Z] \leq C \left( \max_{1\leq j\leq d}\Ep \left[ \sum_{i=1}^n X_{ij} \right]+\Ep [M] \log d \right),
\]
where $C>0$ is a universal constant. 
\end{lemma}
\begin{proof}
See Lemma 9 in \cite{CCK15}.
\end{proof}

\begin{lemma}
\label{lem: deviation ineq nonnegative}
Assume the setting of Lemma \ref{lem: maximal ineq nonnegative}.
(i) For every $\eta > 0, \beta \in (0,1]$ and $t > 0$, 
\[
\Pr \{ Z \geq (1+\eta) \Ep[Z] + t \} \leq 3 \exp \{ -( t/(C\| M \|_{\psi_{\beta}}))^{\beta} \},
\]
where $C > 0$ is a constant depending only on $\eta, \beta$. 
(ii) For every $\eta > 0, s \geq 1$ and $t > 0$, 
\[
\Pr \{ Z  \geq (1+\eta) \Ep[Z] + t \} \leq C'  \Ep[M^{s}]/t^{s},
\]
where $C' > 0$ is a constant depending only on $\eta,s$. 
\end{lemma}
\begin{proof}
See Lemma E.4 in \cite{CCK17}.
\end{proof}

% \section{reference}

\end{document}